\documentclass[11pt,a4paper]{article}

\usepackage{graphicx}
\usepackage{color}
\usepackage{indentfirst}
\usepackage{amsmath,amssymb,bm}
\usepackage{cases}
\usepackage{latexsym,bm,amsmath,amssymb,amsthm}

\begin{document}

\title{\bfseries Multiply Warped Products with a Quarter-symmetric Connection}
\author{Quan Qu, Yong Wang \thanks{Corresponding author.
        \newline \mbox{} \quad     E-mail addresses: quq453@nenu.edu.cn (Q. Qu), wangy581@nenu.edu.cn (Y. Wang).
        \newline  \mbox{} \quad   School of Mathematics and Statistics, Northeast Normal University, Changchun Jilin, 130024, China}}\date{}
\maketitle

\begin{abstract}
In this paper, we study the Einstein warped products and multiply warped products with a quarter-symmetric connection. We also study warped products and multiply warped products with a quarter-symmetric connection with constant scalar curvature. Then apply our results to generalized Robertson-Walker space-times with a quarter-symmetric connection and generalized Kasner space-times with a quarter-symmetric connection.
\end{abstract}

\section{Introduction}
The (singly) warped product $B \times _{b}F$ of two pseudo-Riemannian manifolds $(B,g_{B})$ and $(F,g_{F})$ with a smooth function $b : B \to(0,\infty)$ is a product manifold of form $B \times F$ with the metric tensor $g = g_{B}\oplus b^2g_{F}$. Here, $(B,g_{B})$ is called the base manifold and $(F,g_{F})$ is called as the fiber manifold and $b$ is called as the warping function. The concept of warped products was first introduced by Bishop and O'Neill [1] to construct examples of Riemannian manifolds with negative curvature. In [2], F. Dobarro and E. Dozo had studied the problem of showing when a Riemannian metric of constant scalar curvature can be produced on a product manifolds by a warped product construction from the viewpoint of partial differential equations and variational methods. In [3], Ehrlich, Jung and Kim got explicit solutions to warping function to have a constant scalar curvature for generalized Robertson-Walker space-times. In [4], explicit solutions were also obtained for the warping function to make the space-time as Einstein when the fiber is also Einstein.

One can generalize (singly) warped products to multiply warped products. A multiply warped product $(M,g)$ is the product manifold $M=B \times_{b_{1}}F_{1} \times_{b_{2}}F_{2}\cdots \times_{b_{m}}F_{m}$ with the metric $g=g_{B}\oplus b_{1}^2g_{F_{1}}\oplus b_{2}^2g_{F_{2}}\cdots\oplus b_{m}^2g_{F_{m}}$, where for each $i\in\{1,\cdots,m\},b_{i}:B \to(0,\infty)$ is smooth and $(F_{i},g_{F_{i}})$ is a pseudo-Riemannian manifold. In particular, when $B=(c,d),$ the metric $g_{B}=-dt^2$ is negative and $(F_{i},g_{F_{i}})$ is a Riemannian manifold,
we call $M$ as the multiply generalized Robertson-Walker space-time.

Singly warped products have a natural generalization. A twisted product $(M, g)$ is a product manifold of form
$M = B\times_{b}F$, with a smooth function $b:B \times F \to(0,\infty)$, and the metric tensor $g=g_{B}\oplus b^2g_{F}$. In [5], they showed that mixed Ricci-flat twisted products could be expressed as warped products. As a consequence, any Einstein twisted products are warped products. Similar to the definition of multiply warped product, a multiply twisted product $(M, g)$ is a product manifold of form $M=B \times b_{1}F_{1}\times b_{2}F_{2}\cdots \times b_{m}F_{m}$ with the metric $g=g_{B}\oplus b^2_{1}g_{F_{1}}\oplus b^2_{2}g_{F_{2}}\cdots\oplus b^2_{m}g_{F_{m}}$, where for each $i \in \{1,\cdots ,m\}, b_{i}:B\times F_{i} \to (0,\infty)$ is smooth. So in this paper, we define the multiply twisted products as generalizations of twisted products and multiply warped products.

The definition of a semi-symmetric metric connection was given by H. Hayden in [6]. In 1970, K. Yano [7] considered a semi-symmetric metric connection and studied some of its properties. Then in 1975, Golab [8] introduced the idea of a quarter-symmetric linear connection in differentiable manifold which is a generalization of semi-symmetric connection. A linear connection $\nabla$ on an $n$-dimensional Riemannian manifold $(M,g)$ is called a quarter-symmetric connection if its torsion tensor $T$ of the connection $\nabla$ satisfies
$T(X,Y)=\pi(Y)\phi X-\pi(X)\phi Y,$ where $\pi$ is a 1-form and $\phi$ is a (1,1) tensor field. In particular, if $\phi(X)=X,$ then the quarter-symmetric connection reduces to the semi-symmetric connection.

In [9], Dobarro and \"Unal studied Ricci-flat and Einstein-Lorentzian multiply warped products and considered the case of having constant scalar curvature for multiply warped products and applied their results to generalized Kasner space-times. In [10], S. Sular and C. \"Ozg\"ur studied warped product manifolds with a semi-symmetric metric
connection, they computed curvature of semi-symmetric metric connection and considered Einstein warped product manifolds with a semi-symmetric metric connection. In [11], they studied warped product manifolds with a
semi-symmetric non-metric connection. In [12], we considered multiply warped products with a semi-symmetric metric connection, then apply our results to generalized Robertson-Walker spacetimes with a semi-symmetric metric connection
and generalized Kasner spacetimes with a semi-symmetric metric connection. In [13], we studied curvature of multiply warped products with a semi-symmetric non-metric connection. In this paper, we will generalize our result to warped and multiply warped products with a quarter-symmetric connection.

This paper is arranged as follows. In Section 2, we give the definition of a quarter-symmetric connection and its curvature, then give the formula of the Levi-Civita connection and curvature of singly warped and multiply twisted product. In section 3, we first compute curvature of a singly warped product with a quarter-symmetric connection, then study the generalized Robertson-Walker space-times with a quarter-symmetric connection. In section 4, firstly we compute curvature of multiply twisted products with a quarter-symmetric connection, secondly we study the special multiply warped product with a quarter-symmetric connection, finally we consider the generalized Kasner space-times with a quarter-symmetric connection.

\section{Preliminaries}\label{sec2}
Let $ M $ be a Riemannian manifold with Riemannian metric $g$. A linear connection $\overline\nabla$ on a Riemannian manifold $M$ is called a quarter-symmetric connection if the torsion tensor $T$ of the connection $\overline\nabla$
$$T(X, Y ) = \overline\nabla_{X}Y-\overline\nabla_{Y}X-[X,Y] \eqno{(1)}$$
satisfies  $$T(X, Y ) = \pi(Y)\phi X-\pi(X)\phi Y  \eqno{(2)}$$
where $\pi$ is a 1-form associated with the vector field $P$ on $M$ defined by $\pi(X)=g(X,P)$ and $\phi$ is a (1,1) tensor field. $\overline\nabla$ is called a quarter-symmetric metric connection if it satisfies $\overline\nabla g=0$. $\overline\nabla$ is called a quarter-symmetric non-metric connection if it satisfies $\overline\nabla g\neq0$.

If $\nabla$ is the Levi-Civita connection of $M$, in the equation (2.4) in [14], let $\varphi_{1}=\lambda_{1},\;
\varphi_{2}=0,\;U=P,\;f_{1}=0,\;f_{2}=\lambda_{2}-\lambda_{1},\;U_{2}=P,\;\lambda_{1}\neq0,\;\lambda_{2}\neq0$, then we get a linear connection $\overline\nabla$ defined by
$$\overline\nabla_{X}Y=\nabla_{X}Y+\lambda_{1}\pi(Y)X-\lambda_{2}g(X,Y)P. \eqno{(3)}$$
It is easy to see that:\\
¢Ù when $\lambda_{1}=\lambda_{2}=1,\;\overline\nabla$ is a semi-symmetric metric connection;\\
¢Ú when $\lambda_{1}=\lambda_{2}\neq1,\;\overline\nabla$ is a quarter-symmetric metric connection;\\
¢Û when $\lambda_{1}\neq\lambda_{2},\;\overline\nabla$ is a quarter-symmetric non-metric connection.\par
In [12], we considered the case ¢Ù. In this paper, we will consider cases ¢Ú and ¢Û.

Let $R$ and $\overline R$ be the curvature tensors of $\nabla$ and $\overline\nabla$, respectively. By the equation (3.13) in [14], we can get
\begin{eqnarray}
  \overline{R}(X,Y)Z &=& R(X,Y)Z+\lambda_{1}g(Z,\nabla_{X}P)Y-\lambda_{1}g(Z,\nabla_{Y}P)X  \nonumber \\
   &+&\lambda_{2}g(X,Z)\nabla_{Y}P-\lambda_{2}g(Y,Z)\nabla_{X}P   \nonumber    \\
  \setcounter{equation}{4}
   &+&\lambda_{1}\lambda_{2}\pi(P)[g(X,Z)Y-g(Y,Z)X]   \\
   &+&\lambda_{2}^2 [g(Y,Z)\pi(X)-g(X,Z)\pi(Y)]P     \nonumber  \\
   &+&\lambda_{1}^2\pi(Z)[\pi(Y)X-\pi(X)Y]           \nonumber
\end{eqnarray}
for any vector fields $X,Y,Z$ on $M$.
\newtheorem{remark}{Remark}
\begin{remark}
$\mbox{When }\lambda_{1}=\lambda_{2}=1, \mbox{ we can get the equation }(4) \mbox{ in }[12].$
\end{remark}
\subsection{Warped Product}
Let $(B,g_{B})$ and $(F,g_{F})$ be two Riemannian manifolds and $f:B \to(0,\infty)$ be a smooth function. The warped product is the product manifold $B \times F$ with the metric tensor $g = g_{B}\oplus f^2g_{F}$. The function $f$ is called the warping function of the warped product, and the Hessian of $f$ is defined by $H^f(X,Y)=XYf-(\nabla_{X}Y)f$. \par
We need the following two lemmas from [14], for later use:

\newtheorem{Lemma}{Lemma}[section]
\begin{Lemma}
Let $M=B \times_{f}F $ be a warped product, $\nabla,\nabla^B $and $\nabla^F$ denote the Levi-Civita connection on $M,B $ and $F$, respectively. If $X,Y\in\Gamma(TB)$ and $U,W\in\Gamma(TF)$, then:   \\
$(1)\nabla_{X}Y=\nabla^B_{X}Y; $  \\
$(2)\nabla_{X}U=\nabla_{U}X=\frac{Xf}{f}U ; $  \\
$(3)\nabla_{U}W=-\frac{g(U,W)}{f}grad_{B}f+\nabla^F_{U}W.$
\end{Lemma}

\begin{Lemma}
Let $M=B \times_{f}F $ be a warped product with curvature $R$, If $X,Y,Z\in\Gamma(TB)$ and $U,V,W\in\Gamma(TF)$, then:   \\
$(1)R(X,Y)Z=R^B(X,Y)Z ;$ \\
$(2)R(V,X)Y = -\frac{H^f_{B}(X,Y)}{f}V;$  \\
$(3)R(X,Y)V =R(V,W)X=0 ;$    \\
$(4)R(X,V)W=-\frac{g(V,W)}{f}\nabla^B_{X}grad_{B}f ;$   \\
$(5)R(V,W)U=R^F(V,W)U-\frac{|grad_{B}f|^2_{B}}{f^2}[g(W,U)V-g(V,U)W] .$
\end{Lemma}
\subsection{Multiply Twisted Product}
A multiply twisted product $(M,g)$ is a product manifold of form $M=B \times_{b_{1}}F_{1} \times_{b_{2}}F_{2}\cdots \times_{b_{m}}F_{m}$ with the metric $g=g_{B}\oplus b_{1}^2g_{F_{1}}\oplus b_{2}^2g_{F_{2}}\cdots\oplus b_{m}^2g_{F_{m}}$, where for each $i\in\{1,\cdots,m\},b_{i}:B\times F_{i} \to(0,\infty)$ is smooth. Similarly the Hessian of $b_{i}$ is defined by $H^{b_{i}}(X,Y)=XYb_{i}-(\nabla_{X}Y)b_{i}$. \par
We need the following two lemmas from [12], for later use:
\begin{Lemma}
Let $M=B \times_{b_{1}}F_{1} \times_{b_{2}}F_{2}\cdots \times_{b_{m}}F_{m}$ be a multiply twisted product and let $X,Y\in\Gamma(TB)$ and $U\in\Gamma(TF_{i}),W\in\Gamma(TF_{j})$, then:   \\
$(1)\nabla_{X}Y=\nabla^B_{X}Y ;$  \\
$(2)\nabla_{X}U=\nabla_{U}X=\frac{Xb_{i}}{b_{i}}U  ;$  \\
$(3)\nabla_{U}W=0$  if $i\neq j ;$  \\
$(4)\nabla_{U}W=U(lnb_{i})W+W(lnb_{i})U-\frac{g_{F_{i}}(U,W)}{b_{i}}grad_{F_{i}}b_{i}-b_{i}g_{F_{i}}(U,W)grad_{B}b_{i}
 +\nabla^{F_{i}}_{U}W   \mbox{if $i=j$}.$
\end{Lemma}
\begin{Lemma}
Let $M=B \times_{b_{1}}F_{1} \times_{b_{2}}F_{2}\cdots \times_{b_{m}}F_{m}$ be a multiply twisted product and let $X,Y,Z\in\Gamma(TB)$ and $V\in\Gamma(TF_{i}),W\in\Gamma(TF_{j}),U\in\Gamma(TF_{k})$, then:   \\
$(1)R(X,Y)Z=R^B(X,Y)Z; $ \\
$(2)R(V,X)Y = -\frac{H^{b_{i}}_{B}(X,Y)}{b_{i}}V; $  \\
$(3)R(X,V)W =R(V,W)X=R(V,X)W=0 \quad \mbox{if i $\neq$ j};  $    \\
$(4)R(X,Y)V=0;$   \\
$(5)R(V,W)X=VX(lnb_{i})W-WX(lnb_{i})V$ \quad if $i=j;$  \\
$(6)R(V,W)U=0   \quad\mbox{if $i=j\neq k$  or $i\neq j\neq k$};$\\
$(7)R(U,V)W=-g(V,W)\frac{g_{B}(grad_{B}b_{i},grad_{B}b_{k})}{b_{i}b_{k}}U \mbox{\quad if $i\neq j\neq k$};$   \\
$(8)R(X,V)W=-\frac{g(V,W)}{b_{i}}\nabla^B_{X}(grad_{B}b_{i})+[WX(lnb_{i})]V-g_{F_{i}}(W,V)grad_{F_{i}}X(lnb_{i}) $
\mbox{if $i=j$};\\
$(9)R(V,W)U=g(V,U)grad_{B}(W(lnb_{i}))-g(W,U)grad_{B}(V(lnb_{i}))+R^{F_{i}}(V,W)U \\
\mbox{   }\;\; -\frac{|grad_{B}b_{i}|^2_{B}}{b_{i}^2}[g(W,U)V-g(V,U)W] \quad \mbox{ if $i=j= k$}. $
\end{Lemma}
\begin{remark}
$\mbox{It is easy to see that Lemmas }2.1\mbox{ and }2.2\mbox{ are Corollaries of Lemma }2.3\mbox{ and }
\\2.4,\mbox{ respectively.}$
\end{remark}
Finally we define the curvature, Ricci curvature and scalar curvature as follows:  
$$R(X,Y)=\nabla_{X}\nabla_{Y}-\nabla_{Y}\nabla_{X}-\nabla_{[X,Y]},$$
$$Ric(X,Y)=\sum_{k}\varepsilon_{k}\langle R(X,E_{k})Y,E_{k}\rangle,$$
$$S=\sum_{k}\varepsilon_{k}Ric(E_{k},E_{k})$$
where $E_{k}$ is an orthonormal base of $M$ with $\langle E_{k},E_{k}\rangle=\varepsilon_{k},\;\varepsilon_{k}=\pm1.$
\section{Warped Product with a Quarter-symmetric Connection}\label{sec3}
In this section, we firstly compute curvature, Ricci curvature and scalar curvature of singly warped product with a quarter-symmetric connection, then study the generalized Robertson-Walker space-times with a quarter-symmetric connection.
\subsection{Connection and Curvature}
By Lemma 2.1 and the equation (3), we have the following two Propositions:
\newtheorem{Proposition}{Proposition}[section]
\begin{Proposition}
Let $M=B \times_{f}F $ be a warped product. If $X,Y\in\Gamma(TB)$, $U , W\in\Gamma(TF)$ and $P\in\Gamma(TB),$ then:\\
$(1)\overline\nabla_{X}Y=\overline\nabla^B_{X}Y; $  \\
$(2)\overline\nabla_{X}U=\frac{Xf}{f}U;$   \\
$(3)\overline\nabla_{U}X=\Big[\frac{Xf}{f}+\lambda_{1}\pi(X)\Big]U; $   \\
$(4)\overline\nabla_{U}W=-fg_{F}(U,W)grad_{B}f+\nabla^F_{U}W-\lambda_{2}g(U,W)P. $
\end{Proposition}

\begin{Proposition}
Let $M=B \times_{f}F $ be a warped product. If $X,Y\in\Gamma(TB)$, $U , W\in\Gamma(TF)$ and $P\in\Gamma(TF),$ then:\\
$(1)\overline\nabla_{X}Y=\nabla^B_{X}Y-\lambda_{2}g(X,Y)P;  $  \\
$(2)\overline\nabla_{X}U=\frac{Xf}{f}U+\lambda_{1}\pi(U)X;$   \\
$(3)\overline\nabla_{U}X=\frac{Xf}{f}U; $   \\
$(4)\overline\nabla_{U}W=-\frac{g(U,W)}{f}grad_{B}f+\overline\nabla^F_{U}W. $
\end{Proposition}

By Lemmas 2.1¡¢2.2 and the equation (4), we have the following two Propositions:
\begin{Proposition}
Let $M=B \times_{f}F $ be a warped product. If $X,Y,Z\in\Gamma(TB)$, $U,V,W\in\Gamma(TF)$ and $P\in\Gamma(TB),$ then:
\\
$(1)\overline R(X,Y)Z=\overline R^B(X,Y)Z; $  \\
$(2)\overline R(V,X)Y=-\Big[\frac{H^f_{B}(X,Y)}{f}+\lambda_{2}\frac{Pf}{f}g(X,Y)+\lambda_{1}\lambda_{2}\pi(P)g(X,Y)
+\lambda_{1}g(Y,\nabla_{X}P) \\ \mbox{} \quad -\lambda_{1}^2\pi(X)\pi(Y)\Big]V;$   \\
$(3)\overline R(X,Y)V=0; $   \\
$(4)\overline R(V,W)X=0; $   \\
$(5)\overline R(X,V)W=-g(V,W)\Big[\frac{\nabla_{X}^B(grad_{B}f)}{f}+\lambda_{1}\frac{Pf}{f}X+\lambda_{2}\nabla_{X}P
+\lambda_{1}\lambda_{2}\pi(P)X-\lambda_{2}^2\pi(X)P\Big];$  \\
$(6)\overline R(U,V)W=R^F(U,V)W-\Big[\frac{|grad_{B}f|^2_{B}}{f^2}+(\lambda_{1}+\lambda_{2})\frac{Pf}{f}+
\lambda_{1}\lambda_{2}\pi(P)\Big]\Big[g(V,W)U\\ \mbox{} \quad -g(U,W)V\Big]. $
\end{Proposition}

\begin{Proposition}
Let $M=B \times_{f}F $ be a warped product. If $X,Y,Z\in\Gamma(TB)$, $U,V,W\in\Gamma(TF)$ and $P\in\Gamma(TF),$ then: \\
$(1)\overline R(X,Y)Z=R^B(X,Y)Z+\lambda_{2}\Big[g(X,Z)\frac{Yf}{f}-g(Y,Z)\frac{Xf}{f}\Big]P+\lambda_{1}\lambda_{2}
\pi(P)[g(X,Z)Y\\ \mbox{} \quad-g(Y,Z)X]; $  \\
$(2)\overline R(V,X)Y=-\frac{H^f_{B}(X,Y)}{f}V-\lambda_{1}\pi(V)\frac{Yf}{f}X-\lambda_{2}g(X,Y)\nabla_{V}P
-g(X,Y)[\lambda_{1}\lambda_{2}\pi(P)V \\ \mbox{} \quad -\lambda_{2}^2\pi(V)P] ;$   \\
$(3)\overline R(X,Y)V=\lambda_{1}\pi(V)\Big[\frac{Xf}{f}Y-\frac{Yf}{f}X\Big]; $   \\
$(4)\overline R(V,W)X=\lambda_{1}\frac{Xf}{f}[\pi(W)V-\pi(V)W]; $   \\
$(5)\overline R(X,V)W=-g(V,W)\frac{\nabla_{X}^B(grad_{B}f)}{f}+\lambda_{1}\frac{Xf}{f}\pi(W)V-\lambda_{1}g(W,\nabla_{V}P)X-\lambda_{2}
g(V,W)\frac{Xf}{f}P\\ \mbox{}\quad-\lambda_{1}\lambda_{2}g(V,W)\pi(P)X+\lambda_{1}^2\pi(W)\pi(V)X;$  \\
$(6)\overline R(U,V)W=R^F(U,V)W-\frac{|grad_{B}f|^2_{B}}{f^2}[g(V,W)U-g(U,W)V]+\lambda_{1}[g(W,\nabla_{U}P)V
\\ \mbox{}\quad-g(W,\nabla_{V}P)U]+\lambda_{2}[g(U,W)\nabla_{V}P-g(V,W)\nabla_{U}P)]
+\lambda_{1}\lambda_{2}\pi(P)[g(U,W)V\\ \mbox{}\quad-g(V,W)U]+\lambda_{2}^2[g(V,W)\pi(U)-g(U,W)\pi(V)]P
+\lambda_{1}^2\pi(W)[\pi(V)U-\pi(U)V]. $
\end{Proposition}

By Propositions 3.3 and 3.4 and the definition of the Ricci curvature tensor, we have the following two Propositions:
\begin{Proposition}
Let $M=B \times_{f}F $ be a warped product, ${\rm dim}B=n_{1}, {\rm dim}F=n_{2} \mbox{ and } \\ {\rm dim}M=\overline n=n_{1}+n_{2}$. If $X,Y\in\Gamma(TB)$, $V,W\in\Gamma(TF)$ and $P\in\Gamma(TB),$ then: \\
$(1)\overline {Ric}(X,Y)=\overline {Ric}^B(X,Y)+n_{2}\Big[\frac{H^f_{B}(X,Y)}{f}+\lambda_{2}\frac{Pf}{f}g(X,Y)
+\lambda_{1}\lambda_{2}\pi(P)g(X,Y) \\ \mbox{}\quad+\lambda_{1}g(Y,\nabla_{X}P)-\lambda_{1}^2\pi(X)\pi(Y)\Big]; $\\
$(2)\overline {Ric}(X,V)=\overline {Ric}(V,X)=0; $   \\
$(3)\overline {Ric}(V,W)=\overline {Ric}^F(V,W)+\Big\{\frac{\Delta_{B}f}{f}+(n_{2}-1)\frac{|grad_{B}f|^2_{B}}{f^2}
+[(\overline n-1)\lambda_{1}\lambda_{2}-\lambda_{2}^2]\pi(P)\\ \mbox{}\quad +\lambda_{2}div_{B}P+[(\overline n-1)\lambda_{1}+(n_{2}-1)\lambda_{2}]\frac{Pf}{f}\Big\}g(V,W). $\\
where  $div_{B}P=\sum\limits_{k=1}^{n_{1}}\varepsilon_{k}\langle\nabla_{E_{k}}P,E_{k}\rangle,$ and $E_{k},
1\leq k\leq n_{1} $ is an orthonormal base of $B $ with $\varepsilon_{k}=g(E_{k},E_{k}).$
\end{Proposition}

\begin{Proposition}
Let $M=B \times_{f}F $ be a warped product, ${\rm dim}B=n_{1}, {\rm dim}F=n_{2} \mbox{ and } \\ {\rm dim}M=\overline n=n_{1}+n_{2}$. If $X,Y\in\Gamma(TB)$, $V,W\in\Gamma(TF)$ and $P\in\Gamma(TF),$ then: \\
$(1)\overline {Ric}(X,Y)=\overline {Ric}^B(X,Y)+n_{2}\frac{H^f_{B}(X,Y)}{f}+[(\overline n-1)\lambda_{1}\lambda_{2}
-\lambda_{2}^2]\pi(P)g(X,Y)\\ \mbox{}\quad +\lambda_{2}g(X,Y)div_{F}P; $ \\
$(2)\overline {Ric}(X,V)=[(\overline n-1)\lambda_{1}-\lambda_{2}]\pi(V)\frac{Xf}{f}; $   \\
$(3)\overline {Ric}(V,X)=[\lambda_{2}-(\overline n-1)\lambda_{1}]\pi(V)\frac{Xf}{f}; $   \\
$(4)\overline {Ric}(V,W)=Ric^F(V,W)+g(V,W)\Big\{\frac{\Delta_{B}f}{f}+(n_{2}-1)\frac{|grad_{B}f|^2_{B}}{f^2}
+[(\overline n-1)\lambda_{1}\lambda_{2}-\lambda_{2}^2]\pi(P)\Big\}\\ \mbox{}\quad+[(\overline n-1)\lambda_{1}
-\lambda_{2}]g(W,\nabla_{V}P)+[\lambda_{2}^2+(1-\overline n)\lambda_{1}^2]\pi(V)\pi(W)+\lambda_{2}g(V,W)div_{F}P . $
\end{Proposition}

By Proposition 3.5 and the definition of the scalar curvature, we have the following:
\begin{Proposition}
Let $M=B \times_{f}F $ be a warped product, ${\rm dim}B=n_{1}, {\rm dim}F=n_{2} \mbox{ and } \\ {\rm dim}M=\overline n=n_{1}+n_{2}$. If $P\in\Gamma(TB),$ then the scalar curvature $\overline S$ has the following expression: \par
$\overline S=\overline S^B+2n_{2}\frac{\Delta_{B}f}{f}+\frac{S^F}{f^2}+n_{2}(n_{2}-1)\frac{|grad_{B}f|^2_{B}}{f^2}
+n_{2}(\overline n-1)(\lambda_{1}+\lambda_{2})\frac{Pf}{f}\\ \mbox{}\qquad\quad +[n_{2}(\overline n+n_{1}-1)\lambda_{1}\lambda_{2}-n_{2}(\lambda_{1}^2+\lambda_{2}^2)]\pi(P)+n_{2}(\lambda_{1}+\lambda_{2})div_{B}P. $
\end{Proposition}

By Proposition 3.6 and the definition of the scalar curvature, we have the following:
\begin{Proposition}
Let $M=B \times_{f}F $ be a warped product, ${\rm dim}B=n_{1}, {\rm dim}F=n_{2} \mbox{ and } \\ {\rm dim}M=\overline n=n_{1}+n_{2}$. If $P\in\Gamma(TF),$ then the scalar curvature $\overline S$ has the following expression: \par
$\overline S=\overline S^B+2n_{2}\frac{\Delta_{B}f}{f}+\frac{S^F}{f^2}+n_{2}(n_{2}-1)\frac{|grad_{B}f|^2_{B}}{f^2}
+[\overline n(\overline n-1)\lambda_{1}\lambda_{2}+(1-\overline n)(\lambda_{1}^2+\lambda_{2}^2)]\pi(P)
\\ \mbox{}\qquad\quad +(\overline n-1)(\lambda_{1}+\lambda_{2})div_{F}P.  $
\end{Proposition}

\subsection{Generalized Robertson-Walker Space-times with a Quarter-sym-metric Connection}
\newtheorem{Theorem}[Proposition]{Theorem}
\begin{Theorem}
Let $M=I \times_{f}F $ be a warped product, where $I$ is an open interval in $\mathbb{R},$ ${\rm dim}I=1$ and ${\rm dim}F=\overline n-1(\overline n\geq3).$ Then $(M,\overline\nabla)$ is Einstein if and only if $(F,\nabla^F)$ is Einstein for $P=\frac{\partial}{\partial t}$ or $f$ is a constant on $I$ for  $P\in\Gamma(TF),\;\lambda_{2}\neq(\overline n-1)\lambda_{1}.$
\end{Theorem}
\begin{proof}
¢Ù Assume that $P=\frac{\partial}{\partial t},$ Let $f=e^{\frac{q}{2}},$ by Proposition 3.5, we can write
\begin{eqnarray}
\setcounter{equation}{5}
\lefteqn{\overline {Ric}\Big(\frac{\partial}{\partial t},\frac{\partial}{\partial t}\Big)=(1-\overline n)\Big[\frac{1}{4}(q^\prime)^2+\frac{1}{2}q^{\prime\prime}-\frac{1}{2}\lambda_{2}q^\prime+
\lambda_{1}\lambda_{2}-\lambda_{1}^2\Big]g_{I}\Big(\frac{\partial}{\partial t},\frac{\partial}{\partial t}\Big);}\hspace{12cm} \\
\lefteqn{\overline {Ric}\Big(\frac{\partial}{\partial t},V\Big)=0;}\hspace{12cm} \nonumber \\
\setcounter{equation}{6}
\lefteqn{\overline {Ric}(V,W)=Ric^F(V,W)+e^q\Big\{\frac{\overline n-1}{4}(q^\prime)^2
+\frac{1}{2}\Big[(\overline n-1)\lambda_{1}+(\overline n-2)\lambda_{2}\Big]q^\prime }\hspace{12cm}  \\ \lefteqn{+\lambda_{2}^2+(1-\overline n)\lambda_{1}\lambda_{2}+\frac{1}{2}q^{\prime\prime}\Big\}g_{F}(V,W) }
\hspace{9.8cm}\nonumber
\end{eqnarray}
for any $V,W\in\Gamma(TF).$

Since $M$ is Einstein manifold, we have
$$\overline {Ric}\Big(\frac{\partial}{\partial t},\frac{\partial}{\partial t}\Big)=
\alpha g_{I}\Big(\frac{\partial}{\partial t},\frac{\partial}{\partial t}\Big);   \eqno{(7)} $$
$$\overline {Ric}(V,W)=\alpha e^qg_{F}(V,W). \eqno{(8)}$$
From equations $(5)\mbox{ and }(7),$ we get
$$\alpha=(1-\overline n)\Big[\frac{1}{4}(q^\prime)^2+\frac{1}{2}q^{\prime\prime}-\frac{1}{2}\lambda_{2}q^\prime+
\lambda_{1}\lambda_{2}-\lambda_{1}^2\Big].   \eqno{(9)} $$
Similarly, from equations $(6)\mbox{ and }(8),\mbox{ and by the use of } (9)$, we obtain
$$Ric^F(V,W)=\Big[\frac{1-\overline n}{2}(q^\prime)^2-\frac{\overline n}{2}q^{\prime\prime}
+\Big(\frac{1-\overline n}{2}\lambda_{1}+\frac{1}{2}\lambda_{2}\Big)q^\prime
+(\overline n-1)\lambda_{1}^2-\lambda_{2}^2\Big]e^qg_{F}(V,W)    $$
which implies that $(F,\nabla^F)$ is Einstein manifold.
\par ¢Ú Assume that $P\in \Gamma(TF),$ by Proposition 3.6, we have
$$\overline {Ric}\Big(\frac{\partial}{\partial t},V\Big)=\frac{1}{2}q^\prime[(\overline n-1)\lambda_{1}-\lambda_{2}]\pi(V); \eqno{(10)} $$
$$\overline {Ric}\Big(V,\frac{\partial}{\partial t}\Big)=\frac{1}{2}q^\prime
[\lambda_{2}-(\overline n-1)\lambda_{1}]\pi(V)  \eqno{(11)} $$
for any $V\in\Gamma(TF).$
\par Since $M$ is an Einstein manifold, we can write
$$\overline{Ric}\Big(\frac{\partial}{\partial t},V\Big)=\alpha g\Big(\frac{\partial}{\partial t},V\Big)=0
=\alpha g\Big(V,\frac{\partial}{\partial t}\Big)=\overline {Ric}\Big(V,\frac{\partial}{\partial t}\Big) \eqno{(12)}$$where $\frac{\partial}{\partial t}\in\Gamma(TB)$ and $V\in\Gamma(TF).$   \par Since $\lambda_{2}\neq(\overline n-1)\lambda_{1}, \pi(V)\neq0,$ , using equations $(10),(11),(12),$ then we can get $q^\prime=0$, which means $q$ is a constant on $I$, then $f$ is a constant on $I$.
\end{proof}
\begin{Theorem}
Let $M=B\times_{f}I $ be a warped product, where $I$ is an open interval in $\mathbb{R},\;{\rm dim}I=1$ and
${\rm dim}B=\overline n-1\;(\overline n\geq3).$ Then \\
$(1)$If $(M,\overline\nabla)$ is an Einstein manifold, $P\in\Gamma(TB),\;\nabla^BP=0,$ and $f$ is a constant on $B$, then: $$\overline S^B=[(\overline n-1)(\overline n-2)\lambda_{1}\lambda_{2}+\lambda_{1}^2
+(1-\overline n)\lambda_{2}^2]\pi(P).$$
Furthermore, if $(\overline n-1)(\overline n-2)\lambda_{1}\lambda_{2}+\lambda_{1}^2+(1-\overline n)\lambda_{2}^2=0$, then $\overline S^B=0.$  \\
$(2)$If $(M,\overline\nabla)$ is Einstein manifold, $P\in\Gamma(TI),$ and $\lambda_{2}\neq(\overline n-1)\lambda_{1}$
then $f$ is a constant on $B.$  \\
$(3)$If $f$ is a constant on $B,$ $(B,\nabla^B)$ is Einstein, $P\in\Gamma(TI),$ then
$(M,\overline\nabla)$ is Einstein manifold; furthermore, if $\lambda_{2}=(\overline n-1)\lambda_{1},$ then
$\alpha=\alpha_{B},$ where $\alpha$ and $\alpha_{B}$ denote the Einstein constant on $M$ and $B$, respectively.
\end{Theorem}
\begin{proof}
¢Ù Assume that $(M,\overline\nabla)$ is an Einstein manifold, $P\in \Gamma(TB)$, then
$$\overline{Ric}(X,Y)=\frac{\;\overline s\;}{\;\overline n\;}g(X,Y)   \eqno{(13)} $$  for any $X,Y\in \Gamma(TB).$
Consider that $f$ is a constant on $B$ and by Proposition 3.7, we can get
$$\overline{Ric}(X,Y)=\frac{1}{\;\overline n\;}g(X,Y)\{\overline S^B+[(2\overline n-2)\lambda_{1}\lambda_{2}
-(\lambda_{1}^2+\lambda_{2}^2)]\pi(P)\}. $$
Define $\widehat{S}=\sum\limits_{k=1}^{\overline n-1}\frac{1}{\;\overline n\;}\overline{Ric}(E_{k},E_{k}),$
where $E_{k}$ is an orthonormal base of $B,$ then we have
$$\widehat{S}=\frac{\;\overline n-1\;}{\;\overline n\;}\Big\{\overline S^B+[(2\overline n-2)\lambda_{1}\lambda_{2}
-(\lambda_{1}^2+\lambda_{2}^2)]\pi(P)\Big\}. \eqno{(14)}$$
On the other hand, using Proposition 3.5, we can get
$$\overline {Ric}(X,Y)=\overline {Ric}^B(X,Y)+[\lambda_{1}\lambda_{2}\pi(P)g(X,Y)-\lambda_{1}^2\pi(X)\pi(Y)],$$
then we have
$$\widehat{S}=\overline S^B+[(\overline n-1)\lambda_{1}\lambda_{2}-\lambda_{1}^2]\pi(P).  \eqno{(15)}$$
From equations (14) and (15), we can get
$$\overline S^B=[(\overline n-1)(\overline n-2)\lambda_{1}\lambda_{2}+\lambda_{1}^2
+(1-\overline n)\lambda_{2}^2]\pi(P).$$    \par
¢Ú Assume that $(M,\overline\nabla)$ is Einstein manifold, $P\in\Gamma(TI),$ by Proposition 3.6, we obtain
$$\overline {Ric}(X,P)=[(\overline n-1)\lambda_{1}-\lambda_{2}]\pi(P)\frac{Xf}{f};  \eqno{(16)}  $$
$$\overline {Ric}(P,X)=[\lambda_{2}-(\overline n-1)\lambda_{1}]\pi(P)\frac{Xf}{f}.  \eqno{(17)}   $$
Using the similar proof of Theorem 3.9.¢Ú, we can get $Xf=0,$ which means $f$ is a constant on $B.$   \par
¢Û If $f$ is a constant on $B,$ and $(B,\nabla^B)$ is Einstein, and $P\in\Gamma(TI),$
then $$Ric^B(X,Y)=\alpha_{B}g(X,Y).    \eqno{(18)}$$
By Theorem 3.6, we have
$$\overline {Ric}(X,Y)=\overline {Ric}^B(X,Y)+[(\overline n-1)\lambda_{1}\lambda_{2}-\lambda_{2}^2]\pi(P)g(X,Y). \eqno{(19)}  $$
So by equations (18) and (19), we can easily get
$$\overline{Ric}(X,Y)=\{\alpha_{B}+[(\overline n-1)\lambda_{1}\lambda_{2}-\lambda_{2}^2]\pi(P)\}g(X,Y)$$
which means $(M,\overline\nabla)$ is an Einstein manifold.
\par Furthermore, if $\lambda_{2}=(\overline n-1)\lambda_{1},$ then $\overline{Ric}(X,Y)=\alpha_{B}g(X,Y),$
and $\alpha=\alpha_{B}.$
\end{proof}
\par Now we specially study $M=I \times F $ with the metric tensor $-dt^2+f(t)^2g_{F},\;I$ is an open interval in $\mathbb{R}.$
Let begin with the following theorem:
\begin{Theorem}
Let $M=I \times F $ with the metric tensor $-dt^2+f(t)^2g_{F}, P=\frac{\partial}{\partial t}, {\rm dim}F=l.$ Then
$(M,\overline\nabla)$ is Einstein with the Einstein constant $\alpha$ if and only if the following
conditions are satisfied \\
$(1)$ $(F,\nabla^F)$ is Einstein with the Einstein constant $\alpha_{F};$  \\
$(2)$ $l\Big(\lambda_{2}\frac{f^\prime}{f}-\frac{f^{\prime\prime}}{f}+\lambda_{1}^2-\lambda_{1}\lambda_{2}\Big)
=\alpha;$ \\
$(3)$ $\alpha_{F}-ff^{\prime\prime}+(1-l)(f^\prime)^2+[\lambda_{2}^2-l\lambda_{1}\lambda_{2}-\alpha]f^2
+[l\lambda_{1}+(l-1)\lambda_{2}]ff^\prime=0.$
\end{Theorem}
\begin{proof}
By Proposition $3.5,$ we have
\begin{eqnarray*}
\lefteqn{\overline {Ric}\Big(\frac{\partial}{\partial t},\frac{\partial}{\partial t}\Big)=
-l\Big(\lambda_{2}\frac{f^\prime}{f}-\frac{f^{\prime\prime}}{f}+\lambda_{1}^2-\lambda_{1}\lambda_{2}\Big);}
\hspace{12cm}\\
\lefteqn{\overline{Ric}\Big(\frac{\partial}{\partial t},V\Big)=\overline {Ric}\Big(V,\frac{\partial}{\partial t}\Big)=0; }\hspace{12cm}\\
\lefteqn{\overline{Ric}(V,W)=Ric^F(V,W)+g_{F}(V,W)\{-ff^{\prime\prime}-(l-1){f^\prime}^2
+(\lambda_{2}^2-l\lambda_{1}\lambda_{2})f^2  }\hspace{12cm}\\
\lefteqn{+[l\lambda_{1}+(l-1)\lambda_{2}]ff^\prime\}.  }\hspace{10cm}
\end{eqnarray*}
Then by the Einstein condition, we get Theorem $3.11.$
\end{proof}
\begin{remark}
When $\lambda_{1}=\lambda_{2}=1, $ we can get Corollary $21$ in $[12]$.
\end{remark}
\par Considering the dimension of $F,$ we get Corollaries 3.12 and 3.13 of Theorem 3.11:
\newtheorem{Corollary}[Proposition]{Corollary}
\begin{Corollary}
Let $M=I \times F $ with the metric tensor $-dt^2+f(t)^2g_{F}, P=\frac{\partial}{\partial t}, {\rm dim}F=1.$ Then
$(M,\overline\nabla)$ is Einstein with the Einstein constant $\alpha$ if and only if \\
$f^{\prime\prime}=\lambda_{2}f^\prime+(\lambda_{1}^2-\lambda_{1}\lambda_{2})f-\alpha f$
\end{Corollary}
\begin{remark}
$(1)$From Theorem $3.11,$ we can also get: if ${\rm dim}F=1,$ then $\alpha_{F}=0;$  \par
\mbox{}\qquad\quad  $\;$ $(2)$When $\lambda_{1}=\lambda_{2}=1, $ we can get Corollary $23$ in $[12]$.
\end{remark}
\begin{Corollary}
Let $M=I \times F $ with the metric tensor $-dt^2+f(t)^2g_{F}, P=\frac{\partial}{\partial t}, {\rm dim}F=l>1.$ Then
$(M,\overline\nabla)$ is Einstein with the Einstein constant $\alpha$ if and only if the following conditions are satisfied: \\
$(1)$ $(F,\nabla^F)$ is Einstein with the Einstein constant $\alpha^F;$  \\
$(2)$$f^{\prime\prime}=\lambda_{2}f^\prime+(\lambda_{1}^2-\lambda_{1}\lambda_{2}-\frac{\alpha}{l})f;$\\
$(3)$ $\frac{\alpha_{F}}{1-l}+(f^\prime)^2+\Big[\frac{\alpha}{l}+\lambda_{1}\lambda_{2}+\frac{\lambda_{2}^2
-\lambda_{1}^2}{1-l}\Big]f^2+\Big[\frac{l}{1-l}\lambda_{1}+\frac{(l-2)}{1-l}\lambda_{2}\Big]ff^\prime=0.$
\end{Corollary}
\begin{remark}
When $\lambda_{1}=\lambda_{2}=1, $ we can get Corollary $24$ in $[12]$.
\end{remark}  \par
By Corollary 3.12 and elementary methods for ordinary differential equations, we get:
\begin{Theorem}
Let $M=I \times F $ with the metric tensor $-dt^2+f(t)^2g_{F}, P=\frac{\partial}{\partial t}, {\rm dim}F=1.$ Then
$(M,\overline\nabla)$ is Einstein with the Einstein constant $\alpha$ if and only if \\
$(1)\alpha<(\lambda_{1}-\frac{1}{2}\lambda_{2})^2,
f(t)=c_{1}e^{((\lambda_{2}+\sqrt{(2\lambda_{1}-\lambda_{2})^2-4\alpha})/2)t}
+c_{2}e^{((\lambda_{2}-\sqrt{(2\lambda_{1}-\lambda_{2})^2-4\alpha})/2)t},$  \\
$(2)\alpha=(\lambda_{1}-\frac{1}{2}\lambda_{2})^2, f(t)=c_{1}e^{(\lambda_{2}/2)t}+c_{2}te^{(\lambda_{2}/2)t},$  \\
$(3)\alpha>(\lambda_{1}-\frac{1}{2}\lambda_{2})^2, f(t)=c_{1}e^{(\lambda_{2}/2)t}cos((\sqrt{4\alpha-(2\lambda_{1}-\lambda_{2})^2}/2)t)
\\  \mbox{}\qquad \qquad  \qquad   \qquad  \qquad \;
+c_{2}e^{(\lambda_{2}/2)t}sin((\sqrt{4\alpha-(2\lambda_{1}-\lambda_{2})^2}/2)t).$
\end{Theorem}
\begin{remark}
When $\lambda_{1}=\lambda_{2}=1, $ we can get Corollary $25$ in $[12]$.
\end{remark}
As a Corollary of Theorem $3.14,$ we have
\begin{Corollary}
Let $M=I \times F $ with the metric tensor $-dt^2+f(t)^2g_{F}, P=\frac{\partial}{\partial t}, {\rm dim}F=1, $ and $\lambda_{2}=2\lambda_{1},$ then
$(M,\overline\nabla)$ is Einstein with the Einstein constant $\alpha$ if and only if \\
$(1)\alpha<0,f(t)=c_{1}e^{(\lambda_{1}+\sqrt{-\alpha})t}+c_{2}e^{(\lambda_{1}-\sqrt{-\alpha})t},$  \\
$(2)\alpha=0, f(t)=c_{1}e^{\lambda_{1}t}+c_{2}te^{\lambda_{1}t},$  \\
$(3)\alpha>0, f(t)=c_{1}e^{\lambda_{1}t}cos(\sqrt{\alpha}t)+c_{2}e^{\lambda_{1}t}sin(\sqrt{\alpha}t).$
\end{Corollary}
\begin{Theorem}
Let $M=I \times F $ with the metric tensor $-dt^2+f(t)^2g_{F}, P=\frac{\partial}{\partial t}, {\rm dim}F=l>1.$ Then
$(M,\overline\nabla)$ is Einstein with the Einstein constant $\alpha$ if and only if one of the following conditions is satisfied:\\
$(1)\alpha=(\lambda_{1}^2-\lambda_{1}\lambda_{2})l,\;\alpha_{F}=c_{2}^2(l\lambda_{1}^2-\lambda_{2}^2),\;f(t)=c_{2};$\\
$(2)\lambda_{1}=\lambda_{2},\;\alpha=0,\;\alpha_{F}=(l-1)c_{2}^2\lambda_{1}^2,\;f(t)=c_{1}e^{\lambda_{1}t}+c_{2};$\\
$(3)\lambda_{2}^2-2l\lambda_{1}^2+l\lambda_{1}\lambda_{2}\neq0,\;\lambda_{2}\neq l\lambda_{1},\;
\alpha=\frac{(3l^2+l)\lambda_{1}^2\lambda_{2}^2-(l^2+l)\lambda_{1}\lambda_{1}^3-2l^2\lambda_{1}^3\lambda_{2}}
{(l\lambda_{1}-\lambda_{2})^2},\;\alpha_{F}=0,
\\ \mbox{}\quad  f(t)=c_{0}e^{(l\lambda_{1}^2-\lambda_{2}^2)t/(l\lambda_{1}-\lambda_{2})};$\\
$(4)\lambda_{2}^2-2l\lambda_{1}^2+l\lambda_{1}\lambda_{2}=0,\;\alpha=l(\lambda_{1}-\frac{1}{2}\lambda_{2})^2,\;
\alpha_{F}=0,\;f(t)=c_{1}e^{\frac{\lambda_{2}}{2}t}.$
\end{Theorem}
\begin{proof}
Let $\frac{\alpha}{l}=d_{0},\;\frac{\alpha_{F}}{1-l}=\overline {d_{0}},\;
a_{0}=\frac{\lambda_{2}+\sqrt{(2\lambda_{1}-\lambda_{2})^2-4d_{0}}}{2},\;
b_{0}=\frac{\lambda_{2}-\sqrt{(2\lambda_{1}-\lambda_{2})^2-4d_{0}}}{2},$ then
$a_{0}+b_{0}=\lambda_{2},\;a_{0}b_{0}=d_{0}+\lambda_{1}\lambda_{2}-\lambda_{1}^2.$\\
¢Ù $d_{0}<(\lambda_{1}-\frac{1}{2}\lambda_{2})^2,$ then $f(t)=c_{1}e^{a_{0}t}+c_{2}e^{b_{0}t}.$
By Corollary $3.13(3),$ then
\begin{eqnarray}
\lefteqn{\overline {d_{0}}+c_{1}^2\Big(a_{0}^2+a_{0}b_{0}+\lambda_{1}^2+\frac{\lambda_{2}^2-\lambda_{1}^2}{1-l}
+\frac{l}{1-l}\lambda_{1}a_{0}+\frac{l-2}{1-l}\lambda_{2}a_{0}\Big)e^{2a_{0}t} }\hspace{13cm}  \nonumber \\
\setcounter{equation}{20}
\lefteqn{ +c_{2}^2\Big(b_{0}^2+a_{0}b_{0}+\lambda_{1}^2+\frac{\lambda_{2}^2-\lambda_{1}^2}{1-l}
+\frac{l}{1-l}\lambda_{1}b_{0}+\frac{l-2}{1-l}\lambda_{2}b_{0}\Big)e^{2b_{0}t}  }\hspace{12.4cm}  \\
\lefteqn{ +c_{1}c_{2}\Big[4a_{0}b_{0}+\frac{2\lambda_{2}^2-2l\lambda_{1}^2}{1-l}
+\Big(\frac{l}{1-l}\lambda_{1}+\frac{l-2}{1-l}\lambda_{2}\Big)(a_{0}+b_{0})\Big]e^{(a_{0}+b_{0})t}=0. }\hspace{12.4cm} \nonumber
\end{eqnarray} \par
$1)b_{0}=0,$ we have $d_{0}=\lambda_{1}^2-\lambda_{1}\lambda_{2}<(\lambda_{1}-\frac{1}{2}\lambda_{2})^2,\;
a_{0}=\lambda_{2},\;\alpha=(\lambda_{1}^2-\lambda_{1}\lambda_{2})l.$
By the equation $(20),$ we get
$$\overline {d_{0}}+c_{1}^2\frac{l\lambda_{1}(\lambda_{2}-\lambda_{1})}{1-l}e^{2\lambda_{2}t}
+c_{2}^2\frac{\lambda_{2}^2-l\lambda_{1}^2}{1-l}+c_{1}c_{2}\frac{l(\lambda_{2}+2\lambda_{1})(\lambda_{2}-\lambda_{1})}
{1-l}e^{\lambda_{2}t}=0.$$ Since $e^{2\lambda_{2}t}$ and $e^{\lambda_{2}t}$ are linearly independent, we have
\begin{equation*}
\begin{cases}
c_{1}^2\frac{l\lambda_{1}(\lambda_{2}-\lambda_{1})}{1-l}=0\\
\overline {d_{0}}+c_{2}^2\frac{\lambda_{2}^2-l\lambda_{1}^2}{1-l}=0\\
c_{1}c_{2}\frac{l(\lambda_{2}+2\lambda_{1})(\lambda_{2}-\lambda_{1})}{1-l}=0
\end{cases}
\end{equation*}
$1^\prime. c_{1}=0,\;c_{2}\neq0.$ We have
$\underline{\alpha=(\lambda_{1}^2-\lambda_{1}\lambda_{2})l,\;\alpha_{F}=c_{2}^2(l\lambda_{1}^2-\lambda_{2}^2),\;
f(t)=c_{2}}.$\\
$2^\prime. c_{1}\neq0,\;c_{2}=0.$ We have
$\underline{\lambda_{1}=\lambda_{2},\;\alpha=\alpha_{F}=0,\;f(t)=c_{1}e^{\lambda_{1}t}}.$\\
$3^\prime. c_{1}\neq0,\;c_{2}\neq0.$ We have
$\underline{\lambda_{1}=\lambda_{2},\;\alpha=0,\;\alpha_{F}=(l-1)c_{2}^2\lambda_{1}^2,\;f(t)=c_{1}e^{\lambda_{1}t}+c_{2}}.$
\\It is easy to see that the conclusion of $2^\prime$ is a special case of $3^\prime.$  \par
$2)b_{0}\neq0,$ $1^\prime.\; c_{1}=0,\;c_{2}\neq0.$ Since $e^{2a_{0}t},e^{2b_{0}t}$ and $e^{(a_{0}+b_{0})t}$ are linearly independent, we have $\overline {d_{0}}=0,\;b_{0}^2+a_{0}b_{0}+\lambda_{1}^2
+\frac{\lambda_{2}^2-\lambda_{1}^2}{1-l}+\frac{l}{1-l}\lambda_{1}b_{0}+\frac{l-2}{1-l}\lambda_{2}b_{0}=0,$
then $\alpha_{F}=0,\;\frac{l\lambda_{1}-\lambda_{2}}{1-l}b_{0}=\frac{l\lambda_{1}^2-\lambda_{2}^2}{1-l}.$
If $\lambda_{2}=l\lambda_{1},$ we have $0=l\lambda_{1}^2,$ this is a contradiction with $\lambda_{1}\neq0.$
So $\lambda_{2}\neq l\lambda_{1}$ and $b_{0}=\frac{l\lambda_{1}^2-\lambda_{2}^2}{l\lambda_{1}-\lambda_{2}}.$
Considering that $b_{0}<\frac{\lambda_{2}}{2},$ we get
 $\frac{\lambda_{2}^2+l\lambda_{1}\lambda_{2}-2l\lambda_{1}^2}{l\lambda_{1}-\lambda_{2}}>0.$
$a_{0}=\frac{l\lambda_{1}(\lambda_{2}-\lambda_{1})}{l\lambda_{1}-\lambda_{2}},
d_{0}=\frac{(3l+1)\lambda_{1}^2\lambda_{2}^2-(l+1)\lambda_{1}\lambda_{2}^3-2l\lambda_{1}^3\lambda_{2}}
{(l\lambda_{1}-\lambda_{2})^2},$  $\alpha=\frac{(3l^2+l)\lambda_{1}^2\lambda_{2}^2-(l^2+l)\lambda_{1}\lambda_{2}^3-2l^2\lambda_{1}^3\lambda_{2}}
{(l\lambda_{1}-\lambda_{2})^2}$ and $d_{0}$ satisfies $d_{0}<(\lambda_{1}-\frac{1}{2}\lambda_{2})^2.$
So in this case, we obtain \\ $\underline{\frac{\lambda_{2}^2+l\lambda_{1}\lambda_{2}-2l\lambda_{1}^2}{l\lambda_{1}-\lambda_{2}}>0,
\;\alpha=\frac{(3l^2+l)\lambda_{1}^2\lambda_{2}^2-(l^2+l)\lambda_{1}\lambda_{2}^3-2l^2\lambda_{1}^3\lambda_{2}}
{(l\lambda_{1}-\lambda_{2})^2},\;\alpha_{F}=0,\;
f(t)=c_{2}e^{\frac{l\lambda_{1}^2-\lambda_{2}^2}{l\lambda_{1}-\lambda_{2}}t}.}$\\
$2^\prime.\; c_{1}\neq0,\;c_{2}=0.$ Using the same method we can get
$a_{0}=\frac{l\lambda_{1}^2-\lambda_{2}^2}{l\lambda_{1}-\lambda_{2}},\;
b_{0}=\frac{l\lambda_{1}(\lambda_{2}-\lambda_{1})}{l\lambda_{1}-\lambda_{2}}$ and
$\underline{\frac{\lambda_{2}^2+l\lambda_{1}\lambda_{2}-2l\lambda_{1}^2}{l\lambda_{1}-\lambda_{2}}<0,
\;\alpha=\frac{(3l^2+l)\lambda_{1}^2\lambda_{2}^2-(l^2+l)\lambda_{1}\lambda_{2}^3-2l^2\lambda_{1}^3\lambda_{2}}
{(l\lambda_{1}-\lambda_{2})^2},\;\alpha_{F}=0,\;
f(t)=c_{1}e^{\frac{l\lambda_{1}^2-\lambda_{2}^2}{l\lambda_{1}-\lambda_{2}}t}.}$\\
So by $1^\prime$ and $2^\prime,$ we get Theorem $3.16(3).$ \\
$3^\prime.\; c_{1}\neq0,\;c_{2}\neq0.$  Since $e^{2a_{0}t},e^{2b_{0}t}$ and $e^{(a_{0}+b_{0})t}$ are linearly independent, we have
$$a_{0}^2+a_{0}b_{0}+\lambda_{1}^2+\frac{\lambda_{2}^2-\lambda_{1}^2}{1-l}
+\frac{l}{1-l}\lambda_{1}a_{0}+\frac{l-2}{1-l}\lambda_{2}a_{0}=0; \eqno{(21a)}  $$
$$b_{0}^2+a_{0}b_{0}+\lambda_{1}^2+\frac{\lambda_{2}^2-\lambda_{1}^2}{1-l}
+\frac{l}{1-l}\lambda_{1}b_{0}+\frac{l-2}{1-l}\lambda_{2}b_{0}=0.   \eqno{(21b)}$$
$(21a)-(21b)$ we get $(a_{0}-b_{0})\lambda_{2}+\frac{l\lambda_{1}}{1-l}(a_{0}-b_{0})+\frac{l-2}{1-l}\lambda_{2}(a_{0}-b_{0})=0, $
since $ a_{0}\neq b_{0}$ we have $\lambda_{2}+\frac{l\lambda_{1}}{1-l}+\frac{l-2}{1-l}\lambda_{2}=0,$ then $\lambda_{2}=l\lambda_{1},$ using $(21a)$ again we get $0=l\lambda_{1}^2,$ this is a contradiction with $\lambda_{1}\neq0.$ So in case $3^\prime$ we have no solution.\\
¢Ú$d_{0}=(\lambda_{1}-\frac{1}{2}\lambda_{2})^2,$ then $f(t)=c_{1}e^{\frac{\lambda_{2}}{2}t}+c_{2}te^{\frac{\lambda_{2}}{2}t}.$ By Corollary $3.13(3),$ we get
\begin{eqnarray*}
\lefteqn{\overline {d_{0}}+\Big(\frac{\lambda_{2}}{2}c_{1}+c_{2}+\frac{\lambda_{2}}{2}c_{2}t\Big)^2e^{\lambda_{2}t} +\frac{(5-l)\lambda_{2}^2-4l\lambda_{1}^2}{4(1-l)}(c_{1}+c_{2}t)^2e^{\lambda_{2}t}  }\hspace{11cm}  \\
\lefteqn{+\Big(\frac{l}{1-l}\lambda_{1}+\frac{l-2}{1-l}\lambda_{2}\Big)(c_{1}+c_{2}t)
\Big(\frac{\lambda_{2}}{2}c_{1}+c_{2}+\frac{\lambda_{2}}{2}c_{2}t\Big)e^{\lambda_{2}t}=0  }\hspace{10.5cm}
\end{eqnarray*}
$1^\prime.\; c_{1}=0,\;c_{2}\neq0.$ The coefficient of $e^{\lambda_{2}t}$ is $c_{2}^2=0,$ this is a contradiction with $c_{2}\neq0.$\\
$2^\prime.\; c_{1}\neq0,\;c_{2}=0.$ Then $\overline {d_{0}}=0$ and from the coefficient of $e^{\lambda_{2}t}$ we can get $\underline{\lambda_{2}^2+l\lambda_{1}\lambda_{2}}\\ \underline{-2l\lambda_{1}^2=0,}$ so $\underline{\alpha=d_{0}l=l(\lambda_{1}-\frac{1}{2}\lambda_{2})^2,\;\alpha_{F}=0,\;f(t)
=c_{1}e^{\frac{\lambda_{2}}{2}t}.}$
\\
$3^\prime.\; c_{1}\neq0,\;c_{2}\neq0.$ From the coefficient of $t^2e^{\lambda_{2}t}$ we get
$$\lambda_{2}^2+l\lambda_{1}\lambda_{2}-2l\lambda_{1}^2=0,     \eqno{(22)}$$
by the equation $(22),$ the coefficient of $te^{\lambda_{2}t}$ becomes
$$\frac{(l-3)\lambda_{2}+2l\lambda_{1}}{2(1-l)}c_{2}-\frac{\lambda_{2}^2}{4}c_{1}=0,   \eqno{(23)}$$
the coefficient of $e^{\lambda_{2}t}$ is $\frac{l\lambda_{1}-\lambda_{2}}{1-l}c_{1}c_{2}+c_{2}^2=0,$
then $c_{2}=\frac{\lambda_{2}-l\lambda_{1}}{1-l}c_{1},$ so by equation $(23)$ we get
$(2l-7-l^2)\lambda_{2}=(4l^2-10l)\lambda_{1},$ considering that $2l-7-l^2\neq0, 4l^2-10l\neq0,$ we obtain
$\lambda_{2}=\frac{(4l^2-10l)}{(2l-7-l^2)}\lambda_{1},$ using the equation $ (22)$ again, we have
$(l-1)^2(3l^2-15l+49)=0,$ but $l-1\neq0,3l^2-15l+49\neq0,$ so we have no solution in this case.\\
¢Û$d_{0}>(\lambda_{1}-\frac{1}{2}\lambda_{2})^2,$ let $h_{0}=\frac{1}{2}\sqrt{4d_{0}-(2\lambda_{1}-\lambda_{2})^2},$
then $f(t)=e^{\frac{\lambda_{2}}{2}t}(c_{1}cos(h_{0}t)+c_{2}sin(h_{0}t)).$  By Corollary $3.13(3),$ we get
\begin{eqnarray*}
\lefteqn{\overline {d_{0}}+\Big[\Big(\frac{\lambda_{2}}{2}c_{1}+c_{2}h_{0}\Big)cos(h_{0}t)
+\Big(\frac{\lambda_{2}}{2}c_{2}-c_{1}h_{0}\Big)sin(h_{0}t)\Big]^2e^{\lambda_{2}t} }\hspace{15cm}\\ \lefteqn{+\Big(d_{0}+\lambda_{1}\lambda_{2}+\frac{\lambda_{2}^2-\lambda_{1}^2}{1-l}\Big)
\Big(c_{1}cos(h_{0}t)+c_{2}sin(h_{0}t)\Big)^2e^{\lambda_{2}t}+\Big(\frac{l}{1-l}\lambda_{1}
+\frac{l-2}{1-l}\lambda_{2}\Big)  }\hspace{15cm}\\
\lefteqn{\Big(c_{1}cos(h_{0}t)+c_{2}sin(h_{0}t)\Big)\Big[\Big(\frac{\lambda_{2}}{2}c_{1}+c_{2}h_{0}\Big)cos(h_{0}t)
+\Big(\frac{\lambda_{2}}{2}c_{2}-c_{1}h_{0}\Big)sin(h_{0}t)\Big]e^{\lambda_{2}t}=0.  }\hspace{15cm}
\end{eqnarray*}
The coefficient of $cos^2(h_{0}t)e^{\lambda_{2}t}$ is
$$\Big[d_{0}+\frac{(l+1)\lambda_{2}^2+(4-2l)\lambda_{1}\lambda_{2}-4\lambda_{1}^2}{4(1-l)}\Big]c_{1}^2
+c_{2}^2h_{0}^2-\frac{l\lambda_{1}-\lambda_{2}}{1-l}c_{1}c_{2}h_{0}=0.   \eqno{(24)}$$
Similarly the coefficient of $sin^2(h_{0}t)e^{\lambda_{2}t}$ is
$$\Big[d_{0}+\frac{(l+1)\lambda_{2}^2+(4-2l)\lambda_{1}\lambda_{2}-4\lambda_{1}^2}{4(1-l)}\Big]c_{2}^2
+c_{1}^2h_{0}^2-\frac{l\lambda_{1}-\lambda_{2}}{1-l}c_{1}c_{2}h_{0}=0.   \eqno{(25)}$$
If $c_{1}=0,\;c_{2}\neq0,$ by the equation $(24)$ we get $c_{2}^2h_{0}^2=0,$ then $h_{0}=0,$
this is a contradiction with $h_{0}\neq0;$ \\
If $c_{1}\neq0,\;c_{2}=0,$ by the equation $(25)$ we get $c_{1}^2h_{0}^2=0,$ then $h_{0}=0,$
this is a contradiction with $h_{0}\neq0.$
So $c_{1}\neq0,\;c_{2}\neq0,$ $(24)+(25)$ we get
$$d_{0}=\frac{[(2l-4)\lambda_{1}-l\lambda_{2}](\lambda_{2}-\lambda_{1})}{4(1-l)}.   \eqno{(26)}$$
Considering $d_{0}>(\lambda_{1}-\frac{1}{2}\lambda_{2})^2,$ we have
$$\lambda_{2}^2-2l\lambda_{1}^2+l\lambda_{1}\lambda_{2}>0.  \eqno{(27)}$$
The coefficient of $sin(h_{0}t)cos(h_{0}t)e^{\lambda_{2}t}$ is
$$\frac{\lambda_{2}^2-2l\lambda_{1}^2+l\lambda_{1}\lambda_{2}}{1-l}c_{1}c_{2}
+\frac{l\lambda_{1}-\lambda_{2}}{1-l}h_{0}(c_{2}^2-c_{1}^2)=0.    \eqno{(28)}$$
Using $(24)-(25),$ we get
$$\frac{\lambda_{2}^2-2l\lambda_{1}^2+l\lambda_{1}\lambda_{2}}{2(1-l)}(c_{1}^2-c_{2}^2)
+\frac{2(l\lambda_{1}-\lambda_{2})}{1-l}h_{0}c_{1}c_{2}.       \eqno{(29)}$$
Let $a=\frac{\lambda_{2}^2-2l\lambda_{1}^2+l\lambda_{1}\lambda_{2}}{2(1-l)}<0,\;
b=\frac{l\lambda_{1}-\lambda_{2}}{1-l}h_{0},\;x=c_{1}^2-c_{2}^2,\;y=c_{1}c_{2}\neq0,$ from equations $(28)$ and $(29)$ we obtain
$$ax+2by=0,         \eqno{(30a)} $$
$$2ay-bx=0.         \eqno{(30b)} $$
Through $(30a)\times b+(30b)\times a,$ then  $a^2+b^2=0,$ which means $\Big[\frac{\lambda_{2}^2-2l\lambda_{1}^2+l\lambda_{1}\lambda_{2}}{2(1-l)}\Big]^2+
\Big[\frac{l\lambda_{1}-\lambda_{2}}{1-l}h_{0}\Big]^2=0,$ then we get
$$d_{0}=\frac{(3l+1)\lambda_{1}^2\lambda_{2}^2-(l+1)\lambda_{1}\lambda_{2}^3-2l\lambda_{1}^3\lambda_{2}}
{(l\lambda_{1}-\lambda_{2})^2}.   \eqno{(31)} $$
Using equations $(26)$ and $(31)$, we have
$(\lambda_{1}-\lambda_{2})[(2-l)\lambda_{1}-\lambda_{2}](\lambda_{2}^2+l\lambda_{1}\lambda_{2}-2l\lambda_{2}^2)=0.$
By inequality $(27),$ we get $\lambda_{1}=\lambda_{2}$ or $\lambda_{2}=(2-l)\lambda_{1}.$   \par
$1)\lambda_{1}=\lambda_{2},$ by the equation $(26)$ we get $d_{0}=0,$ this is a contradiction with
$d_{0}>(\lambda_{1}-\frac{1}{2}\lambda_{2})^2=\frac{1}{4}\lambda_{1}^2>0.$  \par
$2)\lambda_{2}=(2-l)\lambda_{1},$  by the equation $(26)$ we get
$d_{0}=\frac{l^2-4}{4}\lambda_{1}^2<\frac{l^2}{4}\lambda_{1}^2.$
On the other hand, $d_{0}>(\lambda_{1}-\frac{1}{2}\lambda_{2})^2=\frac{l^2}{4}\lambda_{1}^2,$ this is a contradiction. \par In a word, we have no solution in the case of $d_{0}>(\lambda_{1}-\frac{1}{2}\lambda_{2})^2.$
\end{proof}
\begin{remark}
When $\lambda_{1}=\lambda_{2}=1, $ we get Theorem $26$ in $[12]$.
\end{remark}
\begin{Proposition}
Let $M=I \times F $ with the metric tensor $-dt^2+f(t)^2g_{F}, P=\frac{\partial}{\partial t}.$ Then
$(M,\overline\nabla)$ has constant scalar curvature $\overline S$ if and only if $(F,\nabla^F)$ has constant scalar curvature $S^F$ and
$$\overline S=\frac{S^F}{f^2}-2l\frac{f^{\prime\prime}}{f}-l(l-1)\frac{(f^\prime)^2}{f^2}+l^2(\lambda_{1}
+\lambda_{2})\frac{f^\prime}{f}+l[\lambda_{1}^2+\lambda_{2}^2-(l+1)\lambda_{1}\lambda_{2}].\eqno{(32)}$$
\end{Proposition}
\begin{proof}
Considering that $M=I \times F $ with the metric tensor $-dt^2+f(t)^2g_{F}\mbox{ and } P=\frac{\partial}{\partial t},$ then by Proposition $3.7$ we get the equation $(32).$ With the fact that $S^F$ is function defined on $F,$ and $f$ is function defined on $I,$ then using variables separation we complete the proof of this Proposition.
\end{proof}
\begin{Proposition}
Let $M=I \times F $ with the metric tensor $-dt^2+f(t)^2g_{F},\;P\in\Gamma(TF)$. If $(M,\overline\nabla)$ has constant scalar curvature $\overline S,$ then\\
$(1)\mbox{If } \lambda_{1}+\lambda_{2}\neq0 \mbox{ and }\lambda_{1}^2+\lambda_{2}^2-\overline n\lambda_{1}\lambda_{2}=0,\mbox{ and }div_{F}P $ is a constant, then $S^F$ is a \\ \mbox{}\quad constant;\\
$(2)\mbox{If } \lambda_{1}=-\lambda_{2}\neq0\mbox{ and }  g_{F}(P,P)$  is a constant, then $S^F$ is a constant;\\
$(3)\mbox{If } \lambda_{1}+\lambda_{2}\neq0 \mbox{ and }\lambda_{1}^2+\lambda_{2}^2-\overline n\lambda_{1}\lambda_{2}\neq0,\mbox{ and }div_{F}P,\;g_{F}(P,P) $ are constants, then\\ \mbox{}\quad $S^F$ is a constant.
\end{Proposition}
\begin{proof}
Considering that $M=I \times F $ with the metric tensor $-dt^2+f(t)^2g_{F}, P\in\Gamma(TF),$ then by Proposition $3.8,$
we can get
$$\overline S=\frac{S^F}{f^2}-2l\frac{f^{\prime\prime}}{f}-l(l-1)\frac{(f^\prime)^2}{f^2}
+[l\overline n\lambda_{1}\lambda_{2}-l(\lambda_{1}^2+\lambda_{2}^2)]f^2g_{F}(P,P)
+l(\lambda_{1}+\lambda_{2})div_{F}P.    \eqno{(33)} $$ Then by variables separation, we have:\\
$(1)\mbox{If } \lambda_{1}+\lambda_{2}\neq0 \mbox{ and }\lambda_{1}^2+\lambda_{2}^2-\overline n\lambda_{1}\lambda_{2}=0,\mbox{ and }div_{F}P $ is a constant, we can obtain
$\overline S=\frac{S^F}{f^2}-2l\frac{f^{\prime\prime}}{f}-l(l-1)\frac{(f^\prime)^2}{f^2}
+l(\lambda_{1}+\lambda_{2})div_{F}P. $ Then $S^F$ is a constant;\\
$(2)\mbox{If } \lambda_{1}=-\lambda_{2}\neq0\mbox{ and }  g_{F}(P,P)$  is a constant, which means
$\lambda_{1}+\lambda_{2}=0 \mbox{ and }\lambda_{1}^2+\lambda_{2}^2-\overline n\lambda_{1}\lambda_{2}\neq0,\mbox{ and }g_{F}(P,P) $ is a constant, we can obtain
$\overline S=\frac{S^F}{f^2}-2l\frac{f^{\prime\prime}}{f}-l(l-1)\frac{(f^\prime)^2}{f^2}
+[l\overline n\lambda_{1}\lambda_{2}-l(\lambda_{1}^2+\lambda_{2}^2)]f^2g_{F}(P,P).  $ then $S^F$ is a constant;\\
$(3)$It is obvious.\\
$(4)$If $\lambda_{1}+\lambda_{2}=0 $ and $\lambda_{1}^2+\lambda_{2}^2-\overline n\lambda_{1}\lambda_{2}=0,$
 then we can get $\lambda_{1}=\lambda_{2}=0,$ which is a contradiction.
\end{proof}
In $(32),$ we make the change of variable $ f(t)=\sqrt{v(t)}$ and have the following equation:\\
$$v^{\prime\prime}(t)+\frac{l-3}{4}\frac{v^\prime(t)^2}{v(t)}-\frac{l}{2}(\lambda_{1}+\lambda_{2})v^\prime(t)
+\Big[(l+1)\lambda_{1}\lambda_{2}-\lambda_{1}^2-\lambda_{2}^2+\frac{\overline S}{l}\Big]v(t)-\frac{S^F}{l}=0. \eqno{(34)}  $$
\begin{remark}
When $\lambda_{1}=\lambda_{2}=1, $ equations $(32),(33),(34)$ respectively become $(20),(21),$ $(22)$ in $[12]$,
then by Proposition $3.17,$ we can get Corollary $27$ in $[12],$ and by Proposition $3.18.(3),$ we can get Corollary $28$ in $[12].$
\end{remark}
\begin{Theorem}
Let $M=I \times F $ with the metric tensor $-dt^2+f(t)^2g_{F}, P=\frac{\partial}{\partial t},$ and
${\rm dim}F=l=3.$ Then
$(M,\overline\nabla)$ has constant scalar curvature $\overline S$ if and only if $(F,\nabla^F)$ has constant scalar curvature $S^F$ and\\
$(1)\overline S<\frac{27}{16}(\lambda_{1}+\lambda_{2})^2+3\lambda_{1}^2+3\lambda_{2}^2-12\lambda_{1}\lambda_{2}
\mbox{ and } \overline S\neq3\lambda_{1}^2+3\lambda_{2}^2-12\lambda_{1}\lambda_{2},\\ \mbox{ } \quad
v(t)=c_{1}e^{(\big(\frac{3}{2}(\lambda_{1}+\lambda_{2})+\sqrt{\frac{9}{4}(\lambda_{1}+\lambda_{2})^2
-\frac{4}{3}\overline S+4\lambda_{1}^2+4\lambda_{2}^2-16\lambda_{1}\lambda_{2}}\big)/2)t}\\ \mbox{ } \qquad \quad \!
+c_{2}e^{(\big(\frac{3}{2}(\lambda_{1}+\lambda_{2})-\sqrt{\frac{9}{4}(\lambda_{1}+\lambda_{2})^2
-\frac{4}{3}\overline S+4\lambda_{1}^2+4\lambda_{2}^2-16\lambda_{1}\lambda_{2}}\big)/2)t}
+\frac{S^F}{12\lambda_{1}\lambda_{2}-3\lambda_{1}^2-3\lambda_{2}^2+\overline S}; $\\
$(2)\overline S=\frac{27}{16}(\lambda_{1}+\lambda_{2})^2+3\lambda_{1}^2+3\lambda_{2}^2-12\lambda_{1}\lambda_{2},\\
\mbox{ } \quad v(t)=c_{1}e^{\frac{3}{4}(\lambda_{1}+\lambda_{2})t}+c_{2}te^{\frac{3}{4}(\lambda_{1}+\lambda_{2})t}
+\frac{S^F}{12\lambda_{1}\lambda_{2}-3\lambda_{1}^2-3\lambda_{2}^2+\overline S}; $\\
$(3)\overline S>\frac{27}{16}(\lambda_{1}+\lambda_{2})^2+3\lambda_{1}^2+3\lambda_{2}^2-12\lambda_{1}\lambda_{2},\\
\mbox{ } \quad v(t)=c_{1}e^{\frac{3}{4}(\lambda_{1}+\lambda_{2})t}cos((\Big(\sqrt{\frac{4}{3}\overline S-4\lambda_{1}^2-4\lambda_{2}^2+16\lambda_{1}\lambda_{2}-\frac{9}{4}(\lambda_{1}+\lambda_{2})^2}\Big)/2)t)\\
\mbox{ } \qquad \;\; \; +c_{2}e^{\frac{3}{4}(\lambda_{1}+\lambda_{2})t}sin((\Big(\sqrt{\frac{4}{3}\overline S-4\lambda_{1}^2-4\lambda_{2}^2+16\lambda_{1}\lambda_{2}-\frac{9}{4}(\lambda_{1}+\lambda_{2})^2}\Big)/2)t)\\
\mbox{ } \qquad \;\; \; +\frac{S^F}{12\lambda_{1}\lambda_{2}-3\lambda_{1}^2-3\lambda_{2}^2+\overline S};$\\
$(4)\overline S=3\lambda_{1}^2+3\lambda_{2}^2-12\lambda_{1}\lambda_{2},\mbox{ and } \lambda_{1}+\lambda_{2}\neq0,
\;v(t)=c_{1}-\frac{2S^F}{9(\lambda_{1}+\lambda_{2})}+c_{2}e^{\frac{3}{2}(\lambda_{1}+\lambda_{2})t};$
$(5)\overline S=3\lambda_{1}^2+3\lambda_{2}^2-12\lambda_{1}\lambda_{2},\mbox{ and } \lambda_{1}+\lambda_{2}=0,
\;v(t)=\frac{S^F}{6}t^2+c_{1}t+c_{2}.$
\end{Theorem}
\begin{proof}
If $l=3,$ then we have a simple differential equation as follows:
$$v^{\prime\prime}(t)-\frac{3}{2}(\lambda_{1}+\lambda_{2})v^\prime(t)
+(4\lambda_{1}\lambda_{2}-\lambda_{1}^2-\lambda_{2}^2+\frac{\overline S}{3})v(t)-\frac{S^F}{3}=0. \eqno{(35)}  $$
¢Ù If $\overline S\neq3\lambda_{1}^2+3\lambda_{2}^2-12\lambda_{1}\lambda_{2},$ putting
$h(t)=(4\lambda_{1}\lambda_{2}-\lambda_{1}^2-\lambda_{2}^2+\frac{\overline S}{3})v(t)-\frac{S^F}{3},$ we get
$h^{\prime\prime}(t)-\frac{3}{2}(\lambda_{1}+\lambda_{2})h^\prime(t)
+(4\lambda_{1}\lambda_{2}-\lambda_{1}^2-\lambda_{2}^2+\frac{\overline S}{3})h(t)=0.$ The above solutions $(1)-(3)$
follow directly from elementary methods for ordinary differential equations.\\
¢Ú If $\overline S=3\lambda_{1}^2+3\lambda_{2}^2-12\lambda_{1}\lambda_{2},$ and $\lambda_{1}+\lambda_{2}\neq0,$
then $v^{\prime\prime}(t)-\frac{3}{2}(\lambda_{1}+\lambda_{2})v^\prime(t)-\frac{S^F}{3}=0,$ and we get solution$(4)$.
\\¢Û If $\overline S=3\lambda_{1}^2+3\lambda_{2}^2-12\lambda_{1}\lambda_{2},$ and $\lambda_{1}+\lambda_{2}=0,$
then $v^{\prime\prime}(t)-\frac{S^F}{3}=0,$ and we get solution$(5)$.
\end{proof}
\begin{remark}
When $\lambda_{1}=\lambda_{2}=1, $ we get Theorem $29$ in $[12]$.
\end{remark}
\begin{Theorem}
Let $M=I \times F $ with the metric tensor $-dt^2+f(t)^2g_{F}, P=\frac{\partial}{\partial t},$ and
${\rm dim}F=l\neq3$ and $S^F=0.$ Then $(M,\overline\nabla)$ has constant scalar curvature $\overline S$ if and only if\\
$(1)\overline S<\frac{l^3}{4(l+1)}(\lambda_{1}+\lambda_{2})^2-l[(l+1)\lambda_{1}\lambda_{2}-\lambda_{1}^2-\lambda_{2}^2],\\
\mbox{ } \quad v(t)=\Big[c_{1}e^{(\Big(\frac{1}{2}(\lambda_{1}+\lambda_{2})+\sqrt{\frac{l^2}{4}(\lambda_{1}+\lambda_{2})^2
-(l+1)\big[(l+1)\lambda_{1}\lambda_{2}-\lambda_{1}^2-\lambda_{2}^2+\frac{\overline S}{l}\big]}\Big)/2)t}
\\\mbox{ } \qquad \quad\; +c_{2}e^{(\Big(\frac{1}{2}(\lambda_{1}+\lambda_{2})-\sqrt{\frac{l^2}{4}(\lambda_{1}
+\lambda_{2})^2-(l+1)\big[(l+1)\lambda_{1}\lambda_{2}-\lambda_{1}^2-\lambda_{2}^2+\frac{\overline S}{l}\big]}\Big)/2)t}\Big]^{\frac{4}{l+1}};$\\
$(2)\overline S=\frac{l^3}{4(l+1)}(\lambda_{1}+\lambda_{2})^2-l[(l+1)\lambda_{1}\lambda_{2}-\lambda_{1}^2-\lambda_{2}^2],
v(t)=\Big[c_{1}e^{\frac{l}{4}(\lambda_{1}+\lambda_{2})t}+c_{2}te^{\frac{l}{4}(\lambda_{1}
+\lambda_{2})t}\Big]^{\frac{4}{l+1}};$\\
$(3)\overline S>\frac{l^3}{4(l+1)}(\lambda_{1}+\lambda_{2})^2-l[(l+1)\lambda_{1}\lambda_{2}-\lambda_{1}^2-\lambda_{2}^2],\\
\mbox{ } \quad v(t)=\Big[c_{1}e^{\frac{l}{4}(\lambda_{1}+\lambda_{2})t}cos{(\Big(\sqrt{(l+1)\big[(l+1)\lambda_{1}\lambda_{2}
-\lambda_{1}^2-\lambda_{2}^2+\frac{\overline S}{l}\big]-\frac{l^2}{4}(\lambda_{1}+\lambda_{2})^2}\Big)/2)t}
\\  \mbox{ }\qquad \quad\;\;
+c_{2}e^{\frac{l}{4}(\lambda_{1}+\lambda_{2})t}sin{(\Big(\sqrt{(l+1)\big[(l+1)\lambda_{1}\lambda_{2}-\lambda_{1}^2
-\lambda_{2}^2+\frac{\overline S}{l}\big]-\frac{l^2}{4}(\lambda_{1}+\lambda_{2})^2}\Big)/2)t}\Big]^{\frac{4}{l+1}}$
\end{Theorem}
\begin{proof}
In this case, $(34)$ is changed into the simpler form
$$v^{\prime\prime}(t)+\frac{l-3}{4}\frac{v^\prime(t)^2}{v(t)}-\frac{l}{2}(\lambda_{1}+\lambda_{2})v^\prime(t)
+\Big[(l+1)\lambda_{1}\lambda_{2}-\lambda_{1}^2-\lambda_{2}^2+\frac{\overline S}{l}\Big]v(t)=0. \eqno{(36)}  $$
Putting $v(t)=w(t)^{\frac{4}{l+1}},$ then $w(t)$ satisfies the equation
$w^{\prime\prime}(t)-\frac{l}{2}(\lambda_{1}+\lambda_{2})w^\prime(t)
+\frac{l+1}{4}\big[(l+1)\lambda_{1}\lambda_{2}-\lambda_{1}^2-\lambda_{2}^2+\frac{\overline S}{l}\big]w(t)=0,$
by the elementary methods for ordinary differential equations, we prove the above theorem.
\end{proof}
\begin{remark}
When $\lambda_{1}=\lambda_{2}=1, $ we get Theorem $30$ in $[12]$.
\end{remark}

When ${\rm dim}F=l\neq3,$ and $S^F\neq0,$ putting $v(t)=w(t)^{\frac{4}{l+1}},$ then $w(t)$ satisfies the following equation
:$$w^{\prime\prime}(t)-\frac{l}{2}(\lambda_{1}+\lambda_{2})w^\prime(t)+\frac{l+1}{4}\Big[(l+1)\lambda_{1}\lambda_{2}
-\lambda_{1}^2-\lambda_{2}^2+\frac{\overline S}{l}\Big]w(t)-\frac{l+1}{4l}S^Fw^{1-\frac{4}{l+1}}=0. \eqno{(37)}$$

\section{Multiply Warped Product with a Quarter-symmetric Connection}\label{sec4}
In this section, firstly we compute curvature of multiply twisted product with a quarter-symmetric connection, secondly we study the special multiply warped product with a quarter-symmetric connection, finally we consider the generalized Kasner space-times with a quarter-symmetric connection.

\subsection{Connection and Curvature}
By Lemma 2.3 and the equation (3), we have the following two Propositions:
\begin{Proposition}
Let $M=B \times_{b_{1}}F_{1} \times_{b_{2}}F_{2}\cdots \times_{b_{m}}F_{m}$ be a multiply twisted product,
If $X,Y\in\Gamma(TB)$, $U\in\Gamma(TF_{i}), W\in\Gamma(TF_{j})$ and $P\in\Gamma(TB).$ Then: \\
$(1)\overline\nabla_{X}Y=\overline\nabla^B_{X}Y; $  \\
$(2)\overline\nabla_{X}U=\frac{Xb_{i}}{b_{i}}U;$   \\
$(3)\overline\nabla_{U}X=\Big[\frac{Xb_{i}}{b_{i}}+\lambda_{1}\pi(X)\Big]U; $   \\
$(4)\overline\nabla_{U}W=0$ if $i\neq j;$\\
$(5)\overline\nabla_{U}W=U(lnb_{i})W+W(lnb_{i})U-\frac{g_{F_{i}}(U,W)}{b_{i}}grad_{F_{i}}b_{i}
-b_{i}g_{F_{i}}(U,W)grad_{B}b_{i}+\nabla^{F_{i}}_{U}W-\\ \mbox{} \; \quad \lambda_{2}g(U,W)P $ if $i=j.$
\end{Proposition}

\begin{Proposition}
Let $M=B \times_{b_{1}}F_{1} \times_{b_{2}}F_{2}\cdots \times_{b_{m}}F_{m}$ be a multiply twisted product.
If $X,Y\in\Gamma(TB)$, $U\in\Gamma(TF_{i}), W\in\Gamma(TF_{j})$ and $P\in\Gamma(TF_{r})$ for a fixed r, then: \\
$(1)\overline\nabla_{X}Y=\nabla^B_{X}Y-\lambda_{2}g(X,Y)P;  $  \\
$(2)\overline\nabla_{X}U=\frac{Xb_{i}}{b_{i}}U+\lambda_{1}\pi(U)X;$   \\
$(3)\overline\nabla_{U}X=\frac{Xb_{i}}{b_{i}}U; $   \\
$(4)\overline\nabla_{U}W=\lambda_{1}g(W,P)U$  if $i\neq j;$ \\
$(5)\overline\nabla_{U}W=U(lnb_{i})W+W(lnb_{i})U-\frac{g_{F_{i}}(U,W)}{b_{i}}grad_{F_{i}}b_{i}
-b_{i}g_{F_{i}}(U,W)grad_{B}b_{i}+\overline\nabla^{F_{i}}_{U}W$ \\ \mbox{}  \quad  if $i=j.$
\end{Proposition}

By Lemmas 2.3, 2.4 and the equation (4), we have the following two Propositions:
\begin{Proposition}
Let $M=B \times_{b_{1}}F_{1} \times_{b_{2}}F_{2}\cdots \times_{b_{m}}F_{m}$ be a multiply twisted product.
If $X,Y,Z\in\Gamma(TB)$, $V\in\Gamma(TF_{i}), W\in\Gamma(TF_{j}),U\in\Gamma(TF_{k})$ and $P\in\Gamma(TB),$ then: \\
$(1)\overline R(X,Y)Z=\overline R^B(X,Y)Z; $  \\
$(2)\overline R(V,X)Y=-\Big[\frac{H^{b_{i}}_{B}(X,Y)}{f}+\lambda_{2}\frac{Pb_{i}}{b_{i}}g(X,Y)
+\lambda_{1}\lambda_{2}\pi(P)g(X,Y)+\lambda_{1}g(Y,\nabla_{X}P) \\ \mbox{}\quad -\lambda_{1}^2\pi(X)\pi(Y)\Big]V;$ \\
$(3)\overline R(X,Y)V=0; $   \\
$(4)\overline R(V,W)X=VX(lnb_{i})W-WX(lnb_{i})V $ if $i=j;$   \\
$(5)\overline R(V,W)U=0 $  if $i=j\neq k$ or $i\neq j\neq k;$ \\
$(6)\overline R(X,V)W=\overline R(V,W)X=\overline R(V,X)W=0$  if $i\neq j;$ \\
$(7)\overline R(X,V)W=WX(lnb_{i})V-g(V,W)\Big[\frac{\nabla_{X}^B(grad_{B}b_{i})}{b_{i}}
+grad_{F_{i}}\frac{X(lnb_{i})}{b_{i}^2}+\lambda_{1}\frac{Pb_{i}}{b_{i}}X+\lambda_{2}\nabla_{X}P\\ \mbox{}\quad
+\lambda_{1}\lambda_{2}\pi(P)X-\lambda_{2}^2\pi(X)P\Big]$  if $i=j;$ \\
$(8)\overline R(U,V)W=-g(V,W)U\Big[\frac{g_{B}(grad_{B}b_{i},grad_{B}b_{k})}{b_{i}b_{k}}
+\lambda_{1}\frac{Pb_{i}}{b_{i}}+\lambda_{2}\frac{Pb_{k}}{b_{k}}+\lambda_{1}\lambda_{2}\pi(P)\Big]$if $i=j\neq k;$
\\$(9)\overline R(U,V)W=g(U,W)grad_{B}(V(lnb_{i}))-g(V,W)grad_{B}(U(lnb_{i}))+R^{F_{i}}(U,V)W-
\\ \mbox{}\quad \Big[\frac{|grad_{B}b_{i}|^2_{B}}{b_{i}^2}+(\lambda_{1}+\lambda_{2})\frac{Pb_{i}}{b_{i}}+
\lambda_{1}\lambda_{2}\pi(P)\Big][g(V,W)U-g(U,W)V]$   if $i=j=k.$
\end{Proposition}

\begin{Proposition}
Let $M=B \times_{b_{1}}F_{1} \times_{b_{2}}F_{2}\cdots \times_{b_{m}}F_{m}$ be a multiply twisted product.
If $X,Y,Z\in\Gamma(TB)$, $V\in\Gamma(TF_{i}), W\in\Gamma(TF_{j}),U\in\Gamma(TF_{k})$ and $P\in\Gamma(TF_{r})$ for a fixed r, then: \\
$(1)\overline R(X,Y)Z=R^B(X,Y)Z+\lambda_{2}\Big[g(X,Z)\frac{Yb_{r}}{b_{r}}-g(Y,Z)\frac{Xb_{r}}{b_{r}}\Big]P
+\lambda_{1}\lambda_{2}\pi(P)[g(X,Z)Y\\ \mbox{}\quad-g(Y,Z)X]; $  \\
$(2)\overline R(V,X)Y=-\frac{H^{b_{i}}_{B}(X,Y)}{b_{i}}V-\lambda_{1}\lambda_{2}\pi(P)g(X,Y)V$ if $i\neq r;$  \\
$(3)\overline R(V,X)Y=-\frac{H^{b_{i}}_{B}(X,Y)}{b_{i}}-\lambda_{1}\pi(V)\frac{Yb_{i}}{b_{i}}X
-\lambda_{2}g(X,Y)\nabla_{V}P-g(X,Y)[\lambda_{1}\lambda_{2}\pi(P)V \\ \mbox{} \quad -\lambda_{2}^2\pi(V)P]$
if $i=r;$\\
$(4)\overline R(X,Y)V=\lambda_{1}\pi(V)[\frac{Xb_{r}}{b_{r}}Y-\frac{Yb_{r}}{b_{r}}X]; $   \\
$(5)\overline R(V,W)X=-\lambda_{1}\delta_{i}^r\frac{Xb_{i}}{b_{i}}\pi(V)W
+\lambda_{1}\delta_{j}^r\frac{Xb_{j}}{b_{j}}\pi(W)V $   if $i\neq j;$ \\
$(6)\overline R(V,W)X=VX(lnb_{i})W-WX(lnb_{i})V-\lambda_{1}\delta_{i}^r\frac{Xb_{i}}{b_{i}}[\pi(V)W-\pi(W)V]$
if $i=j;$  \\
$(7)\overline R(V,W)U=0 $ if $i=j\neq k$ or $i\neq j\neq k;$   \\
$(8)\overline R(X,V)W=\lambda_{1}\frac{Xb_{r}}{b_{r}}\pi(W)V$  if $i\neq j;$\\
$(9)\overline R(X,V)W=WX(lnb_{i})V-g(V,W)\frac{\nabla_{X}^B(grad_{B}b_{i})}{b_{i}}
-grad_{F_{i}}(Xlnb_{i})g_{F_{i}}(V,W)\\ \mbox{}\quad
+\lambda_{1}\frac{Xb_{r}}{b_{r}}\pi(W)V-\lambda_{1}g(W,\nabla_{V}P)X-\lambda_{2}
g(V,W)\frac{Xb_{r}}{b_{r}}P-\lambda_{1}\lambda_{2}g(V,W)\pi(P)X\\ \mbox{}\quad
+\lambda_{1}^2\pi(W)\pi(V)X$ if $i\neq j;$  \\
$(10)\overline R(U,V)W=-g(V,W)\frac{g_{B}(grad_{B}b_{i},grad_{B}b_{k})}{b_{i}b_{k}}U
-\lambda_{1}g(W,\nabla_{V}P)U-\lambda_{2}g(V,W)\nabla_{U}P-\\ \mbox{}\quad \;\;\; \lambda_{1}\lambda_{2}\pi(P)g(V,W)U
+\lambda_{2}^2g(V,W)\pi(U)P+\lambda_{1}^2\pi(W)[\pi(V)U-\pi(U)V]$  if $i=j\neq k;$ \\
$(11)\overline R(U,V)W=g(U,W)grad_{B}(Vlnb_{i})-g(V,W)grad_{B}(Ulnb_{i})+R^{F_{i}}(U,V)W \\ \mbox{} \quad\;\;\;
-\frac{|grad_{B}b_{i}|^2_{B}}{b_{i}^2}[g(V,W)U-g(U,W)V]+\lambda_{1}\lambda_{2}\pi(P)[g(U,W)V-g(V,W)U]$
\\ \mbox{} \quad\;\;\; if $i=j=k\neq r;$\\
$(12)\overline R(U,V)W=g(U,W)grad_{B}(Vlnb_{i})-g(V,W)grad_{B}(Ulnb_{i})+R^{F_{i}}(U,V)W\\ \mbox{} \quad\;\;\;
-\frac{|grad_{B}f|^2_{B}}{f^2}[g(V,W)U-g(U,W)V]+\lambda_{1}[g(W,\nabla_{U}P)V-g(W,\nabla_{V}P)U]
\\ \mbox{}\quad\;\;\;+\lambda_{2}[g(U,W)\nabla_{V}P-g(V,W)\nabla_{U}P)]+\lambda_{1}\lambda_{2}\pi(P)[g(U,W)V-g(V,W)U]
\\ \mbox{}\quad\;\;\;+\lambda_{2}^2[g(V,W)\pi(U)-g(U,W)\pi(V)]P+\lambda_{1}^2\pi(W)[\pi(V)U-\pi(U)V]$
if $i=j=k=r.$ \\where $\delta_{i}^r$ denotes the Kronecker symbol.
\end{Proposition}

By Propositions 4.3 and 4.4 and the definition of the Ricci curvature tensor, we have the following two Propositions:
\begin{Proposition}
Let $M=B \times_{b_{1}}F_{1} \times_{b_{2}}F_{2}\cdots \times_{b_{m}}F_{m}$ be a multiply twisted product,
${\rm dim}M=\overline n,\;{\rm dim}B=n,\;{\rm dim}F_{i}=l_{i}.$
If $X,Y,Z\in\Gamma(TB)$, $V\in\Gamma(TF_{i}), W\in\Gamma(TF_{j})$ and $P\in\Gamma(TB),$ then: \\
$(1)\overline {Ric}(X,Y)=\overline {Ric}^B(X,Y)+\sum\limits_{i=1}^ml_{i}\Big[\frac{H^{b_{i}}_{B}(X,Y)}{b_{i}}
+\lambda_{2}\frac{Pb_{i}}{b_{i}}g(X,Y)+\lambda_{1}\lambda_{2}\pi(P)g(X,Y)
\\ \mbox{}\quad+\lambda_{1}g(Y,\nabla_{X}P)-\lambda_{1}^2\pi(X)\pi(Y)\Big]; $ \\
$(2)\overline {Ric}(X,V)=\overline {Ric}(V,X)=(l_{i}-1)[VX(lnb_{i})]; $   \\
$(3)\overline {Ric}(V,W)=0$  if $i\neq j;$ \\
$(3)\overline {Ric}(V,W)=\overline {Ric}^{F_{i}}(V,W)+\Big\{\frac{\Delta_{B}b_{i}}{b_{i}}
+(l_{i}-1)\frac{|grad_{B}b_{i}|^2_{B}}{b_{i}^2}
+\sum\limits_{s\neq i}l_{s}\frac{g_{B}(grad_{B}b_{i},grad_{B}b_{s})}{b_{i}b_{s}}
\\ \mbox{}\quad +[(\overline n-1)\lambda_{1}\lambda_{2}-\lambda_{2}^2]\pi(P)+\lambda_{2}div_{B}P
+\lambda_{2}\sum\limits_{s\neq i}l_{s}\frac{Pb_{s}}{b_{s}}+[(\overline n-1)\lambda_{1}
+(l_{i}-1)\lambda_{2}]\frac{Pb_{i}}{b_{i}}\Big\}g(V,W)$ \\ \mbox{}\quad if $i=j.$ \\
where  $div_{B}P=\sum\limits_{k=1}^n\varepsilon_{k}\langle\nabla_{E_{k}}P,E_{k}\rangle,$ and $E_{k},
1\leq k\leq n $ is an orthonormal base of $B $ with $\varepsilon_{k}=g(E_{k},E_{k}).$
\end{Proposition}
As a Corollary of Proposition $4.5,$ we have:
\begin{Corollary}
Let $M=B \times_{b_{1}}F_{1} \times_{b_{2}}F_{2}\cdots \times_{b_{m}}F_{m}$ be a multiply twisted product, and
${\rm dim}F_{i}=l_{i}>1,\;P\in\Gamma(TB),$ then $(M,\overline\nabla)$ is mixed Ricci-flat if and only if $M$ can be expressed as a multiply warped product. In particular, if $(M,\overline\nabla)$ is Einstein, then $M$ can be expressed as a multiply warped product.
\end{Corollary}
\begin{proof}
By Proposition $4.5.(2)$ and $(3)$, similar to the proof of Theorem $1$ in $[5],$ we get this Corollary.
\end{proof}
\begin{Proposition}
Let $M=B \times_{b_{1}}F_{1} \times_{b_{2}}F_{2}\cdots \times_{b_{m}}F_{m}$ be a multiply twisted product,
${\rm dim}M=\overline n,\;{\rm dim}B=n,\;{\rm dim}F_{i}=l_{i}.$
If $X,Y,Z\in\Gamma(TB)$, $V\in\Gamma(TF_{i}), W\in\Gamma(TF_{j})$ and $P\in\Gamma(TF_{r})$ for a fixed r, then: \\
$(1)\overline {Ric}(X,Y)=\overline {Ric}^B(X,Y)+\sum\limits_{i=1}^ml_{i}\frac{H^{b_{i}}_{B}(X,Y)}{b_{i}}
+[(\overline n-1)\lambda_{1}\lambda_{2}-\lambda_{2}^2]\pi(P)g(X,Y)\\ \mbox{}\quad +\lambda_{2}g(X,Y)div_{F_{r}}P;$ \\
$(2)\overline {Ric}(X,V)=(l_{i}-1)[VX(lnb_{i})]
+[(\overline n-1)\lambda_{1}-\lambda_{2}]\pi(V)\frac{Xb_{r}}{b_{r}};$\\
$(3)\overline {Ric}(V,X)=(l_{i}-1)[VX(lnb_{i})]
+[\lambda_{2}-(\overline n-1)\lambda_{1}]\pi(V)\frac{Xb_{r}}{b_{r}}; $   \\
$(4)\overline {Ric}(V,W)=0$ if $i\neq j;$\\
$(4)\overline {Ric}(V,W)=Ric^{F_{i}}(V,W)+g(V,W)\Big\{\frac{\Delta_{B}b_{i}}{b_{i}}
+(l_{i}-1)\frac{|grad_{B}b_{i}|^2_{B}}{b_{i}^2}
+\sum\limits_{s\neq i}l_{s}\frac{g_{B}(grad_{B}b_{i},grad_{B}b_{s})}{b_{i}b_{s}}\\ \mbox{}\quad
+[(\overline n-1)\lambda_{1}\lambda_{2} -\lambda_{2}^2]\pi(P)\Big\}
+[(\overline n-1)\lambda_{1}-\lambda_{2}]g(W,\nabla_{V}P)+[\lambda_{2}^2+(1-\overline n)\lambda_{1}^2]\pi(V)\pi(W)
\\ \mbox{}\quad +\lambda_{2}g(V,W)div_{F_{r}}P$  if $i=j.$
\end{Proposition}
As a Corollary of Proposition $4.7,$ we have:
\begin{Corollary}
Let $M=B \times_{b_{1}}F_{1} \times_{b_{2}}F_{2}\cdots \times_{b_{m}}F_{m}$ be a multiply twisted product, and
${\rm dim}F_{i}=l_{i}>1,\;P\in\Gamma(TF_{r}),$ then $(M,\overline\nabla)$ is mixed Ricci-flat if and only if one of the following two conditions is satisfied:\\
$(1)$\;$\lambda_{2}=(\overline n-1)\lambda_{1},$ and $M$ can be expressed as a multiply warped product; \\
$(2)$\;$\lambda_{2}\neq(\overline n-1)\lambda_{1},\;M$ can be expressed as a multiply warped product and $b_{r}$ is only \\ \mbox{} \quad dependent on $F_{r};$  \\
In particular, if $(M,\overline\nabla)$ is Einstein, then $M$ can be expressed as a multiply warped product.
\end{Corollary}
\begin{proof}
By Proposition $4.7.(2)$ and $(3),$ we have that $(M,\overline\nabla)$ is mixed Ricci-flat if and only if
$VX(lnb_{i})=0,\;[(\overline n-1)\lambda_{1}-\lambda_{2}]\pi(V)\frac{Xb_{r}}{b_{r}}=0.$
Similar to the proof of Corollary $4.6,$ we get that\\
¢Ù $\lambda_{2}=(\overline n-1)\lambda_{1},$ and $M$ can be expressed as a multiply warped product.\\
¢Ú $\lambda_{2}\neq(\overline n-1)\lambda_{1},\;M$ can be expressed as a multiply warped product.
When $i\neq r,\;\pi(V)=0.$ When $i=r,$ by $\pi(V)\frac{Xb_{r}}{b_{r}}=0,$ then $b_{r}$ depends only on $F_{r}.$
\end{proof}
By Proposition 4.5 and the definition of the scalar curvature, we have the following:
\begin{Proposition}
Let $M=B \times_{b_{1}}F_{1} \times_{b_{2}}F_{2}\cdots \times_{b_{m}}F_{m}$ be a multiply twisted product,
${\rm dim}M=\overline n,\;{\rm dim}B=n,\;{\rm dim}F_{i}=l_{i}. $ If $P\in\Gamma(TB),$
then the scalar curvature $\overline S$ has the following expression: \\
$\overline S=\overline S^B+2\sum\limits_{i=1}^ml_{i}\frac{\Delta_{B}b_{i}}{b_{i}}
+\sum\limits_{i=1}^m\frac{S^{F_{i}}}{b_{i}^2}+\sum\limits_{i=1}^ml_{i}(l_{i}-1)\frac{|grad_{B}b_{i}|^2_{B}}{b_{i}^2}
+\sum\limits_{i=1}^m\sum\limits_{s\neq i}l_{i}l_{s}\frac{g_{B}(grad_{B}b_{i},grad_{B}b_{s})}{b_{i}b_{s}}
\\\mbox{}\quad+\sum\limits_{i=1}^ml_{i}[(\overline n-1)\lambda_{1}+(n+l_{i}-1)\lambda_{2})]\frac{Pb_{i}}{b_{i}}
+\lambda_{2}\sum\limits_{i=1}^m\sum\limits_{s\neq i}l_{i}l_{s}\frac{Pb_{s}}{b_{s}}
+\sum\limits_{i=1}^ml_{i}[(\overline n+n-1)\lambda_{1}\lambda_{2}\\\mbox{}\quad
-(\lambda_{1}^2+\lambda_{2}^2)]\pi(P)+(\lambda_{1}+\lambda_{2})\sum\limits_{i=1}^ml_{i}div_{B}P.$
\end{Proposition}

By Proposition 4.7 and the definition of the scalar curvature, we have the following:
\begin{Proposition}
Let $M=B \times_{b_{1}}F_{1} \times_{b_{2}}F_{2}\cdots \times_{b_{m}}F_{m}$ be a multiply twisted product,
${\rm dim}M=\overline n,\;{\rm dim}B=n,\;{\rm dim}F_{i}=l_{i}. $ If $P\in\Gamma(TF_{r}),$
then the scalar curvature $\overline S$ has the following expression: \\
$\overline S=\overline S^B+2\sum\limits_{i=1}^ml_{i}\frac{\Delta_{B}b_{i}}{b_{i}}
+\sum\limits_{i=1}^m\frac{S^{F_{i}}}{b_{i}^2}+\sum\limits_{i=1}^ml_{i}(l_{i}-1)\frac{|grad_{B}b_{i}|^2_{B}}{b_{i}^2}
+\sum\limits_{i=1}^m\sum\limits_{s\neq i}l_{i}l_{s}\frac{g_{B}(grad_{B}b_{i},grad_{B}b_{s})}{b_{i}b_{s}}
\\ \mbox{}\quad\;\;+[\overline n(\overline n-1)\lambda_{1}\lambda_{2}
+(1-\overline n)(\lambda_{1}^2+\lambda_{2}^2)]\pi(P)+(\overline n-1)(\lambda_{1}+\lambda_{2})div_{F_{r}}P.  $
\end{Proposition}
\begin{remark}
$(1)$ It is easy to see that Propositions $3.1-3.8$ are Corollaries of Propo-\\ \mbox{} \qquad \qquad \qquad \;\;\; sitions $4.1, 4.2, 4.3, 4.4, 4.5, 4.7, 4.9, 4.10,$ respectively.\\ \mbox{} \qquad \qquad \quad \;
$(2)$ When $\lambda_{1}=\lambda_{2}=1,$ we get Propositions $1,2,4,5,7,9,12,13$ in $[12],$ by \\ \mbox{} \qquad \qquad \qquad \quad \, Propositions $4.1, 4.2, 4.3, 4.4, 4.5, 4.7, 4.9, 4.10,$ respectively.
\end{remark}
\subsection{Special Multiply Warped Product with a Quarter-symmetric Connection}
Let $M=I \times_{b_{1}}F_{1} \times_{b_{2}}F_{2}\cdots \times_{b_{m}}F_{m}$ be a multiply warped product with the metric tensor $g=-dt^2\oplus b_{1}^2g_{F_{1}}\oplus b_{2}^2g_{F_{2}}\cdots\oplus b_{m}^2g_{F_{m}}$ and $I$ is an open interval in $\mathbb{R},$ and $b_{i}\in C^\infty(I),\;{\rm dim}M=\overline n,\;{\rm dim}I=1,\;{\rm dim}F_{i}=l_{i}.$  \par Similar to the proof method of Theorem $3.11,$ we have:
\begin{Theorem}
Let $M=I \times_{b_{1}}F_{1} \times_{b_{2}}F_{2}\cdots \times_{b_{m}}F_{m}$ be a multiply warped product with the metric tensor $g=-dt^2\oplus b_{1}^2g_{F_{1}}\oplus b_{2}^2g_{F_{2}}\cdots\oplus b_{m}^2g_{F_{m}},\;  P=\frac{\partial}{\partial t}.$ Then
$(M,\overline\nabla)$ is Einstein with the Einstein constant $\alpha$ if and only if the following
conditions are satisfied \\
$(1)$ $(F_{i},\nabla^{F_{i}})$ is Einstein with the Einstein constant $\alpha_{i},\;i\in\{1,\dots,m\};$  \\
$(2)$ $\sum\limits_{i=1}^ml_{i}\Big(\lambda_{2}\frac{b_{i}^\prime}{b_{i}}-\frac{b_{i}^{\prime\prime}}{b_{i}}
+\lambda_{1}^2-\lambda_{1}\lambda_{2}\Big)=\alpha;$ \\
$(3)$ $\alpha_{i}-b_{i}b_{i}^{\prime\prime}+(1-l_{i})(b_{i}^\prime)^2+(\lambda_{2}b_{i}^2-b_{i}b_{i}^\prime)
\sum\limits_{s\neq i}l_{s}\frac{b_{s}^\prime}{b_{s}}+[\lambda_{2}^2+(1-\overline n)\lambda_{1}\lambda_{2}]b_{i}^2
+[(\overline n-1)\lambda_{1}+\\ \mbox{} \quad\;\;(l_{i}-1)\lambda_{2}]b_{i}b_{i}^\prime=\alpha b_{i}^2.$
\end{Theorem}
\begin{Theorem}
Let $M=I \times_{b_{1}}F_{1} \times_{b_{2}}F_{2}\cdots \times_{b_{m}}F_{m}$ be a multiply warped product with the metric tensor $g=-dt^2\oplus b_{1}^2g_{F_{1}}\oplus b_{2}^2g_{F_{2}}\cdots\oplus b_{m}^2g_{F_{m}},\;  P\in\Gamma(TF_{r})$ with $g_{F_{r}}(P,P)=1$ and $\overline n>2.$ Then $(M,\overline\nabla)$ is Einstein with the Einstein constant $\alpha$ if and only if the following conditions are satisfied for any $i\in\{1,\dots,m\}:$\\
$(1)\;(F_{i},\nabla^{F_{i}}) (i\neq r)$ is Einstein with the Einstein constant $\alpha_{i},i\in\{1,\dots,m\};$ \\
$(2)\;b_{r}$ is a constant and $\sum\limits_{i=1}^ml_{i}\frac{b_{i}^{\prime\prime}}{b_{i}}=\mu_{0};\;
div_{F_{r}}P=\mu_{1};\;\mu_{0}-\lambda_{2}\mu_{1}+\alpha=[(\overline n-1)\lambda_{1}\lambda_{2}
-\lambda_{2}^2]b_{r}^2,$ \\ \mbox{} \quad where $\mu_{0},\mu_{1}$ are constants; \\
$(3)\;Ric^{F_{r}}(V,W)+\overline\alpha g_{F_{r}}(V,W)=[(\overline n-1)\lambda_{1}^2-\lambda_{2}^2]\pi(V)\pi(W)
-[(\overline n-1)\lambda_{1}-\lambda_{2}]g(W,\nabla_{V}P),$ \\ \mbox{} \quad for $V,W\in\Gamma(TF_{r});$  \\
$(4)\;\alpha_{i}-b_{i}b_{i}^{\prime\prime}+[(\overline n-1)\lambda_{1}\lambda_{2}-\lambda_{2}^2]b_{i}^2b_{r}^2
-b_{i}b_{i}^\prime\sum\limits_{s\neq i}l_{s}\frac{b_{s}^\prime}{b_{s}}
-(l_{i}-1)(b_{i}^\prime)^2=(\alpha-\lambda_{2}\mu_{1})b_{i}^2.$
\end{Theorem}
\begin{proof}
By Proposition $4.7(2)$ and $g_{F_{r}}(P,P)=1,$ we have that $b_{r}$ is a constant. By Proposition $4.7(1),$ we have
$$\overline {Ric}\Big(\frac{\partial}{\partial t},\frac{\partial}{\partial t}\Big)=
\sum\limits_{i=1}^ml_{i}\frac{b_{i}^{\prime\prime}}{b_{i}}+[\lambda_{2}^2
+(1-\overline n)\lambda_{1}\lambda_{2}]b_{r}^2-\lambda_{2}div_{F_{r}}P=-\alpha.$$ By variables separation, we have
$$\sum\limits_{i=1}^ml_{i}\frac{b_{i}^{\prime\prime}}{b_{i}}=\mu_{0};\;div_{F_{r}}P=\mu_{1};\;
\mu_{0}-\lambda_{2}\mu_{1}+\alpha=[(\overline n-1)\lambda_{1}\lambda_{2}-\lambda_{2}^2]b_{r}^2,$$ then we get $(2).$
By Proposition $4.7(3),$ we have
\begin{eqnarray*}
\overline{Ric}(V,W)&=&Ric^{F_{i}}(V,W)+b_{i}^2g_{F_{i}}(V,W)\Big\{-\frac{b_{i}^{\prime\prime}}{b_{i}}
+(l_{i}-1)\frac{-(b_{i}^\prime)^2}{b_{i}^2}+\sum\limits_{s\neq i}l_{s}\frac{-b_{i}^\prime b_{s}^{\prime}}{b_{i}b_{s}}
\\& &+[(\overline n-1)\lambda_{1}\lambda_{2}-\lambda_{2}^2]\pi(P)\Big\}+[(\overline n-1)\lambda_{1}
-\lambda_{2}]g(W,\nabla_{V}P)\\& &+[\lambda_{2}^2+(1-\overline n)\lambda_{1}^2]\pi(V)\pi(W)
+\lambda_{2}g(V,W)div_{F_{r}}P.
\end{eqnarray*}
When $i\neq r,$ then $\nabla_{V}P=\pi(V)=0,$ so
\begin{eqnarray*}
\overline{Ric}(V,W)&=&Ric^{F_{i}}(V,W)+b_{i}^2g_{F_{i}}(V,W)\Big\{-\frac{b_{i}^{\prime\prime}}{b_{i}}
+(l_{i}-1)\frac{-(b_{i}^\prime)^2}{b_{i}^2}+\sum\limits_{s\neq i}l_{s}\frac{-b_{i}^\prime b_{s}^{\prime}}{b_{i}b_{s}}
\\& &+[(\overline n-1)\lambda_{1}\lambda_{2}-\lambda_{2}^2]b_{r}^2\Big\}+\lambda_{2}\mu_{1}b_{i}^2g_{F_{i}}(V,W)
=\alpha b_{i}^2g_{F_{i}}(V,W).
\end{eqnarray*}
By variables separation, we have that $(F_{i},\nabla^{F_{i}}) (i\neq r)$ is Einstein with the Einstein constant $\alpha_{i}$ and $$\alpha_{i}-b_{i}b_{i}^{\prime\prime}+[(\overline n-1)\lambda_{1}\lambda_{2}
-\lambda_{2}^2]b_{i}^2b_{r}^2-b_{i}b_{i}^\prime\sum\limits_{s\neq i}l_{s}\frac{b_{s}^\prime}{b_{s}}
-(l_{i}-1)(b_{i}^\prime)^2=(\alpha-\lambda_{2}\mu_{1})b_{i}^2.$$ Then we get $(1)$ and $(4).$\\
When $i=r$ and $b_{r}$ is a constant, then
\begin{eqnarray*}
Ric^{F_{r}}(V,W)+b_{r}^2\big\{[(\overline n-1)\lambda_{1}\lambda_{2}-\lambda_{2}^2]b_{r}^2
+\lambda_{2}\mu_{1}-\alpha\big\}g_{F_{r}}(V,W)\\=[(\overline n-1)\lambda_{1}^2-\lambda_{2}^2]\pi(V)\pi(W)
-[(\overline n-1)\lambda_{1}-\lambda_{2}]g(W,\nabla_{V}P),
\end{eqnarray*}
let $\overline\alpha=b_{r}^2\big\{[(\overline n-1)\lambda_{1}\lambda_{2}-\lambda_{2}^2]b_{r}^2+\lambda_{2}\mu_{1}-\alpha\big\},$ we get $(3).$
\end{proof}
\begin{Theorem}
Let $M=I \times_{b_{1}}F_{1} \times_{b_{2}}F_{2}\cdots \times_{b_{m}}F_{m}$ be a multiply warped product and $P=\frac{\partial}{\partial t}.$ If $(M,\overline\nabla)$ has constant scalar curvature $\overline S,$ then each $(F_{i},\nabla^{F_{i}})$ has constant scalar curvature $S^{F_{i}}.$
\end{Theorem}
\begin{proof}
By Proposition $4.9,$ we have
\begin{eqnarray}
\overline S &=&-2\sum\limits_{i=1}^ml_{i}\frac{b_{i}^{\prime\prime}}{b_{i}}
+\sum\limits_{i=1}^m\frac{S^{F_{i}}}{b_{i}^2}+\sum\limits_{i=1}^ml_{i}(l_{i}-1)\frac{-(b_{i}^\prime)^2}{b_{i}^2}
+\sum\limits_{i=1}^m\sum\limits_{s\neq i}l_{i}l_{s}\frac{-b_{i}^\prime b_{s}^\prime}{b_{i}b_{s}} \nonumber \\
\setcounter{equation}{38}
& &+\sum\limits_{i=1}^ml_{i}[(\overline n-1)\lambda_{1}+l_{i}\lambda_{2})]\frac{b_{i}^\prime}{b_{i}}
+\lambda_{2}\sum\limits_{i=1}^m\sum\limits_{s\neq i}l_{i}l_{s}\frac{b_{s}^\prime}{b_{s}}  \\
& &-\sum\limits_{i=1}^ml_{i}[\overline n\lambda_{1}\lambda_{2}-(\lambda_{1}^2+\lambda_{2}^2)]. \nonumber
\end{eqnarray}
Note that each $S^{F_{i}}$ is function defined on $F_{i},$ using variables separation we complete this proof.
\end{proof}
\begin{Theorem}
Let $M=I \times_{b_{1}}F_{1} \times_{b_{2}}F_{2}\cdots \times_{b_{m}}F_{m}$ be a multiply warped product and
$P\in\Gamma(TF_{r}).$ If $(M,\overline\nabla)$ has constant scalar curvature $\overline S,$ then \\
$(1)$ each $(F_{i},\nabla^{F_{i}}) (i\neq r)$ has constant scalar curvature $S^{F_{i}};$ \\
$(2)$ If $\lambda_{1}+\lambda_{2}\neq0$ and $\lambda_{1}^2+\lambda_{2}^2-\overline n\lambda_{1}\lambda_{2}=0,$ and $div_{F_{r}}P$ is a constant, then $S^{F_{r}}$ is a constant;\\
$(3)$ If $\lambda_{1}=-\lambda_{2}\neq0$ and $g_{F_{r}}(P,P)$  is a constant, then $S^{F_{r}}$ is a constant;\\
$(4)$ If $\lambda_{1}+\lambda_{2}\neq0$ and $\lambda_{1}^2+\lambda_{2}^2-\overline n\lambda_{1}\lambda_{2}\neq0,$
and $div_{F_{r}}P,g_{F_{r}}(P,P)$ are constants, then $S^{F_{r}}$ is a constant.
\end{Theorem}
\begin{proof}
By Proposition $4.10,$ we have
\begin{eqnarray}
\setcounter{equation}{39}
\overline S &=&-2\sum\limits_{i=1}^ml_{i}\frac{b_{i}^{\prime\prime}}{b_{i}}
+\sum\limits_{i=1}^m\frac{S^{F_{i}}}{b_{i}^2}+\sum\limits_{i=1}^ml_{i}(l_{i}-1)\frac{-(b_{i}^\prime)^2}{b_{i}^2}
+\sum\limits_{i=1}^m\sum\limits_{s\neq i}l_{i}l_{s}\frac{-b_{i}^\prime b_{s}^\prime}{b_{i}b_{s}} \\
& &+[\overline n(\overline n-1)\lambda_{1}\lambda_{2}+(1-\overline n)(\lambda_{1}^2
+\lambda_{2}^2)]b_{r}^2g_{F_{r}}(P,P)+(\overline n-1)(\lambda_{1}+\lambda_{2})div_{F_{r}}P    \nonumber
\end{eqnarray}
then similar to the proof of Proposition $3.18,$ we complete this Proposition.
\end{proof}
\begin{remark}
When $\lambda_{1}=\lambda_{2}=1,$ we get Theorems $15,16,$ Propositions $18,19$ in $[12]$ by Theorems $4.11-4.14,$ respectively.
\end{remark}

\subsection{Generalized Kasner Space-times with a Quarter-symmetric Connection}
In this section, we consider the Einstein and scalar curvature of generalized Kasner space-times with a quarter-symmetric connection. We recall the definition of generalized Kasner space-times in $[5].$
\newtheorem{Definition}[Proposition]{Definition}
\begin{Definition}
A generalized Kasner space-time $(M, g)$ is a Lorentzian multiply warped product of the form
$M=I \times_{\phi^{p_{1}}}F_{1} \times_{\phi^{p_{2}}}F_{2}\cdots \times_{\phi^{p_{m}}}F_{m}$ with the metric tensor
$g=-dt^2\oplus \phi^{2p_{1}}g_{F_{1}}\oplus \phi^{2p_{2}}g_{F_{2}}\cdots\oplus \phi^{2p_{m}}g_{F_{m}},$
where $\phi:I\to(0,\infty)$ is smooth and $p_{i}\in\mathbb{R},$ for any $i\in\{1,\dots,m\}$ and also $I=(t_{1},t_{2}).$
\end{Definition}
We introduce the following parameters $\zeta=\sum\limits_{i=1}^ml_{i}p_{i}$ and $\eta=\sum\limits_{i=1}^ml_{i}p_{i}^2$ for generalized Kasner space-times. By Theorem $4.11$ and direct computations,
we get the following:
\begin{Proposition}
Let $M=I \times_{\phi^{p_{1}}}F_{1} \times_{\phi^{p_{2}}}F_{2}\cdots \times_{\phi^{p_{m}}}F_{m}$ be a generalized Kasner space-time and $P=\frac{\partial}{\partial t}.$ Then $(M,\overline\nabla)$ is Einstein with the Einstein constant $\alpha$ if and only if the following conditions are satisfied for any $i\in\{1,\dots,m\}:$\\
$(1)\;(F_{i},\nabla^{F_{i}})$ is Einstein with the Einstein constant $\alpha_{i},i\in\{1,\dots,m\};$ \\
$(2)\;\zeta\Big(\lambda_{2}\frac{\phi^\prime}{\phi}-\frac{\phi^{\prime\prime}}{\phi}\Big)
-(\eta-\zeta)\frac{(\phi^\prime)^2}{\phi^2}+(\lambda_{1}^2-\lambda_{1}\lambda_{2})(\overline n-1)=\alpha;$\\
$(3)\;\frac{\alpha_{i}}{\phi^{2p_{i}}}-p_{i}\frac{\phi^{\prime\prime}}{\phi}
-(\zeta-1)p_{i}\frac{(\phi^\prime)^2}{\phi^2}
+\{\lambda_{2}\zeta+[(\overline n-1)\lambda_{1}-\lambda_{2}]p_{i}\}\frac{\phi^\prime}{\phi}
=\alpha-\lambda_{2}^2+(\overline n-1)\lambda_{1}\lambda_{2}.$
\end{Proposition}
By the equation $(38),$ we obtain the following:
\begin{Proposition}
Let $M=I \times_{\phi^{p_{1}}}F_{1} \times_{\phi^{p_{2}}}F_{2}\cdots \times_{\phi^{p_{m}}}F_{m}$ be a generalized Kasner space-time and $P=\frac{\partial}{\partial t}.$ Then $(M,\overline\nabla)$ has constant scalar curvature
$\overline S$ if and only if each $(F_{i},\nabla^{F_{i}})$ has constant scalar curvature $S^{F_{i}}$ and
$$\overline S=\sum\limits_{i=1}^m\frac{S^{F_{i}}}{\phi^{2p_{i}}}-2\zeta\frac{\phi^{\prime\prime}}{\phi}
-(\eta+\zeta^2-2\zeta)\frac{(\phi^\prime)^2}{\phi^2}
+(\lambda_{1}+\lambda_{2})\zeta(\overline n-1)\frac{\phi^\prime}{\phi}
+(\overline n-1)(\lambda_{1}^2+\lambda_{2}^2-\overline n\lambda_{1}\lambda_{2}). \eqno{(40)}$$
\end{Proposition}
Next, we first give a classification of four-dimensional generalized Kasner space-times with a quarter-symmetric connection and then consider Ricci tensors and scalar curvatures of them.
\begin{Definition}
Let $M=I \times_{b_{1}}F_{1} \times_{b_{2}}F_{2}\cdots \times_{b_{m}}F_{m}$ be a multiply warped product with the metric tensor $g=-dt^2\oplus b_{1}^2g_{F_{1}}\oplus b_{2}^2g_{F_{2}}\cdots\oplus b_{m}^2g_{F_{m}}.$ Then:\\
$(1)\;(M,g)$ is said to be of type $(I)$ if $m=1$ and ${\rm dim}F=3;$ \\
$(2)\;(M,g)$ is said to be of type $(II)$ if $m=2$ and ${\rm dim}F_{1}=1$ and ${\rm dim}F_{2}=2;$ \\
$(3)\;(M,g)$ is said to be of type $(III)$ if $m=3$ and ${\rm dim}F_{1}=1,\;{\rm dim}F_{2}=1,\;{\rm dim}F_{3}=1.$ \\
\end{Definition}
\subsubsection{Classification of Einstein Type (I) generalized Kasner space-times with a quarter-symmetric connection}
By Theorem $3.16,$ we have given a classification of Einstein Type (I) generalized Kasner space-times with a quarter-symmetric connection.

\subsubsection{Type (I) generalized Kasner space-times with a quarter-symmetric connection with constant scalar curvature}
By Theorem $3.19,$ we have given a classification of Type (I) generalized Kasner space-times with a quarter-symmetric connection with constant scalar curvature.

\subsubsection{Classification of Einstein Type (II) generalized Kasner space-times with a quarter-symmetric connection}
Let $M=I\times_{\phi^{p_{1}}}F_{1}\times_{\phi^{p_{2}}}F_{2}$ be an Einstein Type (II) generalized Kasner space-times
and $P=\frac{\partial}{\partial t}.$ Then $\alpha_{1}=0$ because of
${\rm dim}F_{1}=1.\;\zeta=p_{1}+2p_{2},\;\eta=p_{1}^2+2p_{2}^2.$ By Proposition $4.16,$ we have $$\zeta\Big(\lambda_{2}\frac{\phi^\prime}{\phi}-\frac{\phi^{\prime\prime}}{\phi}\Big)
-(\eta-\zeta)\frac{(\phi^\prime)^2}{\phi^2}+3(\lambda_{1}^2-\lambda_{1}\lambda_{2})=\alpha. \eqno{(41a)}$$
$$-p_{1}\frac{\phi^{\prime\prime}}{\phi}-(\zeta-1)p_{1}\frac{(\phi^\prime)^2}{\phi^2}
+[\lambda_{2}\zeta+(3\lambda_{1}-\lambda_{2})p_{1}]\frac{\phi^\prime}{\phi}
=\alpha-\lambda_{2}^2+3\lambda_{1}\lambda_{2}.    \eqno{(41b)} $$
$$\frac{\alpha_{2}}{\phi^{2p_{2}}}-p_{2}\frac{\phi^{\prime\prime}}{\phi}-(\zeta-1)p_{2}\frac{(\phi^\prime)^2}{\phi^2}
+[\lambda_{2}\zeta+(3\lambda_{1}-\lambda_{2})p_{2}]\frac{\phi^\prime}{\phi}
=\alpha-\lambda_{2}^2+3\lambda_{1}\lambda_{2}.   \eqno{(41c)}$$
where $\alpha_{2}$ is a constant. Consider the following two cases:\\
Case $\underline{(i)\;(\zeta=0).}$ In this case, $p_{2}=-\frac{1}{2}p_{1},\;\eta=\frac{3}{2}p_{1}^2.$ Then by equations $(41a)-(41c),$ we have:
$$-\eta\frac{(\phi^\prime)^2}{\phi^2}+3(\lambda_{1}^2-\lambda_{1}\lambda_{2})=\alpha. \eqno{(42a)}$$
$$p_{1}\Big[-\frac{\phi^{\prime\prime}}{\phi}+\frac{(\phi^\prime)^2}{\phi^2}+(3\lambda_{1}
-\lambda_{2})\frac{\phi^\prime}{\phi}\Big]=\alpha-\lambda_{2}^2+3\lambda_{1}\lambda_{2}.    \eqno{(42b)} $$
$$\frac{\alpha_{2}}{\phi^{-p_{1}}}-\frac{1}{2}p_{1}\Big[-\frac{\phi^{\prime\prime}}{\phi}+\frac{(\phi^\prime)^2}{\phi^2}
+(3\lambda_{1}-\lambda_{2})\frac{\phi^\prime}{\phi}\Big]=\alpha-\lambda_{2}^2+3\lambda_{1}\lambda_{2}. \eqno{(42c)}$$
¢Ù$\;\underline{\eta=0.}$ We have $p_{i}=0,$ by the equation $(42a),$ we get $\alpha=3\lambda_{1}^2-3\lambda_{1}\lambda_{2};$   by the equation $(42b),$ we get $\alpha=\lambda_{2}^2-3\lambda_{1}\lambda_{2};$ then we have $\lambda_{2}^2=3\lambda_{1}^2$ and by the equation $(42c),$ we get $\alpha_{2}=0.$ So we have $\underline{\lambda_{2}^2=3\lambda_{1}^2,\;p_{i}=0,\;
\alpha=3\lambda_{1}^2-3\lambda_{1}\lambda_{2}=\lambda_{2}^2-3\lambda_{1}\lambda_{2},\;\alpha_{1}=\alpha_{2}=0.}$\\
¢Ú$\;\underline{\eta\neq0.}$ We have $p_{i}\neq0.$\\
$1)\;\underline{\alpha_{2}=0.}$ By equations $(42b),(42c),$ we get $\alpha=\lambda_{2}^2-3\lambda_{1}\lambda_{2}$ and
$$-\frac{\phi^{\prime\prime}}{\phi}+\frac{(\phi^\prime)^2}{\phi^2}+(3\lambda_{1}
-\lambda_{2})\frac{\phi^\prime}{\phi}=0;  \eqno{(43a)}$$
$$\frac{(\phi^\prime)^2}{\phi^2}=\frac{3\lambda_{1}^2-\lambda_{2}^2}{\eta}.   \eqno{(43b)} $$
$1^\prime.\;3\lambda_{1}^2-\lambda_{2}^2<0,$ we have no solution;\\
$2^\prime.\;3\lambda_{1}^2-\lambda_{2}^2=0,$ we have $\phi=c,$ then by the equation $(42b),$ we get $\alpha=\lambda_{2}^2-3\lambda_{1}\lambda_{2}.$\\
\mbox{}\quad \; So we have $\underline{\lambda_{2}^2=3\lambda_{1}^2,\;p_{1}\neq0,\;p_{2}\neq0\;\alpha=\lambda_{2}^2-3\lambda_{1}\lambda_{2},\;
\alpha_{1}=\alpha_{2}=0,\;\phi=c.}$\\
$3^\prime.\;3\lambda_{1}^2-\lambda_{2}^2>0,$ we have $\phi=c_{0}e^{\pm\sqrt{\frac{3\lambda_{1}^2-\lambda_{2}^2}{\eta}}t},$ by the equation $(43a),$ we get
$\lambda_{2}=3\lambda_{1},$ considering that $3\lambda_{1}^2-\lambda_{2}^2>0,$ we have $\lambda_{1}^2<0,$ which is a
contradiction. \\
$2)\;\underline{\alpha_{2}\neq0.}$ By equations $(42b),(42c),$ we get
$\frac{\alpha_{2}}{\phi^{-p_{1}}}=\frac{3}{2}(\alpha-\lambda_{2}^2+3\lambda_{1}\lambda_{2}),$ so $\phi=c;$
then by the equation $(42b),$ we get $\alpha-\lambda_{2}^2+3\lambda_{1}\lambda_{2}=0,$ then $\alpha_{2}=0,$ this is a
contradiction. \\
Case $\underline{(ii)\;(\zeta\neq0).}$ Then $\eta\neq0.$ Putting $\phi=\psi^{\frac{\zeta}{\eta}},$ then
$\psi^{\prime\prime}-\lambda_{2}\psi^{\prime}+(\alpha+3\lambda_{1}\lambda_{2}-3\lambda_{1}^2)
\frac{\eta}{\zeta^2}\psi=0.$ Hence:\\
$(1)\;\alpha<\frac{\lambda_{2}^2\zeta^2}{4\eta}+3\lambda_{1}^2-3\lambda_{1}\lambda_{2},\;
\psi=c_{1}e^{\frac{\lambda_{2}+\sqrt{\lambda_{2}^2-4(\alpha+3\lambda_{1}\lambda_{2}-3\lambda_{1}^2)
\frac{\eta}{\zeta^2}}}{2}t}+c_{2}e^{\frac{\lambda_{2}-\sqrt{\lambda_{2}^2-4(\alpha+3\lambda_{1}\lambda_{2}-3\lambda_{1}^2)
\frac{\eta}{\zeta^2}}}{2}t};$\\
$(2)\;\alpha=\frac{\lambda_{2}^2\zeta^2}{4\eta}+3\lambda_{1}^2-3\lambda_{1}\lambda_{2},\;
\psi=c_{1}e^{\frac{\lambda_{2}}{2}t}+c_{2}te^{\frac{\lambda_{2}}{2}t};$\\
$(3)\;\alpha>\frac{\lambda_{2}^2\zeta^2}{4\eta}+3\lambda_{1}^2-3\lambda_{1}\lambda_{2},\\ \mbox{} \quad \;
\psi=c_{1}e^{\frac{\lambda_{2}}{2}t}cos(\frac{\sqrt{4(\alpha+3\lambda_{1}\lambda_{2}-3\lambda_{1}^2)
\frac{\eta}{\zeta^2}-\lambda_{2}^2}}{2}t)
+c_{2}e^{\frac{\lambda_{2}}{2}t}sin(\frac{\sqrt{4(\alpha+3\lambda_{1}\lambda_{2}-3\lambda_{1}^2)
\frac{\eta}{\zeta^2}-\lambda_{2}^2}}{2}t).$\\
We make equations $(41a)-(41c)$ into
$$\frac{\zeta^2}{\eta}\frac{\lambda_{2}\psi^\prime-\psi^{\prime\prime}}{\psi}
=\alpha+3\lambda_{1}\lambda_{2}-3\lambda_{1}^2;     \eqno{(44a)}$$
$$-\frac{p_{1}}{\zeta}\frac{(\phi^\zeta)^{\prime\prime}}{\phi^\zeta}+\frac{\lambda_{2}\zeta+(3\lambda_{1}-\lambda_{2})
p_{1}}{\zeta}\frac{(\phi^\zeta)^\prime}{\phi^\zeta}=\alpha-\lambda_{2}^2+3\lambda_{1}\lambda_{2};    \eqno{(44b)}$$
$$\frac{\alpha_{2}}{\phi^{2p_{2}}}-\frac{p_{2}}{\zeta}\frac{(\phi^\zeta)^{\prime\prime}}{\phi^\zeta}
+\frac{\lambda_{2}\zeta+(3\lambda_{1}-\lambda_{2})p_{2}}{\zeta}\frac{(\phi^\zeta)^\prime}{\phi^\zeta}
=\alpha-\lambda_{2}^2+3\lambda_{1}\lambda_{2}.  \eqno{(44c)}$$
When $p_{1}=p_{2},$ type (II) spaces turn into type (I) spaces, so we assume $p_{1}\neq p_{2}.$
By $(44b)\times p_{2}-(44c)\times p_{1},$ we get
$$\psi^\prime=\frac{p_{1}\alpha_{2}}{\lambda_{2}(p_{2}-p_{1})}\frac{\eta}{\zeta^2}\psi^{1-\frac{2p_{2}\zeta}{\eta}}
+\frac{\alpha-\lambda_{2}^2+3\lambda_{1}\lambda_{2}}{\lambda_{2}}\frac{\eta}{\zeta^2}\psi.  \eqno{(45)}$$
¢Ù$\;\alpha<\frac{\lambda_{2}^2\zeta^2}{4\eta}+3\lambda_{1}^2-3\lambda_{1}\lambda_{2},\;
\psi=c_{1}e^{at}+c_{2}e^{bt},$ where $a=\frac{\lambda_{2}+\sqrt{\lambda_{2}^2-4(\alpha+3\lambda_{1}\lambda_{2}-3\lambda_{1}^2)
\frac{\eta}{\zeta^2}}}{2},\;
b=\\ \mbox{} \quad \frac{\lambda_{2}-\sqrt{\lambda_{2}^2-4(\alpha+3\lambda_{1}\lambda_{2}-3\lambda_{1}^2)
\frac{\eta}{\zeta^2}}}{2}.$ By the equation $(45),$ we get
$$ac_{1}e^{at}+bc_{2}e^{bt}=
\frac{p_{1}\alpha_{2}}{\lambda_{2}(p_{2}-p_{1})}\frac{\eta}{\zeta^2}\big(c_{1}e^{at}+c_{2}e^{bt}\big)
^{1-\frac{2p_{2}\zeta}{\eta}}+\frac{\alpha-\lambda_{2}^2+3\lambda_{1}\lambda_{2}}{\lambda_{2}}\frac{\eta}{\zeta^2}
\big(c_{1}e^{at}+c_{2}e^{bt}\big).  \eqno{(46)}$$
$1)\;\underline{c_{1}=0.}$ We have
$$\Big[b-\frac{\alpha-\lambda_{2}^2+3\lambda_{1}\lambda_{2}}{\lambda_{2}}\frac{\eta}{\zeta^2}\Big]c_{2}e^{bt}
=\frac{p_{1}\alpha_{2}}{\lambda_{2}(p_{2}-p_{1})}\frac{\eta}{\zeta^2}\big(c_{2}e^{bt}\big)
^{1-\frac{2p_{2}\zeta}{\eta}}   \eqno{(47)}$$
$1^\prime.\;\underline{b\neq0,\;p_{1}\alpha_{2}\neq0.}$ By the equation $(47),$ we get $p_{2}=0,$ so $\zeta=p_{1},\;\eta=p_{1}^2$ and
$\frac{\zeta^2}{\eta}=1.$ Then $b=\frac{\lambda_{2}-\sqrt{\lambda_{2}^2-4(\alpha+3\lambda_{1}\lambda_{2}-3\lambda_{1}^2)}}{2},$
so $$b^2-\lambda_{2}b=3\lambda_{1}^2-3\lambda_{2}\lambda_{2}-\alpha.      \eqno{(48)}     $$
On the other hand, $\phi^\zeta=\psi=c_{2}e^{bt},$ so by equations $(44b),(44c),$ we get
$$-b^2+3\lambda_{1}b=\alpha-\lambda_{2}^2+3\lambda_{1}\lambda_{2};    \eqno{(49a)}$$
$$\alpha_{2}+\lambda_{2}b=\alpha-\lambda_{2}^2+3\lambda_{1}\lambda_{2};    \eqno{(49b)}$$
$(48)+(49a),$ we get $(3\lambda_{1}-\lambda_{2})b=3\lambda_{1}^2-\lambda_{2}^2,$ if $\lambda_{2}=3\lambda_{1},$
then $0=-6\lambda_{1}^2,$ this is a contradiction with $\lambda_{1}\neq0,$ so $\lambda_{2}\neq3\lambda_{1},\;
b=\frac{3\lambda_{1}^2-\lambda_{2}^2}{3\lambda_{1}-\lambda_{2}}.$ Then by the equation $(48),$ we have
$\alpha=\frac{18\lambda_{1}^4-36\lambda_{1}^3\lambda_{2}+24\lambda_{1}^2\lambda_{2}^2-6\lambda_{1}\lambda_{2}^3}
{(3\lambda_{1}-\lambda_{2})^2}
=\frac{6\lambda_{1}(\lambda_{1}-\lambda_{2})(3\lambda_{1}^2-3\lambda_{1}\lambda_{2}+\lambda_{2}^2)}
{(3\lambda_{1}-\lambda_{2})^2},$ by the equation $(49b),$ we get
$\alpha_{2}=\frac{18\lambda_{1}^4-18\lambda_{1}^3\lambda_{2}+6\lambda_{1}\lambda_{2}^3-2\lambda_{2}^4}
{(3\lambda_{1}-\lambda_{2})^2}=\frac{2(3\lambda_{1}^2-\lambda_{2}^2)(3\lambda_{1}^2-3\lambda_{1}\lambda_{2}
+\lambda_{2}^2)}{(3\lambda_{1}-\lambda_{2})^2},$ since $3\lambda_{1}^2-3\lambda_{1}\lambda_{2}+\lambda_{2}^2\neq0,$ and $\alpha_{2}\neq0,$ we have $3\lambda_{1}^2\neq\lambda_{2}^2.$ Considering that $b<\frac{\lambda_{2}}{2},$ we get
$\frac{6\lambda_{1}^2-3\lambda_{1}\lambda_{2}-\lambda_{2}^2}{3\lambda_{1}-\lambda_{2}}<0.$\\
So we obtain $\underline{3\lambda_{1}^2\neq\lambda_{2}^2,\;
\frac{6\lambda_{1}^2-3\lambda_{1}\lambda_{2}-\lambda_{2}^2}{3\lambda_{1}-\lambda_{2}}<0,\;p_{1}\neq0,\;p_{2}=0,\;
\alpha=\frac{18\lambda_{1}^4-36\lambda_{1}^3\lambda_{2}+24\lambda_{1}^2\lambda_{2}^2-6\lambda_{1}\lambda_{2}^3}
{(3\lambda_{1}-\lambda_{2})^2},}\\ \underline{\alpha_{1}=0,\;\alpha_{2}=\frac{18\lambda_{1}^4-18\lambda_{1}^3\lambda_{2}
+6\lambda_{1}\lambda_{2}^3-2\lambda_{2}^4}{(3\lambda_{1}-\lambda_{2})^2},\;
\phi=c_{0}e^{\frac{3\lambda_{1}^2-\lambda_{2}^2}{(3\lambda_{1}-\lambda_{2})p_{1}}t}}.$\\
$2^\prime.\;\underline{b\neq0,\;p_{1}\alpha_{2}=0.}$\\
\mbox{} \quad $1^\circ.\;\underline{p_{1}=0.}$ Then $\zeta=2p_{2},\;\eta=2p_{2}^2,\;\frac{\eta}{\zeta^2}=\frac{1}{2}.$ By the equation $(47),$ we get
$b=\frac{\alpha-\lambda_{2}^2+3\lambda_{1}\lambda_{2}}{2\lambda_{2}},$ on the other hand,
$b=\frac{\lambda_{2}-\sqrt{\lambda_{2}^2-2(\alpha+3\lambda_{1}\lambda_{2}-3\lambda_{1}^2)}}{2},$ so we get
$$\alpha^2+(6\lambda_{1}\lambda_{2}-2\lambda_{2}^2)\alpha+3\lambda_{1}^2\lambda_{2}^2-6\lambda_{1}^3\lambda_{2}
+3\lambda_{1}^4=0.  \eqno{(50)}$$
when $\lambda_{1}=\lambda_{2},$ we get $\alpha^2+4\lambda_{1}^2\alpha=0,$ then $\alpha=0$ or $\alpha=-4\lambda_{1}^2.$ If $\alpha=0,$ we have $b=\lambda_{1}=\lambda_{2}<0,$ by the equation $(44c),$ we get
$\alpha_{2}=0;\;\phi=\psi^{\frac{\zeta}{\eta}}=c_{0}e^{\frac{\lambda_{1}}{p_{2}}t}.$  If $\alpha=-4\lambda_{1}^2,$
then $\lambda_{1}=\lambda_{2}>0,\;b=-\lambda_{1},$ then by the equation $(44c),$ we get $\lambda_{1}=0,$ which is a contradiction. \\
When $\lambda_{1}\neq\lambda_{2},\;\Delta=8\lambda_{2}^2(3\lambda_{1}^2-\lambda_{2}^2),$
if $\lambda_{2}^2>3\lambda_{1}^2,$ we have no solution;
if $\lambda_{2}^2=3\lambda_{1}^2,$ we have $b=0,$ which is a contradiction;
if $\lambda_{2}^2<3\lambda_{1}^2,$ we have
$\alpha=\lambda_{2}^2-3\lambda_{1}\lambda_{2}+\sqrt{2\lambda_{2}^2(3\lambda_{1}^2-\lambda_{2}^2)}$
or $\alpha=\lambda_{2}^2-3\lambda_{1}\lambda_{2}-\sqrt{2\lambda_{2}^2(3\lambda_{1}^2-\lambda_{2}^2)}.$
By the equation $(44c),$ we have $\alpha_{2}=0$ and
$$-2b^2+(3\lambda_{1}+\lambda_{2})b=\alpha-\lambda_{2}^2+3\lambda_{1}\lambda_{2}.    \eqno{(51)}$$
When $\alpha=\lambda_{2}^2-3\lambda_{1}\lambda_{2}+\sqrt{2\lambda_{2}^2(3\lambda_{1}^2-\lambda_{2}^2)},$
we have $b=\frac{\sqrt{2\lambda_{2}^2(3\lambda_{1}^2-\lambda_{2}^2)}}{2\lambda_{2}},$ by the equation $(51),$ we have
$\lambda_{1}=\lambda_{2},$ which is a contradiction. When $\alpha=\lambda_{2}^2-3\lambda_{1}\lambda_{2}-\sqrt{2\lambda_{2}^2(3\lambda_{1}^2-\lambda_{2}^2)},$ we have the same
contradiction.\\
So we get $\underline{\lambda_{1}=\lambda_{2}<0,\;p_{1}=0,\;p_{2}\neq0,\;\alpha=\alpha_{1}=\alpha_{2}=0,\;
\phi=c_{0}e^{\frac{\lambda_{1}}{p_{2}}t}.}$ \\
$2^\circ.\;\underline{\alpha_{2}=0.}$ Since $p_{1}=0$ we have discussed, so we assume that $p_{1}\neq0.$
By the equation $(47),$ we have $b=\frac{\alpha-\lambda_{2}^2+3\lambda_{1}\lambda_{2}}{\lambda_{2}}\frac{\eta}{\zeta^2},
\;\phi^\zeta=ce^{b\frac{\zeta^2}{\eta}t},$ then by the equation$(44b),$ we have
$\alpha(\alpha-\lambda_{2}^2+3\lambda_{1}\lambda_{2})=0.$ So when $\lambda_{2}=3\lambda_{1},$ we have $\alpha=0;$
when $\lambda_{2}\neq3\lambda_{1},$ we have $\alpha=0$ or $\alpha=\lambda_{2}^2-3\lambda_{1}\lambda_{2}\neq0.$\par
When $\lambda_{2}=3\lambda_{1},\;\alpha=0,$ then $b=0,$ which is a contradiction;\par
when $\lambda_{2}\neq3\lambda_{1},\;\alpha=\lambda_{2}^2-3\lambda_{1}\lambda_{2}\neq0,$ then $b=0,$ which is a contradiction;\par
when $\lambda_{2}\neq3\lambda_{1},\;\alpha=0,$ then $b=(3\lambda_{1}-\lambda_{2})\frac{\eta}{\zeta^2},$ by the equation
$(44a),$ we get $\frac{\eta}{\zeta^2}=\frac{3\lambda_{1}^2-\lambda_{2}^2}{(3\lambda_{1}-\lambda_{2})^2},$ then
$b=\frac{3\lambda_{1}^2-\lambda_{2}^2}{3\lambda_{1}-\lambda_{2}}.$ Since $b<\frac{\lambda_{2}}{2},$ we get
$\frac{6\lambda_{1}^2-3\lambda_{1}\lambda_{2}-\lambda_{2}^2}{3\lambda_{1}-\lambda_{2}}<0,$ by
$\alpha<\frac{\lambda_{2}^2\zeta^2}{4\eta}+3\lambda_{1}^2-3\lambda_{1}\lambda_{2},$ we get $\lambda_{2}^2<3\lambda_{1}^2.$ $\phi=\psi^{\frac{\zeta}{\eta}}=ce^{\frac{3\lambda_{1}-\lambda_{2}}{\zeta}t}.$\\
So we have $\underline{\lambda_{2}^2<3\lambda_{1}^2,\;
\frac{6\lambda_{1}^2-3\lambda_{1}\lambda_{2}-\lambda_{2}^2}{3\lambda_{1}-\lambda_{2}}<0,\;
p_{1}\neq0,\;p_{2}\neq-\frac{1}{2}p_{1},\;\alpha=\alpha_{1}=\alpha_{2}=0,}\\ \underline{
\phi=ce^{\frac{3\lambda_{1}-\lambda_{2}}{\zeta}t}.}$\\
$3^\prime.\;\underline{b=0.}$ Then $\psi=c_{2},$ by the equation $(44a),$ we have $\alpha=3\lambda_{1}^2-3\lambda_{1}\lambda_{2};$ by the equation $(44b),$ we have
$\alpha=\lambda_{2}^2-3\lambda_{1}\lambda_{2};$ so $\lambda_{2}^2=3\lambda_{1}^2,$ by the equation $(44c),$ we have
$\alpha_{2}=0.$ \par So we have
$\underline{\lambda_{2}^2=3\lambda_{1}^2,\;\zeta\neq0,\;\eta\neq0,\;
\alpha=\lambda_{2}^2-3\lambda_{1}\lambda_{2}=3\lambda_{1}^2-3\lambda_{1}\lambda_{2},\;
\alpha_{1}=\alpha_{2}=0,}\\ \underline{\phi=c.}$\\
$2)\;\underline{c_{2}=0.}$ We have
$$\Big[a-\frac{\alpha-\lambda_{2}^2+3\lambda_{1}\lambda_{2}}{\lambda_{2}}\frac{\eta}{\zeta^2}\Big]c_{1}e^{at}
=\frac{p_{1}\alpha_{2}}{\lambda_{2}(p_{2}-p_{1})}\frac{\eta}{\zeta^2}\big(c_{1}e^{at}\big)
^{1-\frac{2p_{2}\zeta}{\eta}}   \eqno{(52)}$$
$1^\prime.\;\underline{p_{1}\alpha_{2}\neq0.}$ By the equation $(52),$ we get $p_{2}=0,$ so
$\zeta=p_{1},\;\eta=p_{1}^2$ and
$\frac{\zeta^2}{\eta}=1.$ Then $a=\frac{\lambda_{2}+\sqrt{\lambda_{2}^2-4(\alpha+3\lambda_{1}\lambda_{2}-3\lambda_{1}^2)}}{2},$
using the equation $(52)$ again, we get $a-\frac{\alpha-\lambda_{2}^2+3\lambda_{1}\lambda_{2}}{\lambda_{2}}
=\frac{p_{1}\alpha_{2}}{-\lambda_{2}p_{1}}=-\frac{\alpha_{2}}{\lambda_{2}},$ so
$\alpha_{2}=\alpha-\lambda_{2}^2+3\lambda_{1}\lambda_{2}-\lambda_{2}a.$ By the equation $(44b),$ we get
$(3\lambda_{1}-\lambda_{2})^2\alpha=18\lambda_{1}^4-6\lambda_{1}\lambda_{2}^3+24\lambda_{1}^2\lambda_{2}^2
-36\lambda_{1}^3\lambda_{2},$ if $3\lambda_{1}=\lambda_{2},$ then $0=-36\lambda_{1}^2,$ which is a contradiction,
so $3\lambda_{1}\neq\lambda_{2}$ and $\alpha=\frac{18\lambda_{1}^4-6\lambda_{1}\lambda_{2}^3+24\lambda_{1}^2\lambda_{2}^2
-36\lambda_{1}^3\lambda_{2}}{(3\lambda_{1}-\lambda_{2})^2},$ so $a=\frac{\lambda_{2}
+\sqrt{\frac{(6\lambda_{1}^2-3\lambda_{1}\lambda_{2}-\lambda_{2}^2)^2}{(3\lambda_{1}-\lambda_{2})^2}}}{2}.$\par
If $\frac{6\lambda_{1}^2-3\lambda_{1}\lambda_{2}-\lambda_{2}^2}{3\lambda_{1}-\lambda_{2}}>0,$ then
$a=\frac{3\lambda_{1}^2-\lambda_{2}^2}{3\lambda_{1}-\lambda_{2}},\;
\alpha_{2}=\frac{18\lambda_{1}^4-2\lambda_{2}^4+6\lambda_{1}\lambda_{2}^3-18\lambda_{1}^3\lambda_{2}}
{(3\lambda_{1}-\lambda_{2})^2}=\\ \frac{(3\lambda_{1}^2-\lambda_{2}^2)(3\lambda_{1}^2-3\lambda_{1}\lambda_{2}
+\lambda_{2}^2)}{(3\lambda_{1}-\lambda_{2})^2},$ since $3\lambda_{1}^2-3\lambda_{1}\lambda_{2}+\lambda_{2}^2
=3\big(\lambda_{1}-\frac{1}{2}\lambda_{2}\big)^2+\frac{1}{4}\lambda_{2}^2\neq0,$ and $\alpha_{2}\neq0,$ we have $3\lambda_{1}^2\neq\lambda_{2}^2,$ so $a\neq0.$\par
If $\frac{6\lambda_{1}^2-3\lambda_{1}\lambda_{2}-\lambda_{2}^2}{3\lambda_{1}-\lambda_{2}}<0,$ then
$a=\frac{3\lambda_{1}\lambda_{2}-\lambda_{2}^2}{3\lambda_{1}-\lambda_{2}},\;
\alpha_{2}=\frac{18\lambda_{1}^4-\lambda_{2}^4+6\lambda_{1}\lambda_{2}^3-15\lambda_{1}^2\lambda_{2}^2}
{(3\lambda_{1}-\lambda_{2})^2}\neq0,$ so $18\lambda_{1}^4-\lambda_{2}^4+6\lambda_{1}\lambda_{2}^3-15\lambda_{1}^2\lambda_{2}^2\neq0.$\par
Hence, we have $\underline{\lambda_{2}\neq3\lambda_{1},\;\lambda_{2}^2\neq3\lambda_{1}^2,\;
\frac{6\lambda_{1}^2-3\lambda_{1}\lambda_{2}-\lambda_{2}^2}{3\lambda_{1}-\lambda_{2}}>0,\;p_{1}\neq0,\;p_{2}=0,\;
\alpha_{1}=0,}$\\
$\underline{\alpha=\frac{18\lambda_{1}^4-6\lambda_{1}\lambda_{2}^3+24\lambda_{1}^2\lambda_{2}^2
-36\lambda_{1}^3\lambda_{2}}{(3\lambda_{1}-\lambda_{2})^2},\;\alpha_{2}=\frac{18\lambda_{1}^4-2\lambda_{2}^4
+6\lambda_{1}\lambda_{2}^3-18\lambda_{1}^3\lambda_{2}}{(3\lambda_{1}-\lambda_{2})^2},\;
\phi=ce^{\frac{3\lambda_{1}^2-\lambda_{2}^2}{(3\lambda_{1}-\lambda_{2})p_{1}}t}.}$\\
or $\underline{\lambda_{2}\neq3\lambda_{1},\;
18\lambda_{1}^4-\lambda_{2}^4+6\lambda_{1}\lambda_{2}^3-15\lambda_{1}^2\lambda_{2}^2\neq0,\;
\frac{6\lambda_{1}^2-3\lambda_{1}\lambda_{2}-\lambda_{2}^2}{3\lambda_{1}-\lambda_{2}}<0,\;p_{1}\neq0,\;p_{2}=0,\;}$
\\$\underline{\alpha_{1}=0,\alpha=\frac{18\lambda_{1}^4-6\lambda_{1}\lambda_{2}^3+24\lambda_{1}^2\lambda_{2}^2
-36\lambda_{1}^3\lambda_{2}}{(3\lambda_{1}-\lambda_{2})^2},\;\alpha_{2}=\frac{18\lambda_{1}^4-\lambda_{2}^4
+6\lambda_{1}\lambda_{2}^3-15\lambda_{1}^2\lambda_{2}^2}{(3\lambda_{1}-\lambda_{2})^2},\;
\phi=ce^{\frac{3\lambda_{1}\lambda_{2}-3\lambda_{1}^2}{(3\lambda_{1}-\lambda_{2})p_{1}}t}.}$\\
Especially when $\lambda_{1}=\lambda_{2}<0,$ we have $p_{1}\neq0,\;p_{2}=0,\;\alpha=\alpha_{1}=0,\;
\alpha_{2}=2\lambda_{1}^2,\;\phi=c.$\\
$2^\prime.\;\underline{p_{1}\alpha_{2}=0.}$ $1^\circ.\;\underline{p_{1}=0.}$ Then $\zeta=2p_{2},\;\eta=2p_{2}^2,\;\frac{\eta}{\zeta^2}=\frac{1}{2}.$ By the equation $(52),$ we get
$a=\frac{\alpha-\lambda_{2}^2+3\lambda_{1}\lambda_{2}}{2\lambda_{2}},$ on the other hand,
$a=\frac{\lambda_{2}+\sqrt{\lambda_{2}^2-2(\alpha+3\lambda_{1}\lambda_{2}-3\lambda_{1}^2)}}{2},$ so we get
$$\alpha^2+(6\lambda_{1}\lambda_{2}-2\lambda_{2}^2)\alpha+3\lambda_{1}^2\lambda_{2}^2-6\lambda_{1}^3\lambda_{2}
+3\lambda_{1}^4=0.  \eqno{(53)}$$
when $\lambda_{1}=\lambda_{2},$ we get $\alpha^2+4\lambda_{1}^2\alpha=0,$ then $\alpha=0$ or $\alpha=-4\lambda_{1}^2.$ If $\alpha=0,$ we have $a=\lambda_{1}=\lambda_{2}>0,$ by the equation $(44c),$ we get
$\alpha_{2}=0,\;\phi=\psi^{\frac{\zeta}{\eta}}=c_{0}e^{\frac{\lambda_{1}}{p_{2}}t}.$  If $\alpha=-4\lambda_{1}^2,$
then $\lambda_{1}=\lambda_{2}<0,\;a=-\lambda_{1},$ then by the equation $(44c),$ we get $\lambda_{1}=0,$ which is a contradiction. \\
When $\lambda_{1}\neq\lambda_{2},\;\Delta=8\lambda_{2}^2(3\lambda_{1}^2-\lambda_{2}^2),$
if $\lambda_{2}^2>3\lambda_{1}^2,$ we have no solution;
if $\lambda_{2}^2=3\lambda_{1}^2,$ we have $\underline{a=0,\;\alpha_{1}=\alpha_{2}=0,\;\alpha=\lambda_{2}^2-3\lambda_{1}\lambda_{2},\;\phi=c;}$
if $\lambda_{2}^2<3\lambda_{1}^2,$ we have
$\alpha=\lambda_{2}^2-3\lambda_{1}\lambda_{2}+\sqrt{2\lambda_{2}^2(3\lambda_{1}^2-\lambda_{2}^2)}$
or $\alpha=\lambda_{2}^2-3\lambda_{1}\lambda_{2}-\sqrt{2\lambda_{2}^2(3\lambda_{1}^2-\lambda_{2}^2)}.$
By the equation $(44c),$ we have $\alpha_{2}=0$ and
$$-2a^2+(3\lambda_{1}+\lambda_{2})a=\alpha-\lambda_{2}^2+3\lambda_{1}\lambda_{2}.    \eqno{(54)}$$
then we have $\lambda_{1}=\lambda_{2}>0$ or $\lambda_{1}=\lambda_{2}<0,$ but when $\lambda_{1}=\lambda_{2}<0,$ it is
not satisfies the equation $(44c),$ so we get $\underline{\lambda_{1}=\lambda_{2}>0,\;p_{1}=0,\;p_{2}\neq0,\;
\alpha=\alpha_{1}=\alpha_{2}=0,\;\phi=ce^{\frac{\lambda_{1}}{p_{2}}t}}.$\\
$2^\circ.\;\underline{\alpha_{2}=0.}$ Since $p_{1}=0$ we have discussed, so we assume that $p_{1}\neq0.$
By the equation $(52),$ we have $a=\frac{\alpha-\lambda_{2}^2+3\lambda_{1}\lambda_{2}}{\lambda_{2}}\frac{\eta}{\zeta^2},
\;\phi^\zeta=ce^{a\frac{\zeta^2}{\eta}t},$ then by the equation$(44b),$ we have
$\alpha(\alpha-\lambda_{2}^2+3\lambda_{1}\lambda_{2})=0.$ So when $\lambda_{2}=3\lambda_{1},$ we have $\alpha=0;$
when $\lambda_{2}\neq3\lambda_{1},$ we have $\alpha=0$ or $\alpha=\lambda_{2}^2-3\lambda_{1}\lambda_{2}\neq0.$\par
When $\lambda_{2}=3\lambda_{1},\;\alpha=0,$ then $a=0,\;\eta=0,$ which is a contradiction.\par
When $\lambda_{2}\neq3\lambda_{1},\;\alpha=\lambda_{2}^2-3\lambda_{1}\lambda_{2}\neq0,$ then $a=0$ and
$(\lambda_{2}^2-3\lambda_{1}^2)\eta=0.$ \\
If $\lambda_{2}^2\neq3\lambda_{1}^2,$ then $\eta=0,$ which is a contradiction;\\
If $\lambda_{2}^2=3\lambda_{1}^2,$ which satisfies $\lambda_{2}\neq3\lambda_{1},$ then
$\phi^\zeta=c,$ which satisfies equations $(44a)-(44c).$\\
so we have $\underline{\lambda_{2}^2=3\lambda_{1}^2,\;p_{1}\neq0,\;p_{2}\neq-\frac{1}{2}p_{1}.\;
\alpha=\lambda_{2}^2-3\lambda_{1}\lambda_{2},\;\alpha_{1}=\alpha_{2}=0,\;\phi=c.}$\par
When $\lambda_{2}\neq3\lambda_{1},\;\alpha=0,$ then $a=(3\lambda_{1}-\lambda_{2})\frac{\eta}{\zeta^2},$ by the equation $(44a),$ we get $\frac{\eta}{\zeta^2}=\frac{3\lambda_{1}^2-\lambda_{2}^2}{(3\lambda_{1}-\lambda_{2})^2}>0$ and $\lambda_{2}^2<3\lambda_{1}^2,$ then
$a=\frac{3\lambda_{1}^2-\lambda_{2}^2}{3\lambda_{1}-\lambda_{2}}.$ Since $a>\frac{\lambda_{2}}{2},$ we get
$\frac{6\lambda_{1}^2-3\lambda_{1}\lambda_{2}-\lambda_{2}^2}{3\lambda_{1}-\lambda_{2}}>0,$
considering that$\frac{p_{1}^2+2p_{2}^2}{(p_{1}+2p_{2})^2}=\frac{\eta}{\zeta^2}=\frac{3\lambda_{1}^2
-\lambda_{2}^2}{(3\lambda_{1}-\lambda_{2})^2},$ we have
$(3\lambda_{1}^2-3\lambda_{1}\lambda_{2}+\lambda_{2}^2)p_{1}^2-(6\lambda_{1}^2-2\lambda_{2}^2)p_{1}p_{2}
=(3\lambda_{1}^2-6\lambda_{1}\lambda_{2}+3\lambda_{2}^2)p_{2}^2,$ no matter $\lambda_{1}=\lambda_{2}$ or
$\lambda_{1}\neq\lambda_{2},$ we can get $p_{2}\neq0.$\\
So we get $\underline{\lambda_{2}\neq3\lambda_{1},\;\lambda_{2}^2<3\lambda_{1}^2,\;
\frac{6\lambda_{1}^2-3\lambda_{1}\lambda_{2}-\lambda_{2}^2}{3\lambda_{1}-\lambda_{2}}>0,\;
\frac{\eta}{\zeta^2}=\frac{3\lambda_{1}^2-\lambda_{2}^2}{(3\lambda_{1}-\lambda_{2})^2},\;
\alpha=\alpha_{1}=\alpha_{2}=0,}\\ \underline{\phi=ce^{\frac{3\lambda_{1}^2-\lambda_{2}^2}{3\lambda_{1}-\lambda_{2}}
\frac{\zeta}{\eta}t}.}$\\
$3)\;\underline{c_{1}\neq0,\;c_{2}\neq0,\;b\neq0.}$\\
$1^\prime.\;\underline{p_{2}\neq0.}$ Then $e^{at},e^{bt}$ and $(c_{1}e^{at}+c_{2}e^{bt})^{1-2p_{2}\frac{\zeta}{\eta}}$ are linearly independent, by the equation $(46),$ we have
$\Big[a-\frac{\alpha-\lambda_{2}^2+3\lambda_{1}\lambda_{2}}{\lambda_{2}}\frac{\eta}{\zeta^2}\Big]c_{1}=0,\;\;\;
\Big[b-\frac{\alpha-\lambda_{2}^2+3\lambda_{1}\lambda_{2}}{\lambda_{2}}\frac{\eta}{\zeta^2}\Big]c_{2}=0.$\\
Considering that $c_{1}\neq0,\;c_{2}\neq0,$ we have $a=b=\frac{\alpha-\lambda_{2}^2+3\lambda_{1}\lambda_{2}}{\lambda_{2}}\frac{\eta}{\zeta^2},$ which is a contradiction.\\
$2^\prime.\;\underline{p_{2}=0.}$ Then by the equation $(46),$ we have \\
$a-\frac{\alpha-\lambda_{2}^2+3\lambda_{1}\lambda_{2}}{\lambda_{2}}\frac{\eta}{\zeta^2}
-\frac{p_{1}\alpha_{2}}{\lambda_{2}(p_{2}-p_{1})}\frac{\eta}{\zeta^2}=0,\qquad
b-\frac{\alpha-\lambda_{2}^2+3\lambda_{1}\lambda_{2}}{\lambda_{2}}\frac{\eta}{\zeta^2}
-\frac{p_{1}\alpha_{2}}{\lambda_{2}(p_{2}-p_{1})}\frac{\eta}{\zeta^2}=0.$\\
Then $a=b,$ which is a contradiction.\\
$4)\;\underline{c_{1}\neq0,\;c_{2}\neq0,\;b=0.}$ Then $a=\lambda_{2}\neq0$ and
$$ac_{1}e^{at}
=\frac{p_{1}\alpha_{2}}{\lambda_{2}(p_{2}-p_{1})}\frac{\eta}{\zeta^2}(c_{1}e^{at}+c_{2})^{1-2p_{2}\frac{\zeta}{\eta}}
+\frac{\alpha-\lambda_{2}^2+3\lambda_{1}\lambda_{2}}{\lambda_{2}}\frac{\eta}{\zeta^2}(c_{1}e^{at}+c_{2}). \eqno{(55)}$$
$1^\prime.\;\underline{1-2p_{2}\frac{\zeta}{\eta}\neq0.}$ \\If $p_{2}\neq0,$ then $e^{at}$ and $(c_{1}e^{at}+c_{2})^{1-2p_{2}\frac{\zeta}{\eta}}$ are linearly independent, by the equation $(55),$ we have
$a=0,$ this is a contradiction with $a=\lambda_{2}\neq0.$\\
If $p_{2}=0,$ using the same method we can get $a=0,$ this is a contradiction with $a=\lambda_{2}\neq0.$\\
$2^\prime.\;\underline{1-2p_{2}\frac{\zeta}{\eta}=0.}$ \\
Then $p_{2}\neq0,\;\eta=2p_{2}\zeta$ and the equation $(55)$ becomes
$$\Big(\lambda_{2}-\frac{\alpha-\lambda_{2}^2+3\lambda_{1}\lambda_{2}}{\lambda_{2}}\frac{\eta}{\zeta^2}\Big)c_{1}e^{at}
-\frac{p_{1}\alpha_{2}}{\lambda_{2}(p_{2}-p_{1})}\frac{\eta}{\zeta^2}
-\frac{\alpha-\lambda_{2}^2+3\lambda_{1}\lambda_{2}}{\lambda_{2}}\frac{\eta}{\zeta^2}=0.    \eqno{(56)}$$
Then $$\lambda_{2}=\frac{\alpha-\lambda_{2}^2+3\lambda_{1}\lambda_{2}}{\lambda_{2}}\frac{\eta}{\zeta^2}.\eqno{(57)}$$
Since $b=0$ which means $\lambda_{2}=\sqrt{\lambda_{2}^2-4(\alpha+3\lambda_{1}\lambda_{2}-3\lambda_{1}^2)\frac{\eta}{\zeta^2}},$ we get
$\alpha=3\lambda_{1}^2-3\lambda_{1}\lambda_{2},$ then using the equation $(57),$ we have
$\lambda_{2}=\frac{3\lambda_{1}^2-\lambda_{2}^2}{\lambda_{2}}\frac{\eta}{\zeta^2}$ and $\lambda_{2}^2\neq3\lambda_{1}^2,$ considering that
$\eta=2p_{2}\zeta,$ we get $p_{1}=\frac{6\lambda_{1}^2-4\lambda_{2}^2}{\lambda_{2}^2}p_{2}.$\\
If $3\lambda_{1}^2=2\lambda_{2}^2,$ then $p_{1}=0,$ by $\eta=2p_{2}\zeta$ we have $p_{2}=0,$ which is a contradiction.\\
If $3\lambda_{1}^2\neq2\lambda_{2}^2,$ then $p_{1}\neq0,$ and by $\eta=2p_{2}\zeta$ we have $18\lambda_{1}^4-30\lambda_{1}^2\lambda_{2}^2+11\lambda_{2}^4=0,$ then
$\lambda_{1}^2=\frac{5\pm\sqrt{3}}{6}\lambda_{2}^2,$ which satisfies $\lambda_{2}^2\neq3\lambda_{1}^2$ and
$3\lambda_{1}^2\neq2\lambda_{2}^2.$ Then $\alpha_{2}=\frac{(p_{2}-p_{1})(\lambda_{2}^2-3\lambda_{1}\lambda_{2})}{p_{1}},\;p_{1}=(1\pm\sqrt{3})p_{2}\neq0,\;
\zeta=(3\pm\sqrt{3})p_{2},\;\eta=(6\pm2\sqrt{3})p_{2}^2,\;\phi=ce^{\frac{\lambda_{2}}{2p_{1}}t}.$\par
So we get $\underline{\lambda_{1}^2=\frac{5\pm\sqrt{3}}{6}\lambda_{2}^2,\;p_{1}=(1\pm\sqrt{3})p_{2}\neq0,\;
\alpha=3\lambda_{1}^2-3\lambda_{1}\lambda_{2},\;\alpha_{1}=0,\;}$ \\ $\underline{
\alpha_{2}=\frac{(p_{2}-p_{1})(\lambda_{2}^2-3\lambda_{1}\lambda_{2})}{p_{1}},\;\phi=ce^{\frac{\lambda_{2}}{2p_{1}}t}.}$\\
¢Ú$\;\alpha=\frac{\lambda_{2}^2\zeta^2}{4\eta}+3\lambda_{1}^2-3\lambda_{1}\lambda_{2},\;
\psi=c_{1}e^{\frac{\lambda_{2}}{2}t}+c_{2}te^{\frac{\lambda_{2}}{2}t}.$ By the equation $(45),$ we have
$$\Big[\frac{\lambda_{2}}{2}c_{1}+c_{2}-a_{0}c_{1}+\Big(\frac{\lambda_{2}}{2}c_{2}
-a_{0}c_{2}\Big)t\Big]e^{\frac{\lambda_{2}}{2}t}=\frac{p_{1}\alpha_{2}}{\lambda_{2}(p_{2}-p_{1})}\frac{\eta}{\zeta^2}
(c_{1}+c_{2}t)^{1-2p_{2}\frac{\zeta}{\eta}}(e^{\frac{\lambda_{2}}{2}t})^{1-2p_{2}\frac{\zeta}{\eta}},  \eqno{(58)}$$
where $a_{0}=\frac{\lambda_{2}}{4}+\frac{3\lambda_{1}^2-\lambda_{2}^2}{\lambda_{2}^2}\frac{\eta}{\zeta^2}.$\\
$1)\;\underline{c_{2}\neq0.}$ Then by the equation $(58),$ we have $p_{2}=0,$ by the equation $(44c),$ we get
$\alpha_{2}+\lambda_{2}\frac{(\phi^\zeta)^\prime}{\phi^\zeta}=\alpha-\lambda_{2}^2+3\lambda_{1}\lambda_{2},$ then
$\phi^\zeta=c_{0}e^{\frac{\alpha-\lambda_{2}^2+3\lambda_{1}\lambda_{2}-\alpha_{2}}{\lambda_{2}}t};$ on the other hand, $\phi^\zeta=\psi^{\frac{\zeta^2}{\eta}}=\Big(c_{1}e^{\frac{\lambda_{2}}{2}t}
+c_{2}te^{\frac{\lambda_{2}}{2}t}\Big)^{\frac{\zeta^2}{\eta}},$ then
$c_{0}e^{\frac{\alpha-\lambda_{2}^2+3\lambda_{1}\lambda_{2}-\alpha_{2}}{\lambda_{2}}t}
=\Big(c_{1}e^{\frac{\lambda_{2}}{2}t}+c_{2}te^{\frac{\lambda_{2}}{2}t}\Big)^{\frac{\zeta^2}{\eta}},$ this is a contradiction with $c_{2}\neq0.$ \\
$2)\;\underline{c_{2}=0.}$ Then the equation $(58)$ becomes
$$\Big(\frac{\lambda_{2}}{2}c_{1}-a_{0}c_{1}\Big)e^{\frac{\lambda_{2}}{2}t}
=\frac{p_{1}\alpha_{2}}{\lambda_{2}(p_{2}-p_{1})}\frac{\eta}{\zeta^2}
c_{1}^{1-2p_{2}\frac{\zeta}{\eta}}(e^{\frac{\lambda_{2}}{2}t})^{1-2p_{2}\frac{\zeta}{\eta}}.  \eqno{(59)}$$
$1^\prime.\;\underline{a_{0}=\frac{\lambda_{2}}{2}.}$ By the equation $(59),$ we get $p_{1}\alpha_{2}=0.$ Considering that $a_{0}=\frac{\lambda_{2}}{4}+\frac{3\lambda_{1}^2-\lambda_{2}^2}{\lambda_{2}^2}\frac{\eta}{\zeta^2},$ we have
$\lambda_{2}^2\neq3\lambda_{1}^2$ and $\frac{\lambda_{2}^2}{4}\frac{\zeta^2}{\eta}=3\lambda_{1}^2-\lambda_{2}^2,$
then $\alpha=6\lambda_{1}^2-3\lambda_{1}\lambda_{2}-\lambda_{2}^2.$\\
If $p_{1}=0,$ then $\zeta=2p_{2},\;\eta=2p_{2}^2,\;\frac{\eta}{\zeta^2}=\frac{1}{2}.$ By
$\frac{\lambda_{2}^2}{4}\frac{\zeta^2}{\eta}=3\lambda_{1}^2-\lambda_{2}^2,$ we have $\lambda_{2}^2=2\lambda_{1}^2$ and $\alpha=2\lambda_{2}^2-3\lambda_{1}\lambda_{2}.$ By the equation $(44b),$ we get $\phi^\zeta=c_{0}e^{\lambda_{2}t},$ then by the equation $(44c),$ we get $\lambda_{2}=\frac{3}{2}\lambda_{1},$ this is a contradiction with $\lambda_{2}^2=2\lambda_{1}^2.$\\
If $p_{1}\neq0,\;\alpha_{2}=0,$ then $\psi=c_{1}e^{\frac{\lambda_{2}}{2}t},\;
\phi^\zeta=c_{0}e^{\frac{2(3\lambda_{1}^2-\lambda_{2}^2)}{\lambda_{2}}t},$ then by the equation $(44b),$ we have
$\alpha=6\lambda_{1}^2-3\lambda_{1}\lambda_{2}-\lambda_{2}^2=0,$ which satisfies the equation $(44c).\;
\phi=c_{0}e^{\frac{\lambda_{2}}{2}\frac{\zeta}{\eta}t}.$\\
So we get $\underline{p_{1}\neq0,\;\alpha=6\lambda_{1}^2-3\lambda_{1}\lambda_{2}-\lambda_{2}^2=0,\;
\alpha_{1}=\alpha_{2}=0,\;\phi=c_{0}e^{\frac{\lambda_{2}}{2}\frac{\zeta}{\eta}t}.}$\\
$2^\prime.\;\underline{a_{0}\neq\frac{\lambda_{2}}{2}.}$ By the equation $(59),$ we have $p_{2}=0,$ then
$\zeta=p_{1},\;\eta=p_{1}^2,\;\frac{\zeta^2}{\eta}=1,$ so $\alpha=\frac{\lambda_{2}^2}{4}+3\lambda_{1}^2-3\lambda_{1}\lambda_{2},\;\phi^\zeta=c_{1}e^{\frac{\lambda_{2}}{2}t}.$
By equation $(44b),$ we have $6\lambda_{1}^2-3\lambda_{1}\lambda_{2}-\lambda_{2}^2=0,$ by the equation $(44c),$ we have $\alpha_{2}=3\lambda_{1}^2-\frac{5}{4}\lambda_{2}^2.$\\
So we get $\underline{6\lambda_{1}^2-3\lambda_{1}\lambda_{2}-\lambda_{2}^2=0,\;p_{1}\neq0,\;p_{2}=0,\;
\alpha=\frac{\lambda_{2}^2}{4}+3\lambda_{1}^2-3\lambda_{1}\lambda_{2},\;
\alpha_{2}=3\lambda_{1}^2-\frac{5}{4}\lambda_{2}^2,}\\ \underline{\phi=ce^{\frac{\lambda_{2}}{2\zeta}t}.}$\\
¢Û$\alpha>\frac{\lambda_{2}^2\zeta^2}{4\eta}+3\lambda_{1}^2-3\lambda_{1}\lambda_{2},\;
\psi=c_{1}e^{\frac{\lambda_{2}}{2}t}cos(at)+c_{2}e^{\frac{\lambda_{2}}{2}t}sin(at),\;
a=\frac{\sqrt{4(\alpha+3\lambda_{1}\lambda_{2}-3\lambda_{1}^2)\frac{\eta}{\zeta^2}-\lambda_{2}^2}}{2}.$
By the equation $(45),$ we have
\begin{eqnarray}
& &\Big(\frac{\lambda_{2}}{2}c_{1}+ac_{2}\Big)cos(at)+\Big(-ac_{1}+\frac{\lambda_{2}}{2}c_{2}\Big)sin(at) \nonumber\\
&=&\frac{p_{1}\alpha_{2}}{\lambda_{2}(p_{2}-p_{1})}\frac{\eta}{\zeta^2}
\big(c_{1}cos(at)+c_{2}sin(at)\big)^{1-2p_{2}\frac{\zeta}{\eta}}e^{-\lambda_{2}p_{2}\frac{\zeta}{\eta}t}
  \setcounter{equation}{60}\\
&+&\frac{\alpha-\lambda_{2}^2
+3\lambda_{1}\lambda_{2}}{\lambda_{2}}\frac{\eta}{\zeta^2}\big(c_{1}cos(at)+c_{2}sin(at)\big).\nonumber
\end{eqnarray}
$1)\;\underline{p_{2}\neq0.}$ Then by the equation $(60),$ we get $p_{1}\alpha_{2}=0$ and
$$\frac{\lambda_{2}}{2}c_{1}+ac_{2}
=\frac{\alpha-\lambda_{2}^2+3\lambda_{1}\lambda_{2}}{\lambda_{2}}\frac{\eta}{\zeta^2}c_{1};  \eqno{(61a)}  $$
$$-ac_{1}+\frac{\lambda_{2}}{2}c_{2}
=\frac{\alpha-\lambda_{2}^2+3\lambda_{1}\lambda_{2}}{\lambda_{2}}\frac{\eta}{\zeta^2}c_{2}.  \eqno{(61b)}  $$
$(61a)\times c_{2}-(61b)\times c_{1},$ we get $c_{1}^2+c_{2}^2=0,$ this is a contradiction.\\
$2)\;\underline{p_{2}=0.}$ Then by the equation $(60),$ we have
$$\frac{\lambda_{2}}{2}c_{1}+ac_{2}=\frac{p_{1}\alpha_{2}}{\lambda_{2}(p_{2}-p_{1})}\frac{\eta}{\zeta^2}c_{1}
+\frac{\alpha-\lambda_{2}^2+3\lambda_{1}\lambda_{2}}{\lambda_{2}}\frac{\eta}{\zeta^2}c_{1};  \eqno{(62a)}  $$
$$-ac_{1}+\frac{\lambda_{2}}{2}c_{2}=\frac{p_{1}\alpha_{2}}{\lambda_{2}(p_{2}-p_{1})}\frac{\eta}{\zeta^2}c_{2}
+\frac{\alpha-\lambda_{2}^2+3\lambda_{1}\lambda_{2}}{\lambda_{2}}\frac{\eta}{\zeta^2}c_{2}.  \eqno{(62b)}  $$
$(62a)\times c_{2}-(62b)\times c_{1},$ we get $c_{1}^2+c_{2}^2=0,$ this is a contradiction.\\
So we have no solution in ¢Û.

According to above discussions, we get the following Theorem:
\begin{Theorem}
Let $M=I \times_{\phi^{p_{1}}}F_{1} \times_{\phi^{p_{2}}}F_{2}\cdots \times_{\phi^{p_{m}}}F_{m}$ be a generalized Kasner space-time, ${\rm dim}F_{1}=1,\;{\rm dim}F_{2}=2$ and $P=\frac{\partial}{\partial t}.$ Then $(M,\overline\nabla)$ is Einstein with the Einstein constant $\alpha$ if and only if one of the following conditions is satisfied:\\
$(1)\lambda_{2}^2=3\lambda_{1}^2,\;p_{1}=p_{2}=0,\;
\alpha=3\lambda_{1}^2-3\lambda_{1}\lambda_{2}=\lambda_{2}^2-3\lambda_{1}\lambda_{2},\;\alpha_{1}=\alpha_{2}=0;$\\
$(2)\lambda_{2}^2=3\lambda_{1}^2,\;p_{1}^2+p_{2}^2\neq0,\;\alpha=3\lambda_{1}^2-3\lambda_{1}\lambda_{2}
=\lambda_{2}^2-3\lambda_{1}\lambda_{2},\;\alpha_{1}=\alpha_{2}=0,\;\phi=c;$\\
$(3)\lambda_{2}^2\neq3\lambda_{1}^2,\;\lambda_{2}\neq3\lambda_{1},\;
\frac{6\lambda_{1}^2-3\lambda_{1}\lambda_{2}-\lambda_{2}^2}{3\lambda_{1}-\lambda_{2}}\neq0,\;p_{1}\neq0,\;p_{2}=0,\;
\alpha=\frac{18\lambda_{1}^4-6\lambda_{1}\lambda_{2}^3+24\lambda_{1}^2\lambda_{2}^2
-36\lambda_{1}^3\lambda_{2}}{(3\lambda_{1}-\lambda_{2})^2},\;\\ \mbox{} \quad \; \alpha_{1}=0,\alpha_{2}=\frac{18\lambda_{1}^4-2\lambda_{2}^4
+6\lambda_{1}\lambda_{2}^3-18\lambda_{1}^3\lambda_{2}}{(3\lambda_{1}-\lambda_{2})^2},\;
\phi=ce^{\frac{3\lambda_{1}^2-\lambda_{2}^2}{(3\lambda_{1}-\lambda_{2})p_{1}}t};$\\
$(4)\lambda_{2}^2<3\lambda_{1}^2,\;\lambda_{2}\neq3\lambda_{1},\;
\frac{6\lambda_{1}^2-3\lambda_{1}\lambda_{2}-\lambda_{2}^2}{3\lambda_{1}-\lambda_{2}}\neq0,\;p_{1}\neq0,
\;p_{2}\neq-\frac{1}{2}p_{1},\;\frac{\eta}{\zeta^2}=\frac{3\lambda_{1}^2-\lambda_{2}^2}{(3\lambda_{1}-\lambda_{2})^2},\;
\alpha=\\ \mbox{} \quad\, \alpha_{1}=\alpha_{2}=0,
\phi=ce^{\frac{3\lambda_{1}-\lambda_{2}}{\zeta}t};$\\
$(5)\lambda_{2}\neq3\lambda_{1},
\frac{6\lambda_{1}^2-3\lambda_{1}\lambda_{2}-\lambda_{2}^2}{3\lambda_{1}-\lambda_{2}}<0,
18\lambda_{1}^4-\lambda_{2}^4+6\lambda_{1}\lambda_{2}^3-15\lambda_{1}^2\lambda_{2}^2\neq0,
p_{1}\neq0,p_{2}=0,\alpha_{1}=\\ \mbox{} \quad\, 0,\alpha=\frac{18\lambda_{1}^4-6\lambda_{1}\lambda_{2}^3+24\lambda_{1}^2\lambda_{2}^2
-36\lambda_{1}^3\lambda_{2}}{(3\lambda_{1}-\lambda_{2})^2},\;\alpha_{2}=\frac{18\lambda_{1}^4-\lambda_{2}^4
+6\lambda_{1}\lambda_{2}^3-15\lambda_{1}^2\lambda_{2}^2}{(3\lambda_{1}-\lambda_{2})^2},\;
\phi=ce^{\frac{3\lambda_{1}\lambda_{2}-3\lambda_{1}^2}{(3\lambda_{1}-\lambda_{2})p_{1}}t};$\\
$(6)\lambda_{1}=\lambda_{2},\;p_{1}=0,\;p_{2}\neq0,\;\alpha=\alpha_{1}=\alpha_{2}=0,\;
\phi=c_{0}e^{\frac{\lambda_{1}}{p_{2}}t};$\\
$(7)\lambda_{1}^2=\frac{5\pm\sqrt{3}}{6}\lambda_{2}^2,\;p_{1}=(1\pm\sqrt{3})p_{2}\neq0,\;
\alpha=3\lambda_{1}^2-3\lambda_{1}\lambda_{2},\;\alpha_{1}=0,\;
\alpha_{2}=\frac{(p_{2}-p_{1})(\lambda_{2}^2-3\lambda_{1}\lambda_{2})}{p_{1}},\\
\mbox{} \quad \, \phi=ce^{\frac{\lambda_{2}}{2p_{1}}t};$\\
$(8)6\lambda_{1}^2-3\lambda_{1}\lambda_{2}-\lambda_{2}^2=0,\;
p_{1}\neq0,\;\alpha=6\lambda_{1}^2-3\lambda_{1}\lambda_{2}-\lambda_{2}^2=0,\;
\alpha_{1}=\alpha_{2}=0,\;\phi=ce^{\frac{\lambda_{2}}{2}\frac{\zeta}{\eta}t};$\\
$(9)6\lambda_{1}^2-3\lambda_{1}\lambda_{2}-\lambda_{2}^2=0,\;p_{1}\neq0,\;p_{2}=0,\;\frac{\zeta^2}{\eta}=1,\;
\alpha=\frac{\lambda_{2}^2}{4}+3\lambda_{1}^2-3\lambda_{1}\lambda_{2},\;\alpha_{2}=\\ \mbox{} \quad \;
3\lambda_{1}^2-\frac{5}{4}\lambda_{2}^2,\phi=ce^{\frac{\lambda_{2}}{2}\frac{\zeta}{\eta}t}.$
\end{Theorem}

\subsubsection{Type (II) generalized Kasner space-times with a quarter-symmetric connection with constant scalar curvature}
By Proposition $4.17,$ then $(F_{2},\nabla^{F_{2}})$ has constant scalar curvature $S^{F_{2}}$ and
$$\overline S=\frac{S^{F_{2}}}{\phi^{2p_{2}}}-2\zeta\frac{\phi^{\prime\prime}}{\phi}
-(\eta+\zeta^2-2\zeta)\frac{(\phi^\prime)^2}{\phi^2}+3(\lambda_{1}+\lambda_{2})\zeta\frac{\phi^\prime}{\phi}
+3(\lambda_{1}^2+\lambda_{2}^2-4\lambda_{1}\lambda_{2}). \eqno{(63)}$$
¢Ù\;$\zeta=0.\;(1)\;\eta=0.$ then $p_{1}=p_{2}=0$ and
$\overline S=S^{F_{2}}+3(\lambda_{1}^2+\lambda_{2}^2-4\lambda_{1}\lambda_{2});$\\ \mbox{} \qquad \quad \;\;
$(2)\;\eta\neq0.$ then $\overline S=\frac{S^{F_{2}}}{\phi^{2p_{2}}}-\eta\frac{(\phi^\prime)^2}{\phi^2}
+3(\lambda_{1}^2+\lambda_{2}^2-4\lambda_{1}\lambda_{2}),$ which means
$$\eta\frac{(\phi^\prime)^2}{\phi^2}=\frac{S^{F_{2}}}{\phi^{2p_{2}}}
-[\overline S-3(\lambda_{1}^2+\lambda_{2}^2-4\lambda_{1}\lambda_{2})]. \eqno{(64)}$$
¢Ú$\zeta\neq0.$ Putting $\phi=\psi^\frac{2\zeta}{\eta+\zeta^2},$ we get
$$-\frac{4\zeta^2}{\eta+\zeta^2}\psi^{\prime\prime}+\frac{6(\lambda_{1}+\lambda_{2})\zeta^2}{\eta+\zeta^2}\psi^\prime
+(3\lambda_{1}^2+3\lambda_{2}^2-12\lambda_{1}\lambda_{2}
-\overline S)\psi+S^{F_{2}}\psi^{1-\frac{4p_{2}\zeta}{\eta+\zeta^2}}=0.       \eqno{(65)}$$

\subsubsection{Classification of Einstein Type (III) generalized Kasner space-times with a quarter-symmetric connection}
Considering ${\rm dim}F_{1}={\rm dim}F_{2}={\rm dim}F_{3}=1,$ by Remark $4,$ we get $\alpha_{i}=0$ and by Proposition $4.16,$ we have
$$\zeta\Big(\lambda_{2}\frac{\phi^\prime}{\phi}-\frac{\phi^{\prime\prime}}{\phi}\Big)
-(\eta-\zeta)\frac{(\phi^\prime)^2}{\phi^2}+3(\lambda_{1}^2-\lambda_{1}\lambda_{2})=\alpha;  \eqno{(66a)} $$
$$-p_{1}\Big[\frac{\phi^{\prime\prime}}{\phi}+(\zeta-1)\frac{(\phi^\prime)^2}{\phi^2}
+(\lambda_{2}-3\lambda_{1})\frac{\phi^\prime}{\phi}\Big]+\lambda_{2}\zeta\frac{\phi^\prime}{\phi}
=\alpha-\lambda_{2}^2+3\lambda_{1}\lambda_{2}.   \eqno{(66b)}$$
$$-p_{2}\Big[\frac{\phi^{\prime\prime}}{\phi}+(\zeta-1)\frac{(\phi^\prime)^2}{\phi^2}
+(\lambda_{2}-3\lambda_{1})\frac{\phi^\prime}{\phi}\Big]+\lambda_{2}\zeta\frac{\phi^\prime}{\phi}
=\alpha-\lambda_{2}^2+3\lambda_{1}\lambda_{2}.   \eqno{(66c)}$$
$$-p_{3}\Big[\frac{\phi^{\prime\prime}}{\phi}+(\zeta-1)\frac{(\phi^\prime)^2}{\phi^2}
+(\lambda_{2}-3\lambda_{1})\frac{\phi^\prime}{\phi}\Big]+\lambda_{2}\zeta\frac{\phi^\prime}{\phi}
=\alpha-\lambda_{2}^2+3\lambda_{1}\lambda_{2}.   \eqno{(66d)}$$
¢Ù\;$\zeta=\eta=0.$ By the equation $(66a),$ we have $\alpha=3\lambda_{1}^2-3\lambda_{1}\lambda_{2},$ and by the
equation $(66b),$ we have $\alpha=\lambda_{2}^2-3\lambda_{1}\lambda_{2},$ then we get $\lambda_{2}^2=3\lambda_{1}^2.$ \\ So we obtain $\underline{\lambda_{2}^2=3\lambda_{1}^2,\;\alpha=3\lambda_{1}^2-3\lambda_{1}\lambda_{2}= \lambda_{2}^2-3\lambda_{1}\lambda_{2},\;\alpha_{i}=0,\;\zeta=\eta=0.}$\\
¢Ú\;$\zeta=0,\;\eta\neq0.$ $(66b)+(66c)+(66d),$ we get $\alpha=\lambda_{2}^2-3\lambda_{1}\lambda_{2};$
by the equation $(66a),$ we have $\frac{(\phi^\prime)^2}{\phi^2}=\frac{3\lambda_{1}^2-\lambda_{2}^2}{\eta}.$\\
$1)\;3\lambda_{1}^2-\lambda_{2}^2<0,$ we have no solution.\\
$2)\;3\lambda_{1}^2-\lambda_{2}^2=0,$ then $\phi=c,$ which satisfies the equation $(66a).$ \\
$3)\;3\lambda_{1}^2-\lambda_{2}^2>0,$ then $\phi=c_{0}e^{\pm\sqrt{\frac{3\lambda_{1}^2-\lambda_{2}^2}{\eta}}t},$
since $\eta\neq0,$ so at least one $p_{i}\neq0,$ we assume $p_{1}\neq0,$ by the equation $(66b),$ we get
$\lambda_{2}=3\lambda_{1},$ but by $3\lambda_{1}^2-\lambda_{2}^2>0,$ we get $\lambda_{1}^2<0,$ which is a contradiction.\\
So we have $\underline{\lambda_{2}^2=3\lambda_{1}^2,\;\alpha=3\lambda_{1}^2-3\lambda_{1}\lambda_{2}
=\lambda_{2}^2-3\lambda_{1}\lambda_{2},\;\alpha_{i}=0,\;\zeta=0,\;\eta\neq0,\;\phi=c}$ \;in case ¢Ú.\\
¢Û\;$\zeta\neq0,$ then $\eta\neq0.$ If $p_{1}=p_{2}=p_{3},$ we get type $(I),$ if $p_{1}=p_{2}$ or $p_{2}=p_{3}$ or $p_{1}=p_{3},$ we get type $(II),$ so $p_{1}\neq p_{2}\neq p_{3}.$ Let $\phi=\psi^{\frac{\zeta}{\eta}},$
then equations $(66a)-(66d)$ becomes
$$\frac{\zeta^2}{\eta}\frac{\lambda_{2}\psi^\prime-\psi^{\prime\prime}}{\psi}
=\alpha+3\lambda_{1}\lambda_{2}-3\lambda_{2}^2      \eqno{(67a)}$$
$$\frac{p_{1}}{\zeta}\Big[-\frac{(\phi^\zeta)^{\prime\prime}}{\phi^\zeta}
+(3\lambda_{1}-\lambda_{2})\frac{(\phi^\zeta)^\prime}{\phi^\zeta}\Big]+\lambda_{2}\frac{(\phi^\zeta)^\prime}{\phi^\zeta}
=\alpha-\lambda_{2}^2+3\lambda_{1}\lambda_{2}    \eqno{(67b)}$$
$$\frac{p_{2}}{\zeta}\Big[-\frac{(\phi^\zeta)^{\prime\prime}}{\phi^\zeta}
+(3\lambda_{1}-\lambda_{2})\frac{(\phi^\zeta)^\prime}{\phi^\zeta}\Big]+\lambda_{2}\frac{(\phi^\zeta)^\prime}{\phi^\zeta}
=\alpha-\lambda_{2}^2+3\lambda_{1}\lambda_{2}    \eqno{(67c)}$$
$$\frac{p_{3}}{\zeta}\Big[-\frac{(\phi^\zeta)^{\prime\prime}}{\phi^\zeta}
+(3\lambda_{1}-\lambda_{2})\frac{(\phi^\zeta)^\prime}{\phi^\zeta}\Big]+\lambda_{2}\frac{(\phi^\zeta)^\prime}{\phi^\zeta}
=\alpha-\lambda_{2}^2+3\lambda_{1}\lambda_{2}    \eqno{(67d)}$$
$(67b)\times p_{2}-(67c)\times p_{1}$ and considering that $p_{1}\neq p_{2},$ we get
$$\lambda_{2}\frac{(\phi^\zeta)^\prime}{\phi^\zeta}=\alpha-\lambda_{2}^2+3\lambda_{1}\lambda_{2}  \eqno{(68a)}$$
$$\frac{(\phi^\zeta)^\prime}{\phi^\zeta}=\frac{\alpha}{\lambda_{2}}-\lambda_{2}+3\lambda_{1}  \eqno{(68b)}$$
by equations $(67b)$ and $(68a),$ we get $-\frac{(\phi^\zeta)^{\prime\prime}}{\phi^\zeta}+(3\lambda_{1}-\lambda_{2})\frac{(\phi^\zeta)^\prime}{\phi^\zeta}=0,$
then by the equation $(68b),$ we have
$$\frac{(\phi^\zeta)^{\prime\prime}}{\phi^\zeta}=(\frac{3\lambda_{1}}{\lambda_{2}}-1)\alpha
+9\lambda_{1}^2-6\lambda_{1}\lambda_{2}+\lambda_{2}^2     \eqno{(69)}$$
on the other hand, using the equation $(68b),$ we get
$$\phi^\zeta=c_{0}e^{(\frac{\alpha}{\lambda_{2}}-\lambda_{2}+3\lambda_{1})t}   \eqno{(70)} $$
by equations $(69)$ and $(70),$ we obtain
$\alpha^2+(3\lambda_{1}\lambda_{2}-\lambda_{2}^2)\alpha=0,$
so when $\lambda_{2}=3\lambda_{1},$ we have $\alpha=0;$
when $\lambda_{2}\neq3\lambda_{1},$ we have $\alpha=0$ or $\alpha=\lambda_{2}^2-3\lambda_{1}\lambda_{2}\neq0.$\\
$1)\;\lambda_{2}=3\lambda_{1}.$ Then $\alpha=0,$ by the equation $(70),$ we get $\phi^\zeta$ is a constant, then $\psi$ is a constant, so by the equation $(67a),$ we have $\lambda_{1}^2=0,$ which is a contradiction with $\lambda_{1}\neq0.$\\
$2)\;\lambda_{2}\neq3\lambda_{1}.\;1^\prime.\;\alpha=\lambda_{2}^2-3\lambda_{1}\lambda_{2}\neq0.$
By the equation $(70),$ we get $\phi^\zeta$ is a constant, then $\psi$ is a constant, and $\phi=c$ is a constant.
by the equation $(67a),$ we have $\alpha=3\lambda_{1}^2-3\lambda_{1}\lambda_{2}$ and $\lambda_{2}^2=3\lambda_{1}^2.$\\
So we get $\underline{\lambda_{2}^2=3\lambda_{1}^2,\;\alpha=3\lambda_{1}^2-3\lambda_{1}\lambda_{2}
=\lambda_{2}^2-3\lambda_{1}\lambda_{2},\;\alpha_{i}=0,\;\zeta\neq0,\;\eta\neq0,\;\phi=c.}$ \\
$2^\prime.\;\alpha=0.$ By the equation $(67a),$ we have
$$\psi^{\prime\prime}-\lambda_{2}\psi^\prime+(3\lambda_{1}\lambda_{2}-3\lambda_{1}^2)\frac{\eta}{\zeta^2}\psi=0.  \eqno{(71)}$$
by the equation $(70),$ we get $\phi^\zeta=c_{0}e^{(3\lambda_{1}-\lambda_{2})t},$ then
$\phi=ce^{\frac{3\lambda_{1}-\lambda_{2}}{\zeta}t},\;\psi=c_{1}e^{(3\lambda_{1}-\lambda_{2})\frac{\eta}{\zeta^2}t},$
by the equation $(71),$ we get $\frac{\eta}{\zeta^2}=\frac{3\lambda_{1}^2-\lambda_{2}^2}{(3\lambda_{1}-\lambda_{2})^2}.$\\
So we have $\underline{\lambda_{2}\neq3\lambda_{1},\;\alpha=0,\;\alpha_{i}=0,\;
\frac{\eta}{\zeta^2}=\frac{3\lambda_{1}^2-\lambda_{2}^2}{(3\lambda_{1}-\lambda_{2})^2},\;
\phi=ce^{\frac{3\lambda_{1}-\lambda_{2}}{\zeta}t}.}$   \par
According to above discussions, we get the following Theorem:
\begin{Theorem}
Let $M=I \times_{\phi^{p_{1}}}F_{1} \times_{\phi^{p_{2}}}F_{2}\cdots \times_{\phi^{p_{m}}}F_{m}$ be a generalized Kasner space-time for $p_{i}\neq p_{j}$ for $i,j\in\{1,2,3\}$ and ${\rm dim}F_{1}={\rm dim}F_{2}={\rm dim}F_{3}=1,$
and $P=\frac{\partial}{\partial t}.$ Then $(M,\overline\nabla)$ is Einstein with the Einstein constant $\alpha$ if and only if one of the following conditions is satisfied:\\
$(1)\;\lambda_{2}^2=3\lambda_{1}^2,\;\alpha=3\lambda_{1}^2-3\lambda_{1}\lambda_{2}=
\lambda_{2}^2-3\lambda_{1}\lambda_{2},\;\alpha_{i}=0,\;\zeta=\eta=0;$ \\
$(2)\;\lambda_{2}^2=3\lambda_{1}^2,\;\alpha=3\lambda_{1}^2-3\lambda_{1}\lambda_{2}
=\lambda_{2}^2-3\lambda_{1}\lambda_{2},\;\alpha_{i}=0,\;\eta\neq0,\;\phi=c;$  \\
$(3)\;\lambda_{2}\neq3\lambda_{1},\;\alpha=0,\;\alpha_{i}=0,\;
\frac{\eta}{\zeta^2}=\frac{3\lambda_{1}^2-\lambda_{2}^2}{(3\lambda_{1}-\lambda_{2})^2},\;
\phi=ce^{\frac{3\lambda_{1}-\lambda_{2}}{\zeta}t}.$
\end{Theorem}

\subsubsection{Type (III) generalized Kasner space-times with a quarter-symmetric connection with constant scalar curvature}
By Proposition $4.17,$ we get
$$\overline S=-2\zeta\frac{\phi^{\prime\prime}}{\phi}-(\eta+\zeta^2-2\zeta)\frac{(\phi^\prime)^2}{\phi^2}
+3(\lambda_{1}+\lambda_{2})\zeta\frac{\phi^\prime}{\phi}+3(\lambda_{1}^2+\lambda_{2}^2-4\lambda_{1}\lambda_{2}). \eqno{(72)}$$
¢Ù\;$\zeta=\eta=0.$ Then $p_{1}=p_{2}=p_{3}=0,$ and
$\overline S=3(\lambda_{1}^2+\lambda_{2}^2-4\lambda_{1}\lambda_{2}).$\\
¢Ú\;$\zeta=0,\;\eta\neq0.$ Then $[(ln\phi)^\prime]^2=\frac{3(\lambda_{1}^2+\lambda_{2}^2-4\lambda_{1}\lambda_{2})-\overline S}{\eta},$ so we have:\\
$1)\;\overline S>3(\lambda_{1}^2+\lambda_{2}^2-4\lambda_{1}\lambda_{2}),$ we have no solution.\\
$2)\;\overline S=3(\lambda_{1}^2+\lambda_{2}^2-4\lambda_{1}\lambda_{2}),$ then $\phi=c.$  \\
$3)\;\overline S<3(\lambda_{1}^2+\lambda_{2}^2-4\lambda_{1}\lambda_{2}),$ then $\phi=c_{0}e^{\pm\sqrt{\frac{3(\lambda_{1}^2+\lambda_{2}^2-4\lambda_{1}\lambda_{2})-\overline S}{\eta}}t}.$\\
¢Û\;$\zeta\neq0,$ then $\eta\neq0.$ Putting $\phi=\psi^\frac{2\zeta}{\eta+\zeta^2},$ then
$$-\frac{4\zeta^2}{\eta+\zeta^2}\psi^{\prime\prime}+\frac{6(\lambda_{1}+\lambda_{2})\zeta^2}{\eta+\zeta^2}\psi^\prime
+(3\lambda_{1}^2+3\lambda_{2}^2-12\lambda_{1}\lambda_{2}-\overline S)\psi=0.       \eqno{(73)}$$
So we get\\
$1)\;\overline S<\frac{9\zeta^2(\lambda_{1}+\lambda_{2})^2}{4(\eta+\zeta^2)}
+3\lambda_{1}^2+3\lambda_{2}^2-12\lambda_{1}\lambda_{2},$ \\
\begin{eqnarray*}
\psi &=& c_{1}e^{\frac{\frac{3}{2}(\lambda_{1}+\lambda_{2})+\sqrt{\frac{9(\lambda_{1}+\lambda_{2})^2}{4}
-\frac{(\overline S-3\lambda_{1}^2-3\lambda_{2}^2+12\lambda_{1}\lambda_{2})(\eta+\zeta^2)}{\zeta^2}}}{2}t}\\
&+& c_{2}e^{\frac{\frac{3}{2}(\lambda_{1}+\lambda_{2})-\sqrt{\frac{9(\lambda_{1}+\lambda_{2})^2}{4}
-\frac{(\overline S-3\lambda_{1}^2-3\lambda_{2}^2+12\lambda_{1}\lambda_{2})(\eta+\zeta^2)}{\zeta^2}}}{2}t};
\end{eqnarray*}
$2)\;\overline S=\frac{9\zeta^2(\lambda_{1}+\lambda_{2})^2}{4(\eta+\zeta^2)}
+3\lambda_{1}^2+3\lambda_{2}^2-12\lambda_{1}\lambda_{2},\;
\psi=c_{1}e^{\frac{3(\lambda_{1}+\lambda_{2})}{4}t}+c_{2}te^{\frac{3(\lambda_{1}+\lambda_{2})}{4}t};$  \\
$3)\;\overline S>\frac{9\zeta^2(\lambda_{1}+\lambda_{2})^2}{4(\eta+\zeta^2)}
+3\lambda_{1}^2+3\lambda_{2}^2-12\lambda_{1}\lambda_{2},$\\
\begin{eqnarray*}
\psi &=& c_{1}e^{\frac{3}{4}(\lambda_{1}+\lambda_{2})t}cos{\frac{\sqrt{\frac{(\overline S-3\lambda_{1}^2-3\lambda_{2}^2+12\lambda_{1}\lambda_{2})(\eta+\zeta^2)}{\zeta^2}
-\frac{9(\lambda_{1}+\lambda_{2})^2}{4}}}{2}t}\\
&+& c_{2}e^{\frac{3}{4}(\lambda_{1}+\lambda_{2})t}sin{\frac{\sqrt{\frac{(\overline S-3\lambda_{1}^2-3\lambda_{2}^2+12\lambda_{1}\lambda_{2})(\eta+\zeta^2)}{\zeta^2}
-\frac{9(\lambda_{1}+\lambda_{2})^2}{4}}}{2}t}
\end{eqnarray*} \par
According to above discussions, we get the following Theorem:
\begin{Theorem}
Let $M=I \times_{\phi^{p_{1}}}F_{1} \times_{\phi^{p_{2}}}F_{2}\cdots \times_{\phi^{p_{m}}}F_{m}$ be a generalized Kasner space-time and ${\rm dim}F_{1}={\rm dim}F_{2}={\rm dim}F_{3}=1,$ and $P=\frac{\partial}{\partial t}.$ Then $(M,\overline\nabla)$ has constant scalar curvature $\overline S$ if and only if one of the following conditions is satisfied:\\
$(1)\;\zeta=\eta=0,\;\overline S=3(\lambda_{1}^2+\lambda_{2}^2-4\lambda_{1}\lambda_{2}).$ \\
$(2)\;\zeta=0,\;\eta\neq0,$
when $\overline S>3(\lambda_{1}^2+\lambda_{2}^2-4\lambda_{1}\lambda_{2}),$ we have no solution;
when $\overline S=3(\lambda_{1}^2+ \\ \mbox{} \quad \;\lambda_{2}^2-4\lambda_{1}\lambda_{2}),$ then $\phi=c;$
when $\overline S<3(\lambda_{1}^2+\lambda_{2}^2-4\lambda_{1}\lambda_{2}),$ then $\phi=c_{0}e^{\pm\sqrt{\frac{3(\lambda_{1}^2+\lambda_{2}^2-4\lambda_{1}\lambda_{2})-\overline S}{\eta}}t}.$  \\
$(3)\;\zeta\neq0,$\\
\mbox{} \quad $(a)\;\overline S<\frac{9\zeta^2(\lambda_{1}+\lambda_{2})^2}{4(\eta+\zeta^2)}
+3\lambda_{1}^2+3\lambda_{2}^2-12\lambda_{1}\lambda_{2},$
\begin{eqnarray*}
\psi &=& c_{1}e^{\frac{\frac{3}{2}(\lambda_{1}+\lambda_{2})+\sqrt{\frac{9(\lambda_{1}+\lambda_{2})^2}{4}
-\frac{(\overline S-3\lambda_{1}^2-3\lambda_{2}^2+12\lambda_{1}\lambda_{2})(\eta+\zeta^2)}{\zeta^2}}}{2}t}\\
&+& c_{2}e^{\frac{\frac{3}{2}(\lambda_{1}+\lambda_{2})-\sqrt{\frac{9(\lambda_{1}+\lambda_{2})^2}{4}
-\frac{(\overline S-3\lambda_{1}^2-3\lambda_{2}^2+12\lambda_{1}\lambda_{2})(\eta+\zeta^2)}{\zeta^2}}}{2}t};
\end{eqnarray*}
\mbox{} \quad $(b)\;\overline S=\frac{9\zeta^2(\lambda_{1}+\lambda_{2})^2}{4(\eta+\zeta^2)}
+3\lambda_{1}^2+3\lambda_{2}^2-12\lambda_{1}\lambda_{2},\;
\psi=c_{1}e^{\frac{3(\lambda_{1}+\lambda_{2})}{4}t}+c_{2}te^{\frac{3(\lambda_{1}+\lambda_{2})}{4}t};$  \\
\mbox{} \quad $(c)\;\overline S>\frac{9\zeta^2(\lambda_{1}+\lambda_{2})^2}{4(\eta+\zeta^2)}
+3\lambda_{1}^2+3\lambda_{2}^2-12\lambda_{1}\lambda_{2},$\\
\begin{eqnarray*}
\psi &=& c_{1}e^{\frac{3}{4}(\lambda_{1}+\lambda_{2})t}cos{\frac{\sqrt{\frac{(\overline S-3\lambda_{1}^2-3\lambda_{2}^2+12\lambda_{1}\lambda_{2})(\eta+\zeta^2)}{\zeta^2}
-\frac{9(\lambda_{1}+\lambda_{2})^2}{4}}}{2}t}\\
&+& c_{2}e^{\frac{3}{4}(\lambda_{1}+\lambda_{2})t}sin{\frac{\sqrt{\frac{(\overline S-3\lambda_{1}^2-3\lambda_{2}^2+12\lambda_{1}\lambda_{2})(\eta+\zeta^2)}{\zeta^2}
-\frac{9(\lambda_{1}+\lambda_{2})^2}{4}}}{2}t}
\end{eqnarray*}
\end{Theorem}
\begin{remark}
When $\lambda_{1}=\lambda_{2}=1,$ we get Propositions $32,33$ Theorems $35,37,36,$ in $[12]$ by Propositions $4.16,4.17,$ Theorems $4.19-4.21,$ respectively.
\end{remark}
{\bfseries Acknowledgement.}
This work was supported by NSFC No.11271062 and NCET-13-0721. We would like to thank the referee for his(her) careful reading and helpful comments.


\clearpage

\end{document}